\pdfoutput=1

\documentclass[11pt,letterpaper]{article}
\usepackage{fullpage}

\usepackage[cp1250]{inputenc}
\usepackage[T1]{fontenc}
\usepackage{calligra}
\usepackage{amsfonts}
\usepackage{amsmath}
\usepackage{amsthm}
\usepackage{amssymb}

\usepackage[dvips]{graphicx}
\usepackage{url}
\usepackage{color}
\usepackage{graphicx}
\usepackage{algorithmic}
\usepackage{algorithm}
\usepackage{verbatim}
\usepackage{subfigure}

\usepackage{hyperref} 

\newcommand{\del}[1]{}

\newcommand{\ve}[2]{\langle #1 ,  #2 \rangle}

\newcommand{\eqdef}{\stackrel{\text{def}}{=}}

\newcommand{\cPsi}{\Omega}

\newcommand{\Srv}{\hat{S}}   

\newcommand{\JJ}[0]{\mathcal J}

\newcommand{\vsubset}[2]{#1_{[#2]}}

\newcommand{\R}{\mathbf{R}}
\newcommand{\Prob}{\mathbf{P}}
\newcommand{\E}{\mathbf{E}}

\newcommand{\vc}[2]{#1^{(#2)}}

\newcommand{\ncs}[2]{\|#1\|^2_{(#2)}}

\newcommand{\Rw}[2]{\mathcal R_{#1}(#2)}
\newcommand{\Rws}[2]{\mathcal R^2_{#1}(#2)}

\newcommand{\nbp}[2]{\|#1\|_{(#2)}}   
\newcommand{\nbd}[2]{\|#1\|_{(#2)}^*} 

\newcommand{\U}{U}
\newcommand{\N}{N}

\DeclareMathOperator{\support}{Supp}       

\newcommand{\Lip}{L}

\DeclareMathOperator{\Exp}{\mathbf{E}}           

\DeclareMathOperator{\dom}{dom}         

\DeclareMathOperator{\Diag}{Diag}       
\DeclareMathOperator{\diag}{diag}       

\theoremstyle{plain}
\newtheorem{theorem}{Theorem}
\newtheorem{proposition}[theorem]{Proposition}

\newtheorem{lemma}[theorem]{Lemma}

\theoremstyle{definition}

\newtheorem{definition}[theorem]{Definition}

\newcommand{\NN}{{[n]}}
\newcommand{\pp}{\Prob}
\newcommand{\ABC}{\phantom{$\frac{\frac{A}{B}}{\frac{A}{B}}$}}

\title{Parallel Coordinate Descent Methods for Big Data Optimization \footnote{
This paper was awarded the \textbf{16th IMA Leslie Fox Prize in Numerical Analysis} (2nd Prize; for M.T.) in June 2013. The work of the first author was supported by EPSRC grants EP/J020567/1 (Algorithms for Data Simplicity) and EP/I017127/1 (Mathematics for Vast Digital Resources). The second author was supported by the Centre for Numerical Algorithms and Intelligent Software (funded by EPSRC grant EP/G036136/1 and the Scottish Funding Council). An open source code with an efficient implementation of the algorithm(s) developed in this paper is published here: \href{http://code.google.com/p/ac-dc/}{http://code.google.com/p/ac-dc/.} }}

\author{Peter Richt\'arik \qquad \qquad   Martin Tak\'a\v{c}\\\\\emph{School of Mathematics}\\\emph{University of Edinburgh}\\\emph{United Kingdom}}

\date{November 23, 2012\\ (revised November 23, 2013)}

\begin{document}

\maketitle

\begin{abstract}

In this work we show that randomized (block) coordinate descent methods can be accelerated by parallelization when applied to the problem of minimizing the sum of a \emph{partially separable} smooth convex function and a simple separable convex function. The theoretical speedup, as compared to the serial method, and referring to the number of iterations needed to approximately solve the problem with high probability, is a simple expression depending on the number of parallel processors and a natural and easily computable measure of separability of the smooth component of the objective function. In the worst case, when no degree of separability is present, there may be no speedup; in the best case, when the problem is separable, the speedup is equal to the number of processors.  Our analysis also works in the mode when the number of blocks being updated at each iteration is random, which allows for modeling situations with busy or unreliable processors. We show that our algorithm is able to solve a LASSO problem involving a matrix with 20 billion nonzeros in 2 hours on a large memory node with 24 cores.

\paragraph{Keywords:}  Parallel coordinate descent, big data optimization, partial separability, huge-scale optimization, iteration complexity, expected separable over-approximation, composite objective, convex optimization, LASSO.
\end{abstract}

\section{Introduction} \label{sec:intro}

\paragraph{Big data optimization.}
Recently there has been a surge in  interest in the design of algorithms suitable for  solving convex optimization problems with a huge number of variables \cite{RT:UCDC, Nesterov-Subgrad-Huge}. Indeed, the size of problems arising in fields such as machine learning \cite{Bradley:PCD-paper}, network analysis \cite{zargham2011accelerated}, PDEs \cite{Xu:space_decomposition}, truss topology design \cite{RT:TTD2011} and compressed sensing \cite{Li:CDOMACSGA} usually grows with our capacity to solve them, and is projected to grow dramatically in the next decade. In fact, much of computational science is currently facing the ``big data'' challenge, and this work is aimed at developing optimization algorithms suitable for the task.

\paragraph{Coordinate descent methods.}
Coordinate descent  methods  (CDM) are one of the most successful classes of algorithms in the big data optimization domain. Broadly speaking, CDMs are based on the strategy of updating a single coordinate (or a single block of coordinates) of the vector of variables at each iteration. This  often drastically reduces memory requirements as well as the arithmetic complexity of a single iteration, making the methods easily implementable and scalable. In certain applications, a single iteration can amount to as few as 4 multiplications and additions  only \cite{RT:TTD2011}! On the other hand, many more iterations are necessary for convergence than it is usual for classical gradient methods. Indeed, the number of iterations a CDM requires to solve a smooth convex optimization problem is $O(\tfrac{n \tilde{L} R^2}{\epsilon})$, where $\epsilon$ is the error tolerance, $n$ is the number variables (or blocks of variables), $\tilde{L}$ is the average of the Lipschitz constants of the gradient of the objective function associated with the variables (blocks of variables) and $R$ is the  distance from the starting iterate to the set of optimal solutions. On balance, as observed by numerous authors, serial CDMs are  much more efficient for big data optimization problems than most other competing approaches, such as gradient methods \cite{Nesterov:2010RCDM, RT:TTD2011}.

\paragraph{Parallelization.}
We wish to point out that for truly huge-scale problems it is absolutely necessary to \emph{parallelize}. This is  in line with the rise and ever  increasing availability of high performance computing systems built around multi-core processors, GPU-accelerators and computer clusters, the success of which is rooted in massive parallelization. This simple observation, combined with the remarkable scalability of serial CDMs, leads to our belief that the study of parallel coordinate descent methods (PCDMs) is a very timely topic.

\paragraph{Research Idea.}
The work presented in this paper was motivated by the desire to answer the following question:

\begin{quote}
\emph{Under what natural and easily verifiable structural assumptions on the objective function does parallelization of a coordinate descent method lead to acceleration?}
\end{quote}

Our starting point was the following simple observation. Assume that we wish to minimize a separable function $F$ of $n$ variables (i.e., a function that can be written as a sum of $n$ functions each of which depends on a single variable only). For simplicity, in this thought experiment, assume that there are no constraints. Clearly, the problem of minimizing $F$ can be trivially decomposed into $n$ independent univariate problems. Now, if we have $n$ processors/threads/cores, each assigned with the task of solving one of these problems, the number of parallel iterations should not depend on the dimension of the problem\footnote{For simplicity, assume the distance from the starting point to the set of optimal solutions does not depend on the dimension.}. In other words, we get an $n$-times speedup compared to the situation with a single processor only. Note that any parallel algorithm of this type can be viewed as a parallel coordinate descent method. Hence, a PCDM with $n$ processors should be $n$-times  faster than a serial one. If  $\tau$ processors are used instead, where $1\leq \tau \leq n$, one would expect a $\tau$-times speedup.

By extension, one would perhaps expect that optimization problems with objective functions which are ``close to being separable''  would also be amenable to acceleration by parallelization, where the acceleration factor $\tau$ would be reduced with the reduction of the ``degree of separability''. One of the main messages of this paper is an affirmative answer to this. Moreover, we give explicit and simple formulae for the speedup factors.

As it turns out, and as we discuss later in this section, many real-world big data optimization problems are, quite naturally, ``close to being separable''. We believe that this means that PCDMs is a very promising class of algorithms when it comes to solving structured big data optimization problems.


\paragraph{Minimizing a partially separable composite objective.} In this paper we study the problem
\begin{equation}\label{eq:P}
\text{minimize} \quad \left\{ F(x) \eqdef f(x) + \cPsi(x)\right\} \quad \text{subject to} \quad x\in \R^\N,
\end{equation}
where $f$ is a (block) \emph{partially separable} smooth convex function and $\cPsi$ is a simple (block) separable convex function.
 We allow $\cPsi$ to have values in $\R\cup\{\infty\}$, and for regularization purposes we assume $\cPsi$ is proper and closed. While \eqref{eq:P} is seemingly an unconstrained problem, $\cPsi$ can be chosen to model simple convex constraints on individual blocks of variables. Alternatively, this function can be used to enforce a certain structure (e.g., sparsity) in the solution. For a more detailed account we refer the reader to \cite{RT:UCDC}. Further, we assume that this problem has a minimum ($F^*>-\infty$). What we mean by ``smoothness'' and ``simplicity'' will be made precise in the next section.

Let us now describe the key concept of partial separability. Let $x\in \R^\N$ be decomposed into $n$ non-overlapping blocks of variables $x^{(1)},\dots,x^{(n)}$ (this will be made precise in Section~\ref{SEC:PCDM}). We assume throughout the paper that $f:\R^\N \to \R$ is  \emph{partially separable of degree $\omega$}, i.e., that it can be written in the form
\begin{equation}\label{eq:strucutre_of_f}
f(x) = \sum_{J\in \JJ} f_J(x),
\end{equation}
where $\JJ$ is a finite collection of nonempty subsets of $\NN\eqdef \{1,2,\dots,n\}$ (possibly containing identical sets multiple times),
$f_J$ are differentiable convex functions such that $f_J$ depends on blocks $x^{(i)}$ for $i\in J$ only, and
\begin{equation}\label{eq:omega} |J| \leq \omega \quad \text{for all} \quad J \in \JJ.\end{equation}
Clearly, $1\leq \omega \leq n$. The PCDM algorithms we develop and analyze in this paper only need to know $\omega$, they do not need to know the decomposition of $f$ giving rise to this $\omega$.

\paragraph{Examples of partially separable functions.} Many objective functions naturally encountered in the big data setting are partially separable. Here we give examples of three  loss/objective functions frequently used in the machine learning literature and also elsewhere. For simplicity, we assume all blocks are of size 1 (i.e., $N=n$). Let
\begin{equation}\label{eq:ML_function}f(x) =  \sum_{j=1}^m {\cal L}(x,A_j,y_j),\end{equation}
where $m$ is the number of examples, $x\in \R^n$ is the vector of features, $(A_j,y_j) \in \R^n\times \R$ are labeled examples and ${\cal L}$ is one of  the three loss functions listed in Table~\ref{tbl:ML3}. Let $A\in \R^{m \times n}$ with row $j$ equal to $A_j^T$.
\begin{table}[!ht]
\begin{center}
\begin{tabular}{|l l|}
  \hline
   & \\
  Square Loss & $\tfrac{1}{2}(A_j^T x - y_j)^2$ \\
   & \\
  Logistic Loss & $\log(1 + e^{-y_j A_j^T x})$ \\
   & \\
  Hinge Square Loss & $\tfrac{1}{2}\max\{0, 1- y_j A_j^Tx\}^2$ \\
   & \\
  \hline
\end{tabular}
\end{center}
\caption{Three examples of loss of functions}
\label{tbl:ML3}
\end{table}
Often, each example depends  on a few features only; the maximum over all features is the degree of partial separability $\omega$. More formally, note that the $j$-th function in the sum \eqref{eq:ML_function} in all cases depends on $\|A_j\|_0$ coordinates of $x$ (the number of nonzeros in the $j$-th row of $A$) and hence $f$ is partially separable of degree \[\omega = \max_j \|A_j\|_0.\] All three functions of Table~\ref{tbl:ML3} are smooth (based on the definition of smoothness in the next section). We refer the reader to \cite{hogwild} for more examples of interesting (but nonsmooth) partially separable functions arising in graph cuts and matrix completion.

\paragraph{Brief literature review.}
Several papers were written recently studying the iteration complexity of \emph{serial} CDMs of various flavours and in various settings. We will only provide  a brief summary here, for a  more detailed account we refer the reader to \cite{RT:UCDC}.

Classical CDMs update the coordinates in a cyclic order; the first attempt at analyzing the complexity of such a method is due to \cite{Saha10finite}. Stochastic/randomized CDMs, that is, methods where the coordinate to be updated is chosen randomly, were first analyzed for quadratic objectives \cite{SV:Kaczmarz2009, Leventhal:2008:RMLC}, later independently generalized to $L_1$-regularized problems \cite{ShalevTewari09} and smooth block-structured problems \cite{Nesterov:2010RCDM}, and finally unified and refined in \cite{RT:SPARS11, RT:UCDC}. The problems considered in the above papers are either unconstrained or have (block) separable constraints. Recently, randomized CDMs were developed for problems with linearly coupled constraints \cite{Necoara:Coupled, Necoara:rcdm-coupled}.

A greedy CDM for $L_1$-regularized problems was first analyzed in \cite{RT:TTD2011}; more work on this topic include \cite{Li:CDOMACSGA,DRT:2009}. A CDM with inexact updates was first proposed and analyzed in \cite{TRG:InexactCDM}.  Partially separable problems were independently studied in \cite{hogwild}, where an asynchronous parallel stochastic gradient algorithm was developed to solve them. 

When writing this paper, the authors were aware only of the parallel CDM  proposed and analyzed in \cite{Bradley:PCD-paper}. Several papers on the topic  appeared around the time this paper was finalized or after \cite{Necoara:Parallel, YHSD:2012, STHH:2012, STHH:2012, PYY:2013}. Further papers on various aspects of the topic of parallel CDMs, building on the work in this paper, include \cite{minibatch-ICML2013, RT:Hydra2013, FR:SPCDM2013, RT:NSync}.  


%
%

\paragraph{Contents.} We start in Section~\ref{SEC:PCDM} by describing the block structure of the problem, establishing notation and detailing assumptions. Subsequently we propose and comment in detail on two parallel coordinate descent methods. In Section~\ref{SEC:Contribute} we summarize the main contributions of this paper. In Section~\ref{SEC:Block_Samplings} we deal with issues related to the selection of the blocks to be updated in each iteration. It will involve the development of some elementary random set theory. Sections~\ref{SEC:SO}-\ref{SEC:ESO_for_PS_functions} deal with issues related to the computation of the update to the selected blocks and develop a theory of Expected Separable Overapproximation (ESO), which is a novel tool we propose for the analysis of our algorithms. In Section~\ref{SEC:Iteration_Complexity} we analyze the iteration complexity of our methods and finally,  Section~\ref{SEC:Numerical_Experiments} reports on
promising computational results. For instance, we conduct an experiment with a big data (cca 350GB) LASSO problem with a billion variables. We are able to solve the problem using one of our methods on a large memory machine with 24 cores in 2 hours, pushing the difference between the objective value at the starting iterate and the optimal point from $10^{22}$ down to $10^{-14}$. We also conduct experiments on real data problems coming from machine learning.


\section{Parallel Block Coordinate Descent Methods} \label{SEC:PCDM}

In Section~\ref{sec:block_structure} we formalize the block structure of the problem, establish notation\footnote{Table~\ref{tbl:notation} in the appendix summarizes some of the key notation used frequently in the paper. } that will be used in the rest of the paper and list assumptions. In Section~\ref{sub:Algorithms} we propose two parallel block coordinate descent methods and comment in some detail on the steps.

\subsection{Block structure, notation and assumptions}\label{sec:block_structure}

Some elements of the setup described in this section was initially used in the analysis of block coordinate descent methods  by  Nesterov \cite{Nesterov:2010RCDM} (e.g., block structure, weighted norms and block Lipschitz constants).

The block structure of \eqref{eq:P} is given by a decomposition of $\R^\N$ into $n$ subspaces as follows. Let $\U\in \R^{\N\times \N}$ be a column permutation\footnote{The reason why we work with a \emph{permutation} of the identity matrix, rather than with the identity itself, as in  \cite{Nesterov:2010RCDM}, is to enable the blocks being formed by \emph{nonconsecutive} coordinates of $x$. This way we establish notation which makes it possible to work with (i.e., analyze the properties of) multiple block decompositions, for the sake of picking the best one, subject to some criteria. Moreover, in some applications the coordinates of $x$  have a natural ordering to  which the natural or efficient block structure does not correspond.} of the $\N\times \N$ identity matrix and further let $\U = [\U_1,\U_2,\dots,\U_n]$ be a decomposition of $\U$ into $n$ submatrices, with $\U_i$ being of size $\N \times \N_i$,
where $\sum_i \N_i = \N$.

\begin{proposition}[Block decomposition\footnote{This is a straightforeard result; we do not claim any novelty  and include it solely for the benefit of the reader.}] \label{prop:decomposition}Any vector $x\in \R^\N$ can be written uniquely as \begin{equation}\label{eq:block_decomposition}x = \sum_{i=1}^n \U_i x^{(i)},\end{equation} where $x^{(i)} \in \R^{N_i}$.
Moreover, $x^{(i)}=\U_i^T x$.
\end{proposition}
\begin{proof} Noting that $UU^T=\sum_i U_i U_i^T$ is the $\N\times \N$ identity matrix, we have $x=\sum_i U_i U_i^T x$. Let us now show uniqueness.
Assume that $x =\sum_i U_i x_1^{(i)} = \sum_i U_i x_2^{(i)}$, where $x_1^{(i)},x_2^{(i)}\in \R^{N_i}$.
Since
\begin{equation}\label{eq:U_iU_j} \U_j^T \U_i =
\begin{cases} \N_j\times \N_j \quad \text{identity matrix,} & \text{ if } i=j,\\
 \N_j\times \N_i \quad \text{zero matrix,}&  \text{ otherwise,}
 \end{cases}
\end{equation}
for every $j$ we get $0 = U_j^T (x-x) = U_j^T \sum_i U_i (x_1^{(i)}-x_2^{(i)}) = x_1^{(j)}-x_2^{(j)}$.
\end{proof}

In view of the above proposition, from now on we  write $x^{(i)}\eqdef U_i^T x \in \R^{N_i}$, and  refer to $x^{(i)}$ as the \emph{$i$-th block} of $x$.  The definition of partial separability in the introduction is with respect to these blocks. For simplicity,  we will sometimes write $x = (x^{(1)},\dots,x^{(n)})$.

\paragraph{Projection onto a set of blocks.} For $S\subset \NN$ and $x\in \R^\N$ we write
\begin{equation}\label{eq:lllop09}x_{[S]} \eqdef \sum_{i\in S} U_i x^{(i)}.\end{equation}
That is, given $x\in \R^\N$,  $x_{[S]}$ is the vector in $\R^\N$ whose blocks $i\in S$ are identical to those of $x$, but whose other blocks are zeroed out. In view of Proposition~\ref{prop:decomposition}, we can equivalently define $x_{[S]}$ block-by-block as follows
\begin{equation}\label{eq:lllop09jhkjh}(x_{[S]})^{(i)} = \begin{cases}x^{(i)}, \qquad & i\in S,\\ 0 \;(\in\R^{N_i}),  \qquad &\text{otherwise.}\end{cases}\end{equation}

\paragraph{Inner products.} The standard Euclidean inner product in spaces $\R^\N$ and $\R^{N_i}$, $i\in \NN$, will be denoted by $\ve{\cdot}{\cdot}$. Letting $x,y \in \R^\N$, the relationship between these inner products is given by
\[\ve{x}{y} \overset{\eqref{eq:block_decomposition}}{=} \ve{\sum_{j=1}^n \U_j x^{(j)}}{\sum_{i=1}^n \U_i y^{(i)}} = \sum_{j=1}^n \sum_{i=1}^n \ve{U_i^T U_j x^{(j)}}{y^{(i)}} \overset{\eqref{eq:U_iU_j}}{=}  \sum_{i=1}^n \ve{x^{(i)}}{y^{(i)}}.\]
For any $w\in \R^n$ and $x,y\in \R^\N$ we further define \begin{equation}\label{eq:weighted_inner_product}\ve{x}{y}_w \eqdef \sum_{i=1}^n w_i \ve{x^{(i)}}{y^{(i)}}.\end{equation}
For vectors  $z=(z_1,\dots,z_n)^T \in \R^n$ and $w = (w_1,\dots,w_n)^T \in \R^n$ we write $w\odot z \eqdef (w_1 z_1, \dots, w_n z_n)^T$.

\paragraph{Norms.} Spaces $\R^{N_i}$, $i \in \NN$, are equipped with a pair of conjugate norms: $\|t\|_{(i)} \eqdef \ve{B_i t}{t}^{1/2}$, where $B_i$ is an $N_i\times N_i$ positive definite matrix and $\|t\|_{(i)}^* \eqdef \max_{\|s\|_{(i)}\leq 1} \ve{s}{t} = \ve{B_i^{-1}t}{t}^{1/2}$, $t\in \R^{N_i}$. For $w\in \R^n_{++}$, define a pair of conjugate norms in $\R^\N$ by
\begin{equation}\label{eq:norms} \|x\|_w = \left[\sum_{i=1}^n w_i \ncs{\vc{x}{i}}{i}\right]^{1/2}, \quad
\|y\|_w^* \eqdef \max_{\|x\|_w\leq 1} \ve{y}{x} = \left[\sum_{i=1}^n w_i^{-1} ( \nbd{y^{(i)}}{i})^2\right]^{1/2}.\end{equation}
Note that these norms are induced by the inner product \eqref {eq:weighted_inner_product} and the matrices $B_1,\dots,B_n$. Often we will use $w=L\eqdef (L_1,L_2,\dots,L_n)^T\in \R^n$, where the constants $L_i$ are defined below.


\paragraph{Smoothness of $f$.} We assume throughout the paper that the gradient of $f$ is block  Lipschitz, uniformly in $x$, with positive constants $\Lip_1,\dots,\Lip_n$, i.e., that for all $x\in \R^\N$, $i\in \NN$ and $t\in \R^{N_i}$,
\begin{equation}\label{eq:f_iLipschitzder}
 \nbd{\nabla_i f(x+\U_i t)-\nabla_i f(x)}{i} \leq \Lip_i \nbp{t}{i},
  \end{equation}
where $\nabla_i f(x)  \eqdef (\nabla f(x))^{(i)} = \U^T_i \nabla f(x) \in \R^{N_i}$. An important consequence of \eqref{eq:f_iLipschitzder} is the following standard inequality \cite{NesterovBook}:
\begin{equation}\label{eq:Lipschitz_ineq}f(x+\U_i t) \leq f(x) + \ve{\nabla_i f(x)}{t} + \tfrac{\Lip_i}{2}\nbp{t}{i}^2.\end{equation}

\paragraph{Separability of $\cPsi$.} We assume that\footnote{For examples of separable and block separable functions we refer the reader to \cite{RT:UCDC}. For instance, $\Omega(x)=\|x\|_1$  is separable and block separable (used in sparse optimization); and $\Omega(x)=\sum_i \|x^{(i)}\|$, where the norms are standard Euclidean norms, is block separable (used in group lasso). One can model block constraints by setting $\Omega_i(x^{(i)}) = 0$ for $x \in X_i$, where $X_i$ is some closed convex set, and $\Omega_i(x^{(i)})=+\infty$ for $x \notin X_i$.} $\cPsi: \R^\N\to \R\cup \{+\infty\}$ is (block) separable, i.e., that it can be decomposed as follows:
\begin{equation}\label{eq:Psi_block_def}
  \cPsi(x)=\sum_{i=1}^n \cPsi_i(x^{(i)}),
\end{equation}
where the functions $\cPsi_i:\R^{N_i}\to \R \cup \{+\infty\}$ are convex and closed.

\paragraph{Strong convexity.} In one of our two complexity results (Theorem~\ref{thm:complexity-strongly-convex-case}) we will assume that either $f$ or $\cPsi$ (or both) is strongly convex. A function $\phi:\R^N\to \R\cup \{+\infty\}$ is strongly convex with respect to the norm $\|\cdot\|_w$ with convexity parameter $\mu_{\phi}(w) \geq 0$ if for all $x,y \in \dom \phi$,
\begin{equation}\label{eq:strong_def}\phi(y)\geq \phi(x) + \ve{\phi'(x)}{y-x} + \tfrac{\mu_{\phi}(w)}{2}\|y-x\|_w^2, \end{equation}
where $\phi'(x)$ is any subgradient of $\phi$ at $x$. The case with $\mu_\phi(w)=0$ reduces to convexity. Strong convexity of $F$ may come from $f$ or $\cPsi$ (or both); we write $\mu_f(w)$ (resp.\ $\mu_\cPsi(w)$) for the (strong) convexity parameter of $f$ (resp.\ $\cPsi$). It follows from \eqref{eq:strong_def} that \begin{equation}\label{eq:mu_F}\mu_{F}(w) \geq \mu_{f}(w)+ \mu_{\cPsi}(w).\end{equation}

The following  characterization of strong convexity will  be useful. For all $x,y \in \dom \phi$ and $\lambda \in [0,1]$,
\begin{equation}\label{eq:strong_2}
\phi(\lambda x+ (1-\lambda) y) \leq \lambda \phi(x) + (1-\lambda)\phi(y) - \tfrac{\mu_\phi(w)\lambda(1-\lambda)}{2}\|x-y\|_w^2.
\end{equation}
It can be shown using \eqref{eq:Lipschitz_ineq} and \eqref{eq:strong_def} that $\mu_f(w)\leq \tfrac{\Lip_i}{w_i}$.


\subsection{Algorithms} \label{sub:Algorithms}

In this paper we develop and study two generic parallel coordinate descent methods. The main method is PCDM1; PCDM2 is its ``regularized'' version which explicitly enforces monotonicity. As we will see, both of these methods come in many variations, depending on how Step~ 3 is performed.

\begin{algorithm}[!ht]
\caption{Parallel Coordinate Descent Method 1 (PCDM1)}
\begin{algorithmic}[1] \label{alg:PCD1}
 \STATE Choose initial point $x_0\in\R^\N$
\FOR {$k=0,1,2,\dots$}
 \STATE Randomly generate a set of blocks $S_k \subseteq \{1,2,\dots,n\}$
 \STATE $x_{k+1} \leftarrow x_k + (h(x_k))_{[S_k]}$
\ENDFOR
\end{algorithmic}
\end{algorithm}

\begin{algorithm}[!ht]
\caption{Parallel Coordinate Descent Method 2 (PCDM2)}
\begin{algorithmic}[1] \label{alg:PCD2}
 \STATE Choose initial point $x_0\in\R^\N$
\FOR {$k=0,1,2,\dots$}
 \STATE Randomly generate a set of blocks $S_k \subseteq \{1,2,\dots,n\}$
 \STATE $x_{k+1} \leftarrow x_k + (h(x_k))_{[S_k]}$
 \STATE If $F(x_{k+1}) > F(x_k)$, then $x_{k+1} \leftarrow x_k$
\ENDFOR
\end{algorithmic}
\end{algorithm}


Let us comment on the individual steps of the two methods.

\textbf{Step 3.} At the beginning of iteration $k$ we pick a random set ($S_k$) of blocks to be updated (in parallel) during that iteration. The set $S_k$ is a realization of a random set-valued mapping $\Srv$ with values in $2^\NN$ or, more precisely, it the sets $S_k$ are iid random sets with the distribution of $\hat{S}$. For brevity, in this paper we refer to such a mapping by the name \emph{sampling}. We limit our attention to \emph{uniform} samplings, i.e., random sets having the following property: $\Prob(i \in \hat{S})$ is independent of $i$. That is, the probability that a block gets selected is the same for all blocks. Although we give an iteration complexity result covering all such samplings (provided that each block has a chance to be updated, i.e., $\Prob(i \in \hat{S}) > 0$), there are interesting subclasses of uniform samplings (such as doubly uniform and nonoverlapping uniform samplings; see Section~\ref{SEC:Block_Samplings}) for which we give better results.


\textbf{Step 4.} For $x\in \R^\N$ we define\footnote{A similar map was used in \cite{Nesterov:2010RCDM} (with $\Omega\equiv 0$ and $\beta=1$) and \cite{RT:UCDC} (with $\beta=1$) in the analysis of \emph{serial} coordinate descent methods in the smooth and composite case, respectively. In loose terms, the novelty here is the introduction of the 
parameter $\beta$ and in developing theory which describes what value $\beta$ should have. Maps of this type are known as composite  gradient mapping in the literature, and were introduced in \cite{Nesterov:2007composite}}.
\begin{equation}\label{eq:h(x)} h(x) \eqdef \arg \min_{h \in \R^\N} H_{\beta,w}(x,h),\end{equation}
where
\begin{equation}\label{eq:H_{beta,w}}H_{\beta,w}(x,h) \eqdef f(x) + \ve{\nabla f(x)}{h} + \tfrac{\beta}{2}\|h\|_w^2 + \cPsi(x+h),\end{equation}
and $\beta>0$, $w=(w_1,\dots,w_n)^T \in \R^n_{++}$ are parameters of the method that we will comment on later. Note that in view of \eqref{eq:block_decomposition}, \eqref{eq:norms} and \eqref{eq:Psi_block_def}, $H_{\beta,w}(x,\cdot)$ is block separable;
\[H_{\beta,w}(x,h) = f(x) + \sum_{i=1}^n \left\{ \ve{\nabla_i f(x)}{h^{(i)}} + \tfrac{\beta w_i}{2}\|h^{(i)}\|_{(i)}^2 + \cPsi_i(x^{(i)} + h^{(i)})\right\}.\]
Consequently, we have $h(x) = (h^{(1)}(x),\cdots, h^{(n)}(x)) \in \R^\N$, where
\[h^{(i)}(x) = \arg \min_{t\in \R^{N_i}} \{\ve{\nabla_i f(x)}{t} + \tfrac{\beta w_i}{2}\|t\|_{(i)}^2 + \cPsi_i(x^{(i)}+t)\}.\]
We mentioned in the introduction that besides (block) separability, we require $\cPsi$ to be ``simple''. By this we mean that the above optimization problem leading to $h^{(i)}(x)$ is ``simple'' (e.g., it has  a closed-form solution).  Recall from \eqref{eq:lllop09jhkjh} that $(h(x_k))_{[S_k]}$ is the vector in $\R^\N$ identical to $h(x_k)$ except for blocks $i \notin S_k$, which are zeroed out. Hence, Step 4 of both methods can be written as follows:
\begin{center}
\begin{tabular}{l}
  In parallel for $i\in S_k$ do: \; $x_{k+1}^{(i)} \leftarrow x_k^{(i)} + h^{(i)}(x_k)$.
\end{tabular}
\end{center}

Parameters $\beta$ and $w$ depend on $f$ and $\Srv$ and stay constant throughout the algorithm. We are not ready yet to explain \emph{why} the update is computed via \eqref{eq:h(x)} and \eqref{eq:H_{beta,w}} because we need  technical tools, which will be developed in Section~\ref{SEC:Block_Samplings}, to do so. Here it suffices to say that the parameters $\beta$ and $w$ come from a separable quadratic overapproximation of $\E[f(x+h_{[\Srv]})]$, viewed as a function of $h\in \R^\N$. Since expectation is involved, we refer to this by the name Expected Separable Overapproximation (ESO). This novel concept, developed in this paper, is one of the main tools of our complexity analysis. Section~\ref{SEC:SO} motivates and formalizes the concept, answers the \emph{why} question, and develops some basic ESO theory.

Section~\ref{SEC:ESO_for_PS_functions} is devoted to the computation of $\beta$ and $w$ for partially separable $f$ and various special classes of uniform samplings $\Srv$. Typically we will have $w_i=\Lip_i$, while $\beta$ will depend on \emph{easily computable}  properties of $f$  and $\Srv$. For example, if $\Srv$ is chosen as a subset of $\NN$ of cardinality  $\tau$, with each subset chosen with the same probability (we say that $\Srv$ is $\tau$-nice) then, assuming $n>1$, we may choose $w=L$ and $\beta =1+ \tfrac{(\omega-1)(\tau-1)}{n-1}$, where $\omega$ is the degree of partial separability of $f$. More generally, if $\Srv$ is any uniform sampling with the property $|\Srv|=\tau$ with probability 1, then we may choose $w=L$ and $\beta=\min\{\omega,\tau\}$. Note that in both cases $w=L$ and that the latter $\beta$ is always larger than (or equal to) the former one. This  means, as we will see in Section~\ref{SEC:Iteration_Complexity}, that we can give better complexity results for the former,  more specialized, sampling. We analyze several more options for $\Srv$ than the two just described, and compute parameters $\beta$ and $w$ that should be used with them (for a summary, see Table~\ref{tbl:ESO}).


\textbf{Step 5.} The reason why, besides PCDM1, we also consider PCDM2, is the following: in some situations we are not able to analyze the iteration complexity of PCDM1 (non-strongly-convex $F$ where monotonicity of the method is not guaranteed by other means than by directly enforcing it by inclusion of Step 5). Let us remark that this issue arises for general $\cPsi$ only. It does not exist for $\cPsi=0$, $\cPsi(\cdot) = \lambda \|\cdot\|_1$ and for $\cPsi$ encoding simple constraints on individual blocks; in these cases one does not need to consider PCDM2. Even in the case of general $\cPsi$ we sometimes get monotonicity for free, in which case there is no need to enforce it. Let us stress, however, that we do not recommend implementing PCDM2 as this would introduce too much overhead; in our experience PCDM1 works well even in cases when we can only analyze PCDM2.

\section{Smmary of Contributions}
\label{SEC:Contribute}

In this section we summarize the main contributions of this paper (not in order of significance).

\begin{enumerate}
\item \textbf{Problem generality.} We give the first complexity analysis for  \emph{parallel} coordinate descent methods for  problem \eqref{eq:P} in its full generality.

\item \textbf{Complexity.} We show theoretically (Section~\ref{SEC:Iteration_Complexity}) and numerically (Section~\ref{SEC:Numerical_Experiments}) that PCDM accelerates on its serial counterpart for partially separable problems. In particular, we establish two complexity theorems giving lower bounds on the number of iterations $k$ sufficient for one or both of the PCDM variants (for details, see the precise statements in Section~\ref{SEC:Iteration_Complexity}) to produce a random iterate $x_k$ for which the problem is approximately solved with high probability, i.e.,  $\Prob(F(x_k)-F^* \leq \epsilon) \geq 1-\rho$.  The results,  summarized in Table~\ref{tbl:summary_of_complexity_results}, hold under the standard assumptions listed in Section~\ref{sec:block_structure} \emph{and} the additional assumption that $f,\hat{S},\beta$ and $w$ satisfy the following inequality for all $x,h\in \R^N$:

\begin{equation}\label{eq:ESO_first}\Exp[f(x+h_{[\hat{S}]})] \leq f(x) + \tfrac{\Exp[|\hat{S}|]}{n}\left(\ve{\nabla f(x)}{h} + \tfrac{\beta}{2}\|h\|_w^2\right).\end{equation}

This inequality, which we call Expected Separable Overapproximation (ESO),  is the main new theoretical tool that we develop in this paper for the analysis of our methods (Sections~\ref{SEC:Block_Samplings}-\ref{SEC:ESO_for_PS_functions} are devoted to the development of this theory).

\begin{table}[!h]
\begin{center}
\begin{tabular}{|c|c|c|}
\hline
Setting & Complexity & Theorem\\
\hline
\phantom{$a$} & & \\
Convex $f$  & ${\cal O} \left(\frac{\beta n}{\Exp[|\hat{S}|] }\frac{1}{\epsilon}\log\left(\tfrac{1}{\rho}\right) \right)$ & 19 \\
\phantom{$a$} & &\\
\hline
\phantom{$a$} & &\\
\begin{tabular}{c}
Strongly convex $f$ \\
$\mu_f(w)+\mu_\cPsi(w)>0$
\end{tabular}  & $\frac{n}{\Exp[|\hat{S}|]} \frac{\beta + \mu_\cPsi(w)}{\mu_f(w)+\mu_\cPsi(w)} \log\left(\frac{F(x_0)-F^*}{\epsilon \rho}\right)$ & 20\\
\phantom{$a$} & &\\
\hline
\end{tabular}
\end{center}
\caption{Summary of the main complexity results for PCDM established in this paper.}
\label{tbl:summary_of_complexity_results}
\end{table}

The main observation here is that as the average number of block updates per iteration increases (say, $\hat{\tau}=\Exp[|\hat{S}|]$), enabled by the utilization of more processors, the leading term in the complexity estimate, $n/\hat{\tau}$, decreases in proportion. However, $\beta$ will generally grow with $\hat{\tau}$, which has an adverse effect on the speedup. Much of the theory in this paper goes towards producing formulas for $\beta$ (and $w$), for partially separable $f$ and various classes of \emph{uniform} samplings $\hat{S}$. Naturally, the ideal situation is when $\beta$ does not grow with  $\hat{\tau}$ at all, or if it only grows very slowly. We show that this is the case for partially separable functions $f$ with small $\omega$. For instance, in the extreme case when $f$ is separable ($\omega=1$), we have $\beta=1$ and  we obtain linear speedup in $\hat{\tau}$. As $\omega$ increases, so does $\beta$, depending on the law governing $\hat{S}$. Formulas for $\beta$ and $\omega$ for various samplings $\hat{S}$ are summarized in Table~\ref{tbl:ESO}.

\item \textbf{Algorithm unification.} Depending on the choice of the block structure (as implied by the choice of $n$ and the matrices $U_1,\dots,U_n$) and the way blocks are selected at every iteration (as given by the choice of  $\Srv$), our framework encodes a \emph{family} of known and new algorithms\footnote{All the methods are in their proximal variants due to the inclusion of the term $\Omega$ in the objective.} (see Table~\ref{tbl:special_cases}). 

\begin{table}[!h]
\begin{center}
\begin{tabular}{|c|c|c|}
\hline
Method & Parameters & Comment\\
\hline
Gradient descent  & $n=1$  & \cite{Nesterov:2007composite}\\
Serial random CDM  & $\N_i=1$ for all $i$ and $\Prob(|\Srv|=1)=1$ & \cite{RT:UCDC} \\
Serial block random CDM  & $N_i\geq 1$ for all $i$ and $\Prob(|\Srv|=1)=1$ & \cite{RT:UCDC}\\
Parallel random CDM & $\Prob(|\hat{S}|>1) > 0 $ & NEW\\
Distributed random CDM & $\hat{S}$ is a distributed sampling &  \cite{RT:Hydra2013}\footnote{Remark at revision: This work was planned as a follow-up paper at the time of submitting this paper in Nov 2012; but has appeared in the time between than and revision.} \\
\hline
\end{tabular}
\end{center}
\caption{New and known gradient methods obtained as special cases of our general framework.}
\label{tbl:special_cases}
\end{table}

In particular, PCDM is the first method which ``continuously'' interpolates between serial coordinate descent  and  gradient (by manipulating $n$ and/or $\E[|\hat{S}|]$).

\item \textbf{Partial separability.}  We give the first analysis of a coordinate descent type method dealing with a partially separable loss / objective. In order to run the method, we need to know the Lipschitz constants $\Lip_i$ and the degree of partial separability $\omega$. It is crucial that these quantities are often easily computable/predictable in the huge-scale setting. For example, if $f(x) = \tfrac{1}{2}\|Ax-b\|^2$ and we choose all blocks to be of size $1$, then $L_i$ is equal to the squared Euclidean norm of the $i$-th column of $A$ and $\omega$ is equal to the maximum number of nonzeros in a row of $A$. Many problems in the big data setting have small $\omega$ compared to $n$.

\item \textbf{Choice of blocks.} To the best of our knowledge, existing randomized strategies for paralleling  gradient-type methods (e.g.,  \cite{Bradley:PCD-paper}) assume that $\hat{S}$ (or an equivalent thereof, based on the method) is chosen as a subset of $[n]$ of a fixed cardinality, uniformly at random. We refer to such $\hat{S}$ by the name \emph{nice sampling} in this paper. We relax this assumption and our treatment is hence much more general. In fact, we allow for $\hat{S}$ to be any \emph{uniform} sampling. It is possible to further consider \emph{nonuniform samplings}\footnote{Revision note: See  \cite{RT:NSync}.}, but this is beyond the scope of this paper. 

In particular, as a special case, our method allows for a variable number of blocks to be updated throughout the iterations (this is achieved by the introduction of  \emph{doubly uniform samplings}). This may be useful in some settings such as when the problem is being solved in parallel by $\tau$  unreliable processors each of which computes its update $h^{(i)}(x_k)$ with probability $p_b$ and is busy/down with probability $1-p_b$ (\emph{binomial sampling}).

Uniform, doubly uniform, nice, binomial and other samplings are defined, and their properties studied, in Section~\ref{SEC:Block_Samplings}.

\item \textbf{ESO and formulas for $\beta$ and $w$.} 
In Table~\ref{tbl:ESO} we list parameters $\beta$ and $w$ for which ESO inequality \eqref{eq:ESO_first} holds. Each row corresponds to a specific sampling $\hat{S}$  (see Section~\ref{SEC:Block_Samplings} for the definitions). The last 5 samplings are special cases of one or more of the first three samplings. Details such as what is $\nu,\gamma$ and ``monotonic'' ESO are explained in appropriate sections later in the text. When a specific sampling $\hat{S}$ is used in the algorithm to select blocks in each iteration, the corresponding parameters $\beta$ and $w$ are to be used in the method for the computation of the update (see \eqref{eq:h(x)} and \eqref{eq:H_{beta,w}}). 

\begin{table}[!ht]
\centering
\footnotesize
\begin{tabular}{|l|c|c|c|c|c|}
\hline
\ABC sampling $\Srv$& $\Exp[|\hat{S}|]$ & $\beta$ & $w$ & \begin{tabular}{c}
ESO\\
monotonic?
\end{tabular} & Follows from \\
\hline
\hline
\ABC uniform & $\E[|\Srv|] $ & $1$ & $\nu \odot \Lip$ & No & Thm~\ref{thm:proper_uniform_ESO}\\
 \hline
\ABC nonoverlapping uniform & $\tfrac{n}{l}$  & $1$ & $\gamma \odot \Lip$ & Yes & Thm~\ref{thm:ESO-nonoverlapping}\\
 \hline
\ABC doubly uniform & $\Exp[|\Srv|]$  & $1+\frac{ (\omega-1)\left(\frac{\Exp[|\Srv|^2]}{\Exp[|\Srv|]}-1\right)}{\max(1,n-1)}$ & $\Lip$ & No & Thm~\ref{thm:ESO-DU}\\
 \hline
 \hline
\ABC $\tau$-uniform  & $\tau$  & $\min\{\omega,\tau\}$ & $\Lip$ & Yes & Thm~\ref{thm:proper_uniform_ESO}\\
 \hline
\ABC $\tau$-nice & $\tau$  & $1+      \frac{ (\omega-1)(\tau-1)}{\max(1,n-1)}$  &  $\Lip$ & No & Thm~\ref{thm:ESO-nice}/\ref{thm:ESO-DU} \\
 \hline
\ABC $(\tau,p_b)$-binomial & $\tau p_b$  & $1+      \frac{p_b(\omega-1)(\tau-1)}{\max(1,n-1)}$ & $\Lip$ & No & Thm~\ref{thm:ESO-DU}\\
 \hline
\ABC serial & $1$  & $1$ & $\Lip$ & Yes & Thm~\ref{thm:ESO-nonoverlapping}/\ref{thm:ESO-nice}/\ref{thm:ESO-DU} \\
 \hline
\ABC fully parallel & $n$  & $\omega$ & $\Lip$ & Yes & Thm~\ref{thm:ESO-nonoverlapping}/\ref{thm:ESO-nice}/\ref{thm:ESO-DU}\\
\hline
\end{tabular}
\caption{Values of parameters $\beta$ and $w$ for various samplings $\hat{S}$.}
\label{tbl:ESO}
\end{table}

En route to proving the iteration complexity results for our algorithms, we develop a theory of deterministic and expected separable overapproximation (Sections~\ref{SEC:SO} and \ref{SEC:ESO_for_PS_functions}) which we  believe  is of independent interest, too. For instance, methods based on ESO can be compared favorably to the Diagonal Quadratic Approximation (DQA) approach used in the decomposition of stochastic optimization programs \cite{DQA1995}.


\item \textbf{Parallelization speedup.}
Our complexity results can be used to derive theoretical \emph{parallelization speedup factors}. For several variants of our method, in case of a non-strongly convex objective, these are given in Section~\ref{sec:Complexity-convex} (Table~\ref{tbl:upperbounds}). For instance, in the case when all block are updated at each iteration (we later refer to $\Srv$ having this property by the name \emph{fully parallel} sampling), the speedup factor is equal to $\tfrac{n}{\omega}$. If the problem is separable ($\omega=1$), the speedup is equal to $n$; if the problem is not separable ($\omega=n$), there may be no speedup. For strongly convex $F$ the situation is even better; the details are given in Section~\ref{sec:composite_strong}.


\item \textbf{Relationship to existing results.}    To the best of our knowledge, there are just two papers analyzing a parallel coordinate descent algorithm for convex optimization problems\cite{Bradley:PCD-paper, Necoara:Parallel}. In the first paper all blocks are of size $1$, $\Srv$ corresponds to what we call in this paper a \emph{$\tau$-nice} sampling (i.e., all sets of $\tau$ coordinates are updated at each iteration with equal probability) and hence their algorithm is somewhat comparable to one of the many variants of our general method. While the analysis in \cite{Bradley:PCD-paper} works for a restricted range of values of $\tau$, our results hold for all $\tau\in \NN$. Moreover, the authors consider a more restricted class of functions $f$ and the special case $\cPsi=\lambda \|x\|_1$, which is simpler to analyze. Lastly, the theoretical speedups obtained in \cite{Bradley:PCD-paper}, when compared to the serial CDM method, depend on a quantity $\sigma$ that is hard to compute in big data settings (it involves the computation of an eigenvalue of a huge-scale matrix). Our speedups are expressed in terms of natural and easily computable quantity: the degree $\omega$ of partial separability of $f$. In the setting considered by \cite{Bradley:PCD-paper}, in which more structure is available, it turns out that $\omega$ is  an upper bound\footnote{Revision note requested by a reviewer: In the time since this paper was posted to arXiv, a number of follow-up papers were written analyzing parallel coordinate descent methods and establishing connections between a discrete quantity analogous to $\omega$ (degree of partial/Nesterov separability) and a spectral quantity analogous to $\sigma$ (largest eigenvalue of a certain matrix), most notably \cite{FR:SPCDM2013,RT:Hydra2013}. See also \cite{minibatch-ICML2013}, which uses a spectral quantity, which can be directly compared to $\omega$.} 
on $\sigma$. Hence, we show that one can develop the theory in a more general setting, and that it is not necessary to compute $\sigma$ (which may be complicated in the big data setting).  The parallel CDM method of the second paper \cite{Necoara:Parallel} only allows all blocks to be updated at each iteration. Unfortunately, the analysis (and the method) is too coarse as it does not offer any theoretical speedup when compared to its serial counterpart. In the special case when only a single block is updated in each iteration, uniformly at random, our theoretical results specialize to those established in \cite{RT:UCDC}.


\item \textbf{Computations.} We demonstrate that our method is able to solve a LASSO problem involving a matrix with a billion columns and 2 billion rows on a large memory node with 24 cores in 2 hours (Section~\ref{SEC:Numerical_Experiments}), achieving a $20\times$ speedup compared to the serial variant and pushing the residual by more than 30 degrees of magnitude. While this is done on an artificial problem under ideal conditions (controlling for small $\omega$), large speedups are possible in real data with $\omega$ small relative to $n$. We also perform additional experiments on real data sets from machine learning (e.g., training linear SVMs) to illustrate that the predictions of our theory match reality.

\item \textbf{Code.} The open source code with an efficient implementation of the algorithm(s) developed in this paper is published here: \href{http://code.google.com/p/ac-dc/}{http://code.google.com/p/ac-dc/.}
\end{enumerate}

\section{Block Samplings} \label{SEC:Block_Samplings}

In Step 3 of both PCDM1 and PCDM2 we choose a random set  of blocks  $S_k$ to be updated at the current iteration.
Formally, $S_k$ is a realization of a  \emph{random set-valued mapping} $\Srv$ with values in $2^{\NN}$, the collection of subsets of $[n]$. For brevity, in this paper we  refer to $\Srv$ by the name \emph{sampling}. A sampling $\Srv$ is uniquely characterized by the probability mass function
\begin{equation}
\label{eq:p(S)-general} 
\pp(S) \eqdef \Prob(\Srv = S), \quad S\subseteq \NN;
\end{equation}
that is, by assigning probabilities to all subsets of $\NN$. Further, we let $p = (p_1,\dots,p_n)^T$, where
\begin{equation}\label{eq:p_i}p_i \eqdef \Prob(i \in \Srv).\end{equation}
In Section~\ref{sec:Block_Samplings_1} we describe those samplings for which we analyze our methods and in Section~\ref{sec:technical} we prove several technical results,  which will be useful in the rest of the paper.

\subsection{Uniform, Doubly Uniform and Nonoverlapping Uniform Samplings} \label{sec:Block_Samplings_1}

A sampling is \emph{proper}  if $p_i>0$ for all blocks $i$. That is, from the perspective of PCDM, under a proper sampling  each block gets updated with a positive probability at each iteration. Clearly, PCDM can not converge for a sampling that is not proper.

A sampling $\Srv$ is \emph{uniform} if all blocks get updated with the same probability, i.e., if $p_i=p_j$ for all $i,j$. We show in \eqref{eq:simple2unif} that, necessarily, $p_i = \tfrac{\E[|\Srv|]}{n}$.  Further, we say $\Srv$ is \emph{nil} if $\Prob(\emptyset) = 1$. Note that a uniform sampling is proper if and only if it is not nil.

All our iteration complexity results in this paper are for PCDM used with a proper uniform sampling (see Theorems~\ref{thm:complexity-convex-case} and \ref{thm:complexity-strongly-convex-case}) for which we can compute $\beta$ and $w$ giving rise to an inequality (we we call ``expected separable overapproximation'') of the form \eqref{eq:general_form_for_expectation}. We derive such inequalities for all proper uniform samplings (Theorem~\ref{thm:proper_uniform_ESO}) as well as refined results for two special subclasses thereof: doubly uniform samplings (Theorem~\ref{thm:ESO-DU}) and nonoverlapping uniform samplings (Theorem~\ref{thm:ESO-nonoverlapping}). We will now give the definitions:

\begin{enumerate}
\item \textbf{Doubly Uniform (DU) samplings.} A DU sampling is one which generates all sets of equal cardinality with equal probability. That is, $\pp(S')=\pp(S'')$ whenever $|S'| = |S''|$. The name comes from the fact that this definition postulates a different uniformity property, ``standard'' uniformity is a consequence. Indeed, let us show that a DU sampling is necessarily uniform. Let $q_j = \Prob(|\Srv| = j)$ for $j=0,1,\dots, n$ and note that from the definition   we know that whenever $S$ is of cardinality $j$, we have $\pp(S) = q_j/{n \choose j}$. Finally, using this we obtain
\[p_i = \sum_{S:i\in S}  \pp(S) = \sum_{j=1}^n \sum_{\substack{S: i \in S\\ |S|=j}} \pp(S) = \sum_{j=1}^n \sum_{\substack{S: i \in S\\ |S|=j}} \tfrac{q_j}{{n \choose j}} =\sum_{j=1}^n   \tfrac{{n-1\choose j-1}}{{n\choose j}} q_j = \tfrac{1}{n}\sum_{j=1}^n q_j j = \tfrac{\E[|\Srv|]}{n}.\]

It is clear that each DU sampling is uniquely characterized by the vector of probabilities $q$; its density function is given by
\begin{equation}\pp(S) = \frac{q_{|S|}}{{n \choose |S|}}, \quad S \subseteq \NN. \label{eq:p(S)}\end{equation}

\item \textbf{Nonoverlapping Uniform (NU) samplings.} A  NU sampling is one which is uniform and which assigns positive probabilities only to sets forming a partition of $\NN$.
Let $S^1,S^2,\dots, S^l$ be a partition of $\NN$, with $|S^j|>0$ for all $j$. The density function of a NU sampling corresponding to this partition is given by
    \begin{equation}\label{eq:non-overlap-uniform}\pp(S) = \begin{cases}\tfrac{1}{l}, & \quad \text{if } S \in \{S^1,S^2,\dots,S^l\},\\ 0, & \quad \text{otherwise.}\end{cases}
    \end{equation}
    Note that $\E[|\Srv|] = \tfrac{n}{l}$.

\end{enumerate}

\noindent Let us now describe several interesting special cases of DU and NU samplings:

\begin{itemize}
\item[3.] \textbf{Nice sampling.} Fix $1\leq \tau \leq n$. A $\tau$-nice sampling is a DU sampling with $q_\tau =  1$.

\emph{Interpretation:} There are $\tau$ processors/threads/cores available. At the beginning of each iteration we choose a set of blocks using a $\tau$-nice sampling (i.e., each subset of $\tau$ blocks is chosen with the same probability), and assign each block to a dedicated processor/thread/core. Processor assigned with block $i$ would compute and apply the update $h^{(i)}(x_k)$. This is the sampling we use in our computational experiments.

\item[4.] \textbf{Independent sampling.} Fix $1\leq \tau \leq n$. A $\tau$-independent sampling is a DU sampling with
\[q_k = \begin{cases}{n \choose k} c_k, \quad & k=1,2,\dots,\tau,\\
0, \quad & k=\tau+1, \dots, n,
\end{cases}\]
where $c_1 = \left(\tfrac{1}{n}\right)^\tau$ and $c_{k} = \left(\tfrac{k}{n}\right)^\tau - \sum_{i=1}^{k-1} {k \choose i}c_i$ for $k \geq 2$.

\emph{Interpretation:} There are $\tau$ processors/threads/cores available. Each processor chooses one of the $n$ blocks, uniformly at random and independently of the other processors. It turns out that the set $\Srv$ of  blocks selected this way is DU with $q$ as given above. Since in one parallel iteration of our methods each block in $\Srv$ is updated exactly once, this means that if two or more processors pick the same block, all but one will be idle. On the other hand, this sampling can be generated extremely easily and in parallel! For $\tau\ll n$ this sampling is a good (and fast) approximation of the $\tau$-nice sampling. For instance, for $n=10^3$ and $\tau=8$ we have $q_8=0.9723$, $q_7=0.0274$, $q_6=0.0003$ and $q_k\approx 0$ for $k=1,\dots,5$.

\item[5.] \textbf{Binomial sampling.} Fix $1\leq \tau \leq n$ and $0< p_b \leq 1$. A $(\tau,p_b)$-binomial sampling is defined as a DU sampling with
\begin{equation} \label{eq:bin1}q_k =  {\tau \choose k} p_b^k (1-p_b)^k, \quad k=0,1,\dots,\tau.\end{equation}
Notice that $\E[|\Srv|]  =\tau p_b$ and $\E[|\Srv|^2] = \tau p_b(1+ \tau p_b - p_b)$.

\emph{Interpretation:} Consider the following  situation with \emph{independent equally unreliable processors}. We have $\tau$ processors, each of which is at any given moment available with probability $p_b$ and busy with probability $1-p_b$, independently of the availability of the other processors. Hence, the number of available processors (and hence blocks that can be updated in parallel) at each iteration is a binomial random variable with parameters $\tau$ and $p_b$. That is, the number of available processors is equal to $k$ with probability $q_k$.
\begin{itemize}
\item \emph{Case 1 (explicit selection of blocks):} We learn that $k$ processors are available \emph{at the beginning} of each iteration. Subsequently, we choose $k$ blocks using a $k$-nice sampling and ``assign one block'' to each of the $k$ available processors.
\item \emph{Case 2 (implicit selection of blocks):} We choose $\tau$ blocks using a $\tau$-nice sampling and assign one to each of the $\tau$ processors (we do not know which will be available at the beginning of the iteration). With probability $q_k$, $k$ of these will send their updates. It is easy to check that the resulting effective sampling of blocks is $(\tau,p_b)$-binomial.
\end{itemize}


\item[6.] \textbf{Serial sampling.} This is a DU sampling with $q_1 =  1$. Also, this is a NU sampling with $l=n$ and $S^j=\{j\}$ for $j=1,2,\dots,l$. That is, at each iteration we update a single block, uniformly at random. This was studied in \cite{RT:UCDC}.
\item[7.] \textbf{Fully parallel sampling.} This is a DU sampling with $q_n = 1$. Also, this is a NU sampling with $l=1$ and $S^1 = \NN$. That is, at each iteration we update all blocks.
\end{itemize}



The following simple result says that the intersection between the class of DU and NU samplings is very thin.  A sampling is called \emph{vacuous} if $\Prob(\emptyset)>0$.

\begin{proposition}
There are precisely two nonvacuous samplings which are both DU and NU: i) the serial   sampling and ii) the fully parallel sampling.
\end{proposition}
\begin{proof} Assume $\Srv$ is nonvacuous, NU and DU. Since $\Srv$ is nonvacuous, $\Prob(\Srv = \emptyset)=0$. Let $S\subset \NN$ be any set for which $\Prob(\Srv=S)>0$. If $1<|S|<n$, then there exists $S'\neq S$ of the same cardinality as $S$ having a nonempty intersection with $S$. Since $\Srv$ is doubly uniform, we must have $\Prob(\Srv=S') = \Prob(\Srv = S')>0$. However, this contradicts the fact that $\Srv$ is non-overlapping. Hence, $\Srv$ can only generate sets of cardinalities $1$ or $n$ with positive probability, but not both. One option leads to the fully parallel sampling, the other one leads to the serial sampling.
\end{proof}

\subsection{Technical results} \label{sec:technical}

For a given sampling $\Srv$ and $i,j \in \NN$ we let
\begin{equation}\label{eq:p_ij}p_{ij} \eqdef \Prob(i \in \Srv, j\in \Srv)  = \sum_{S: \{i,j\}\subset S} \pp(S).\end{equation}
The following simple result has several consequences  which will be used throughout the paper.

\begin{lemma}[Sum over a random index set]\label{lem:basic} Let $\emptyset \neq J\subset \NN$ and $\Srv$ be any sampling. If $\theta_i$, $i\in \NN$, and $\theta_{ij}$, for $(i,j) \in \NN\times \NN$ are real constants, then\footnote{Sum over an empty index set  will, for convenience, be defined to be zero.}
\[\Exp \left[ \sum_{i\in J \cap \Srv} \theta_i \right] = \sum_{i\in J} p_i \theta_i,\]
\begin{equation}\label{eq:0987i4}\Exp \left[ \sum_{i\in J \cap \Srv} \theta_i \;|\; |J\cap \Srv| = k\right] = \sum_{i\in J} \Prob(i \in \Srv \;|\; |J \cap \Srv| = k) \theta_i,\end{equation}
\begin{equation}\label{eq:jkajksa87}\Exp \left[ \sum_{i\in J\cap \Srv} \sum_{j \in J\cap \Srv}\theta_{ij} \right] = \sum_{i \in J} \sum_{j\in J} p_{ij} \theta_{ij}.\end{equation}
\end{lemma}

\noindent {\em Proof.}
We prove the first statement, proof of the remaining statements is essentially identical:
\begin{align*}
\Exp \left[ \sum_{i\in J \cap \Srv} \theta_i \right] \overset{\eqref{eq:p(S)-general}}{=} \sum_{S\subset \NN} \left(\sum_{i\in J\cap S} \theta_i \right) \pp(S) = \sum_{i\in J} \sum_{S: i\in S}  \theta_i \pp(S)
= \sum_{i\in J} \theta_i \sum_{S: i\in S}  \pp(S) = \sum_{i\in J} p_i \theta_i. \quad \qed
\end{align*}



The consequences are summarized in the next theorem and the discussion that follows.

\begin{theorem}\label{thm:first_theorem} Let $\emptyset \neq J \subset \NN$ and $\Srv$ be an arbitrary sampling. Further, let $a,h\in \R^\N$, $w\in \R^n_+$ and let $g$ be a block separable function, i.e., $g(x) = \sum_i g_i(x^{(i)})$. Then
\begin{eqnarray}
\label{eq:expected_sampling_size}\Exp\left[|J\cap\Srv|\right] &=& \sum_{i\in J} p_i,\\
\label{eq:expected_sampling_size2}\Exp\left[|J\cap\Srv|^2\right] &=& \sum_{i\in J} \sum_{j \in J} p_{ij},\\
\label{eq:simple3}\Exp \left[\ve{a}{h_{[\Srv]}}_w\right] &=& \ve{a}{h}_{p\odot w},\\
\label{eq:simple4}\Exp \left[\|h_{[\Srv]}\|_w^2 \right] &= & \|h\|^2_{p \odot w},\\
\label{eq:js0as9} \Exp \left[g(x+\vsubset{h}{\Srv})\right] &=& \sum_{i=1}^n \left[p_i g_i(x^{(i)}+h^{(i)}) + (1-p_i)g_i(x^{(i)}) \right].
\end{eqnarray}
Moreover, the matrix $P \eqdef (p_{ij})$ is positive semidefinite.
\end{theorem}

\noindent {\em Proof.}
Noting that $|J\cap \Srv| = \sum_{i\in J\cap \Srv} 1$, $|J\cap \Srv|^2 = (\sum_{i\in J\cap \Srv} 1)^2 = \sum_{i \in J\cap \Srv}\sum_{j \in J \cap \Srv} 1$, $\ve{a}{h_{[\Srv]}}_w = \sum_{i\in \Srv} w_i \ve{a^{(i)}}{h^{(i)}}$, $\|h_{[\Srv]}\|_w^2 = \sum_{i\in \Srv} w_i \|h^{(i)}\|_{(i)}^2$ and
\[g(x+\vsubset{h}{\Srv}) = \sum_{i\in \Srv} g_i(x^{(i)}+h^{(i)}) + \sum_{i\notin \Srv} g_i(x^{(i)}) = \sum_{i\in \Srv} g_i(x^{(i)}+h^{(i)}) + \sum_{i=1}^n g_i(x^{(i)}) - \sum_{i\in \Srv} g_i(x^{(i)}),\]
all five identities  follow directly by applying Lemma~\ref{lem:basic}. Finally, for any $\theta = (\theta_1,\dots,\theta_n)^T\in \R^n$,
\[\theta^T P \theta = \sum_{i=1}^n \sum_{j=1}^n p_{ij} \theta_i \theta_j \overset{\eqref{eq:jkajksa87}}{=} \E [(\sum_{i \in \Srv} \theta_i)^2]\geq 0. \quad \qed \]

The above results hold for arbitrary samplings. Let us specialize them, in order of decreasing generality, to uniform, doubly uniform and nice samplings.

\begin{itemize}
\item \textbf{Uniform samplings.} If $\Srv$ is uniform, then from  \eqref{eq:expected_sampling_size} using $J=\NN$ we get
\begin{equation}\label{eq:simple2unif}p_i = \tfrac{\Exp \left[|\Srv|\right]}{n}, \qquad i \in \NN.\end{equation}
Plugging \eqref{eq:simple2unif} into \eqref{eq:expected_sampling_size}, \eqref{eq:simple3}, \eqref{eq:simple4} and \eqref{eq:js0as9} yields
\begin{equation}\label{eq:simple2unifXXX}\Exp \left[|J\cap\Srv|\right] = \tfrac{|J|}{n}\E[|\Srv|],\end{equation}
\begin{equation}\label{eq:simple3unif}\Exp \left[\ve{a}{h_{[\Srv]}}_w\right] =  \tfrac{\Exp\left[|\Srv|\right]}{n} \ve{a}{h}_w,\end{equation}
\begin{equation}\label{eq:simple4unif}\E \left[\|h_{[\Srv]}\|_w^2 \right] =   \tfrac{\Exp\left[|\Srv|\right]}{n}  \|h\|^2_{w},\end{equation}
\begin{equation}\label{eq:separable_uniform}\Exp \left[g(x+\vsubset{h}{\Srv})\right] = \tfrac{\Exp[|\Srv|]}{n} g(x+h) + \left(1-\tfrac{\Exp [|\Srv|]}{n}\right)g(x).
\end{equation}

\item \textbf{Doubly uniform samplings.} Consider the case $n>1$; the case $n=1$ is trivial. For doubly uniform $\Srv$, $p_{ij}$ is constant for $i\neq j$:
\begin{equation}\label{eq:p_ij_equals}p_{ij} = \tfrac{\E[|\Srv|^2-|\Srv|]}{n(n-1)}. \end{equation}
Indeed, this follows from
\[p_{ij} = \sum_{k=1}^n \Prob(\{i,j\}\subseteq \Srv \;|\; |\Srv| = k)\Prob(|\Srv|=k) = \sum_{k=1}^n \tfrac{k(k-1)}{n(n-1)}\Prob(|\Srv|=k).\]
Substituting \eqref{eq:p_ij_equals} and \eqref{eq:simple2unif} into \eqref{eq:expected_sampling_size2} then gives
\begin{equation}\label{eq:98sjs8}\E[|J \cap \Srv|^2] = (|J|^2 - |J|)\tfrac{\E[|\Srv|^2-|\Srv|]}{n\max\{1,n-1\}} + |J|\tfrac{|\Srv|}{n}.\end{equation}

\item \textbf{Nice samplings.} Finally, if $\Srv$ is $\tau$-nice (and $\tau \neq 0$), then $\E[|\Srv|]=\tau$ and $\E[|\Srv|^2] = \tau^2$, which used in \eqref{eq:98sjs8} gives
\begin{equation}\label{eq:00sjs738}\E[|J \cap \Srv|^2] =  \tfrac{|J|\tau}{n}\left(1+ \tfrac{(|J| - 1)(\tau-1)}{\max\{1,n-1\}}\right).\end{equation}
Moreover, assume that $\Prob(|J \cap \Srv|=k) \neq 0$ (this happens precisely when $0\leq k \leq |J|$ and $k \leq \tau \leq n-|J|+k$). Then for all $i \in J$,
\[
\Prob(i \in \Srv \;|\; |J\cap \Srv| = k) =  \frac{{|J| -1 \choose k-1}{n-|J| \choose \tau - k}}{{|J| \choose k}{n-|J| \choose \tau-k}} =
\frac{k}{|J|}.
\]
Substituting this into \eqref{eq:0987i4} yields
\begin{equation}\label{eq:JEJEJEJE}
\E\left[\sum_{i \in J \cap \Srv} \theta_i \;|\; |J\cap \Srv| = k \right] = \tfrac{k}{|J|}\sum_{i\in J} \theta_i.
\end{equation}

\end{itemize}


\section{Expected Separable Overapproximation} \label{SEC:SO}


Recall that given $x_k$, in PCDM1 the next iterate is the random vector $x_{k+1} = x_k + h_{[\Srv]}$ for a particular choice of $h \in \R^\N$. Further recall that in PCDM2,
\[x_{k+1} = \begin{cases}x_k+h_{[\Srv]}, & \text{if } F(x_k+h_{[\Srv]})\leq F(x_k),\\ x_k, & \text{otherwise,}\end{cases}\]
again for a particular choice of $h$. While in Section~\ref{SEC:PCDM} we mentioned  how $h$ is computed, i.e., that $h$ is the minimizer of $H_{\beta,w}(x,\cdot)$ (see \eqref{eq:h(x)} and \eqref{eq:H_{beta,w}}), we did not explain \emph{why} is $h$ computed this way. The reason for this is that the tools needed for this were not yet developed at that point (as we will see, some results from Section~\ref{SEC:Block_Samplings} are needed). In this section we give an answer to this \emph{why} question.

Given $x_k\in \R^\N$, after one step of PCDM1 performed with update $h$ we get $\E[F(x_{k+1})\;|\; x_k]  = \E[F(x_k+h_{[\Srv]})\;|\; x_k]$.
On the the other hand, after one step of PCDM2  we have
\[\E[F(x_{k+1})\;|\; x_k] = \E[\min\{F(x_k+h_{[\Srv]}),F(x_k)\}\;|\; x_k] \leq \min\{\E[F(x_k+h_{[\Srv]})\;|\; x_k],F(x_k)\}.\]
So, for both PCDM1 and PCDM2 the following estimate holds,
\begin{equation}\label{eq:ineq888}\E[F(x_{k+1})\;|\; x_k]  \leq \E[F(x_k+h_{[\Srv]})\;|\; x_k].\end{equation}
A good choice for $h$ to be used in the algorithms would be one minimizing the right hand side of inequality \eqref{eq:ineq888}. At the same time, we would like the minimization process to be decomposable so that the updates $h^{(i)}$, $i \in \Srv$, could be computed in parallel. However, the problem of finding such $h$ is intractable in general  even if we do not require parallelizability. Instead, we propose to construct/compute a ``simple'' separable overapproximation of the right-hand side of \eqref{eq:ineq888}. Since the overapproximation will be separable, parallelizability is guaranteed; ``simplicity'' means that the updates $h^{(i)}$ can be computed easily (e.g., in closed form).

From now on we replace, for simplicity and w.l.o.g., the random vector $x_k$ by a fixed deterministic vector $x\in \R^\N$. We can thus remove conditioning in \eqref{eq:ineq888} and instead study the quantity $\E[F(x+h_{[\Srv]})]$. Further, fix $h \in \R^\N$. Note that if  we can find $\beta>0$ and $w\in \R^n_{++}$ such that
\begin{eqnarray}\label{eq:general_form_for_expectation}
 \Exp\left[f\left(x+\vsubset{h}{\Srv}\right)\right]
&\leq&
  f(x)+ \tfrac{\E[|\Srv|]}{n} \left( \ve{\nabla f(x)}{h}  + \tfrac{\beta}{2}\|h\|_w^2 \right),
 \end{eqnarray}
we indeed find a simple separable overapproximation of $\E[F(x+h_{[\Srv]})]$:
\begin{eqnarray}
 \Exp[F(x+\vsubset{h}{\Srv})]
 &\stackrel{\eqref{eq:P}}{=}& \Exp [f(x+\vsubset{h}{\Srv})+\cPsi(x+\vsubset{h}{\Srv})] \notag \\
 &\overset{\eqref{eq:general_form_for_expectation}, \eqref{eq:separable_uniform}}{\leq}& f(x)+
\tfrac{\E[|\Srv|]}{n}\left( \ve{ \nabla f(x)}{h} + \tfrac{\beta}{2} \|h\|_w^2\right) + \left(1-\tfrac{\E[|\Srv|]}{n}\right)\cPsi(x) + \tfrac{\E[|\Srv|]}{n} \cPsi(x+h) \notag\\
& =&  \left(1-\tfrac{\E[|\Srv|]}{n}\right)F(x) + \tfrac{\E[|\Srv|]}{n} H_{\beta,w}(x,h), \label{eq:1}
\end{eqnarray}
where we recall from \eqref{eq:H_{beta,w}} that
$H_{\beta,w}(x,h) =  f(x)+ \ve{\nabla f(x)}{h} +\tfrac{\beta}{2}      \|h\|_w^2 +     \cPsi(x+h)$.

That is, \eqref{eq:1} says that the expected objective value after one parallel step of our methods, if block $i\in \Srv$ is updated by $h^{(i)}$, is bounded above by a convex combination of $F(x)$ and  $H_{\beta,w}(x,h)$. The natural choice of $h$ is to set
\begin{equation}\label{eq:ggg67} h(x) = \arg \min_{h\in \R^\N} H_{\beta,w}(x,h).\end{equation} Note that this is precisely the choice we make in our methods. Since $H_{\beta,w}(x,0) = F(x)$, both PCDM1 and PCDM2 are monotonic in expectation.

The above discussion leads to the following definition.

\begin{definition}[Expected Separable Overapproximation (ESO)] Let $\beta > 0$, $w\in \R^n_{++}$ and let $\Srv$ be a proper uniform sampling. We say that $f:\R^\N\to \R$ admits a $(\beta,w)$-ESO with respect to $\Srv$ if  inequality \eqref{eq:general_form_for_expectation} holds for all $x,h\in \R^\N$.
For simplicity, we  write $(f,\Srv) \sim ESO(\beta,w)$.
\end{definition}

A few remarks:
\begin{enumerate}
\item \textbf{Inflation.} If $(f,\Srv) \sim ESO(\beta , w)$, then for $\beta' \geq \beta$ and $w'\geq w$, $(f,\Srv)\sim ESO(\beta',w')$.
\item \textbf{Reshuffling.}  Since for any $c>0$ we have $\|h\|_{c w}^2 = c\|h\|_{w}^2$, one can ``shuffle'' constants between $\beta$ and $w$ as follows:
\begin{equation}\label{eq:ESO_shift} (f,\Srv)\sim ESO(c \beta,w) \quad \Leftrightarrow \quad (f,\Srv)\sim ESO(\beta,c w), \qquad  c > 0.\end{equation}
\item \textbf{Strong convexity.} If $(f,\Srv) \sim ESO(\beta , w)$, then
\begin{equation}\label{eq:beta-mu}\beta \geq \mu_f(w).\end{equation}
Indeed, it suffices to take expectation in \eqref{eq:strong_def} with $y$ replaced by $x+h_{[\Srv]}$ and compare the resulting inequality with  \eqref{eq:general_form_for_expectation} (this gives $\beta\|h\|_w^2 \geq \mu_f(w)\|h\|_w^2$, which must hold for all $h$).
\end{enumerate}

Recall that Step 5 of PCDM2 was introduced so as to explicitly enforce monotonicity into the method as in some situations, as we will see in Section~\ref{SEC:Iteration_Complexity}, we can only analyze a monotonic algorithm. However, sometimes even PCDM1 behaves monotonically (without enforcing this behavior externally as in PCDM2). The following definition captures this.

\begin{definition}[Monotonic ESO] Assume $(f,\Srv) \sim ESO(\beta,w)$ and let $h(x)$ be as in \eqref{eq:ggg67}. We say that the ESO is \emph{monotonic} if  $F(x+(h(x))_{[\Srv]}) \leq F(x)$, with probability 1, for all $x \in \dom F$.
\end{definition}

\subsection{Deterministic Separable Overapproximation (DSO) of Partially Separable Functions} \label{subsec:DSO}

The following theorem will be useful in deriving  ESO  for  uniform samplings (Section~\ref{sec:uniform}) and  nonoverlapping uniform samplings (Section~\ref{sec:NOuniform}). It will also be useful in establishing monotonicity of some ESOs (Theorems~\ref{thm:proper_uniform_ESO} and \ref{thm:ESO-nonoverlapping}).

\begin{theorem}[DSO]\label{thm:globalLipConstant} Assume $f$ is partially separable (i.e., it can be written in the form \eqref{eq:strucutre_of_f}).
Letting $\support(h)\eqdef \{i \in \NN\;:\; h^{(i)}\neq 0\}$, for all $x, h\in \R^\N$ we have
\begin{equation}\label{eq:upperBoundOnLipConst}
 f(x+h) \leq f(x)+\ve{\nabla f(x)}{ h} +  \frac{\max_{J\in\JJ} |J\cap \support(h)|}2 \|h\|_L^2.
\end{equation}
\end{theorem}

\noindent {\em Proof.}
Let us fix $x$  and define $\phi(h)\eqdef f(x+h) - f(x)- \ve{\nabla f(x)}{h}$.
Fixing $h$, we need to show that
$\phi(h) \leq \frac{\theta}{2} \|h\|_L^2$ for $\theta = \max_{J\in \JJ} \theta^J$, where
$\theta^J \eqdef |J\cap \support(h)|$. One can define functions $\phi^J$ in an analogous fashion from the constituent functions $f_J$, which satisfy
\begin{equation}\label{eq:phi-sum}\phi(h) = \sum_{J\in \JJ} \phi^J (h),\end{equation}
\begin{equation}\label{eq:987987} \phi^J(0) = 0, \qquad J \in \JJ.
\end{equation}
Note that \eqref{eq:Lipschitz_ineq} can be written as
\begin{equation}\label{eq:987980980} \phi(\U_i h^{(i)}) \leq \tfrac{\Lip_i}{2}\|h^{(i)}\|_{(i)}^2, \qquad i=1,2,\dots,n.
\end{equation}
Now, since $\phi^J$ depends on the intersection of $J$ and the support of its argument only, we have
\begin{equation}\label{eq:090909} \phi(h) \stackrel{\eqref{eq:phi-sum}}{=} \sum_{J\in \JJ} \phi^J(h) =  \sum_{J\in \JJ} \phi^J\left(\sum_{i=1}^n \U_i h^{(i)}\right)
= \sum_{J\in \JJ} \phi^J\left( \sum_{i\in J\cap \support(h)} \U_i\vc{h}{i}\right).\end{equation}

The argument in the last expression can be written as a convex combination of $1+\theta^J$ vectors: the zero vector (with weight $\tfrac{\theta-\theta^J}{\theta}$) and the $\theta^J$ vectors $\{\theta\U_i \vc{h}{i}: i\in J\cap \support(h)\}$ (with weights $\tfrac{1}{\theta}$):
\begin{equation}\label{eq:convex_comb}\sum_{i\in J\cap \support(h)} \U_i \vc{h}{i} = \left(\tfrac{\theta-\theta^J}{\theta} \times 0 \right) + \left(\tfrac{1}{\theta} \times \sum_{i\in J\cap \support(h)} \theta \U_i\vc{h}{i}\right).\end{equation}

Finally, we plug \eqref{eq:convex_comb} into \eqref{eq:090909} and use convexity and some simple algebra:
\begin{eqnarray*}
 \phi(h)&\leq&  \sum_{J\in \JJ} \left[\tfrac{\theta-\theta^J}{\theta} \phi^J(0) + \tfrac{1}{\theta}\sum_{i\in J\cap \support(h)} \phi^J(\theta\U_i\vc{h}{i})\right] \stackrel{\eqref{eq:987987}}{=} \tfrac{1}{\theta}  \sum_{J\in \JJ}  \sum_{i\in J\cap \support(h)} \phi^J(\theta\U_i\vc{h}{i}) \\
 &=& \tfrac{1}{\theta}  \sum_{J\in \JJ}  \sum_{i=1}^n \phi^J(\theta\U_i\vc{h}{i}) = \tfrac{1}{\theta} \sum_{i=1}^n \sum_{J\in \JJ}  \phi^J(\theta\U_i\vc{h}{i}) \stackrel{\eqref{eq:phi-sum}}{=} \tfrac{1}{\theta} \sum_{i=1}^n \phi(\theta\U_i\vc{h}{i}) \\
 &\stackrel{\eqref{eq:987980980}}{\leq}& \tfrac{1}{\theta} \sum_{i=1}^n  \tfrac{\Lip_i}{2}  \ncs{\theta\vc{h}{i}}{i} =  \tfrac{\theta}{2}\|h\|_L^2. \qquad \qed
\end{eqnarray*}

Besides the usefulness of the above result in deriving ESO inequalities, it is interesting on its own for the following reasons.

\begin{enumerate}
\item \textbf{Block Lipschitz continuity of $\nabla f$.} The DSO inequality \eqref{eq:upperBoundOnLipConst} is a generalization of \eqref{eq:Lipschitz_ineq} since \eqref{eq:Lipschitz_ineq} can be recovered from \eqref{eq:upperBoundOnLipConst} by choosing $h$ with $\support(h)=\{i\}$ for $i\in \NN$.
\item \textbf{Global Lipschitz continuity of $\nabla f$.} The DSO inequality also says that the gradient of $f$ is Lipschitz with Lipschitz constant $\omega$ with respect to the norm $\|\cdot\|_L$:
    \begin{equation}\label{eq:BigLipschitz}f(x+h) \leq f(x) + \ve{\nabla f(x)}{h} + \tfrac{\omega}{2}\|h\|_L^2.\end{equation}
    Indeed, this follows from \eqref{eq:upperBoundOnLipConst} via $\max_{J\in \JJ} |J\cap \support(h)| \leq \max_{J\in \JJ} |J| = \omega$.
     For $\omega=n$ this has been shown in \cite{Nesterov:2010RCDM}; our result for partially separable functions appears to be new.
\item \textbf{Tightness of the global Lipschitz constant.} The Lipschitz constant $\omega$ is ``tight'' in the following sense: there are functions for which $\omega$ cannot be replaced in \eqref{eq:BigLipschitz} by any smaller number. We will show this on a simple example. Let $f(x)=\tfrac{1}{2}\|Ax\|^2$ with $A\in \R^{m\times n}$ (blocks are of size 1). Note that we can write $f(x+h) = f(x) + \ve{\nabla f(x)}{h} + \tfrac{1}{2}h^T A^T A h$,     and that $L=(L_1,\dots,L_n)=\diag(A^TA)$. Let $D=\Diag(L)$. We need to argue that there exists $A$ for which $\sigma \eqdef \max_{h\neq 0} \tfrac{h^T A^T A h}{\|h\|_L^2} = \omega$. Since we know that $\sigma\leq \omega$ (otherwise \eqref{eq:BigLipschitz} would not hold), all we need to show is that there is $A$ and $h$ for which
    \begin{equation}\label{eq:tightness007}h^T A^T A h = \omega h^T D h.\end{equation}

    Since $f(x) = \sum_{i=1}^m (A_j^Tx)^2$, where $A_j$ is the $j$-th row of $A$, we assume that each row of $A$ has at most $\omega$ nonzeros (i.e., $f$ is partially separable of degree $\omega$). Let us pick $A$ with the following further properties: a) $A$ is a 0-1 matrix, b) all rows of $A$ have exactly $\omega$ ones, c) all columns of $A$ have exactly the same number ($k$) of ones.
    Immediate consequences: $L_i = k$ for all $i$, $D = k I_n$ and $\omega m = kn$. If we let $e_m$ be the $m\times 1$ vector of all ones and $e_n$ be the $n\times 1$ vector of all ones, and set $h = k^{-1/2}e_n$, then
    \[h^T A^T A h = \tfrac{1}{k} e_n^T A^T A e_n = \tfrac{1}{k} (\omega e_m)^T (\omega e_m)  =\tfrac{\omega^2 m}{k} = \omega n = \omega \tfrac{1}{k}e_n^T k I_n e_n = \omega h^T D h,\]
    establishing \eqref{eq:tightness007}. Using similar techniques one can easily prove the following more general result: Tightness also occurs for matrices $A$ which in each row contain $\omega$ identical nonzero elements (but which can vary from row to row).

\end{enumerate}

\subsection{ESO for a convex combination of samplings}

Let $\hat{S}_1, \hat{S}_2, \dots, \hat{S}_m$ be a collection of samplings and let $q\in \R^m$ be a probability vector. By $\sum_j q_j \Srv_j$ we denote the sampling $\Srv$ given by
\begin{equation}\label{eq:convx_comb_2_samplings}\Prob\left(\Srv =S\right) = \sum_{j=1}^m q_j \Prob(\Srv_j = S).\end{equation}
This procedure allows us to build new samplings from existing ones. A natural interpretation of $\Srv$ is that it arises from a two stage process as follows. Generating a set via $\Srv$ is equivalent to first choosing $j$ with probability $q_j$, and then generating a set via $\Srv_j$.

\begin{lemma}\label{lem:double_stuff} Let $\hat{S}_1, \hat{S}_2, \dots, \hat{S}_m$ be arbitrary samplings, $q\in \R^m$ a probability vector and $\kappa: 2^\NN \to \R$ any function mapping subsets of $\NN$ to reals. If we let $\Srv = \sum_{j} q_j \Srv_j$, then
\begin{enumerate}
 \item[(i)]  $\E[\kappa(\Srv)] = \sum_{j=1}^m q_j \E[\kappa(\Srv_j)]$,
 \item[(ii)] $\E[|\Srv|] = \sum_{j=1}^m q_j \E[|\Srv_j|] $,
 \item[(iii)] $\Prob(i \in \Srv) = \sum_{j=1}^m q_j \Prob(i\in \Srv_j)$, for any $i=1,2,\dots,n$,
 \item[(iv)] If $\Srv_1,\dots,\Srv_m$ are uniform (resp.\ doubly uniform), so is $\Srv$.
\end{enumerate}
\end{lemma}
\begin{proof} Statement (i) follows by writing $\E [\kappa(\Srv) ]$ as
\[
\sum_{S} \Prob (\Srv = S)\kappa(S)
\overset{\eqref{eq:convx_comb_2_samplings}}{=}  \sum_{S}  \sum_{j=1}^m q_j \Prob(\Srv_j = S)\kappa(S) =  \sum_{j=1}^m q_j \sum_{S}   \Prob(\Srv_j = S)\kappa(S) = \sum_{j=1}^m q_j \E[\kappa(\Srv_j)].
\]
Statement (ii) follows from (i) by choosing $\kappa(S) = |S|$, and (iii) follows from (i) by choosing $\kappa$ as follows: $\kappa(S) = 1$ if $i\in S$ and $\kappa(S)=0$ otherwise. Finally, if the samplings $\Srv_j$ are uniform, from \eqref{eq:simple2unif}  we know that $\Prob(i \in \Srv_j) = \E[|\Srv_j|]/n$ for all $i$ and $j$. Plugging this into identity (iii) shows that $\Prob(i \in \Srv)$ is independent of $i$, which shows that $\Srv$ is uniform. Now assume that $\Srv_j$ are doubly uniform. Fixing arbitrary $\tau \in \{0\} \cup \NN$, for every $S \subset \NN$ such that $|S|=\tau$ we have \[\Prob(\Srv = S)
\overset{\eqref{eq:convx_comb_2_samplings}}{=} \sum_{j=1}^m q_j \Prob(\Srv_j = S) =  \sum_{j=1}^m q_j \frac{\Prob(|\Srv_j|=\tau)}{{n \choose \tau}}.\]
As the last expression depends on $S$  via $|S|$ only, $\Srv$ is doubly uniform.
\end{proof}

Remarks:
\begin{enumerate}
\item If we fix $S\subset \NN$ and define $k(S') = 1$ if $S'=S$ and $k(S')=0$ otherwise, then statement (i) of Lemma\ref{lem:double_stuff} reduces to \eqref{eq:convx_comb_2_samplings}.
\item All samplings arise as a combination of \emph{elementary} samplings, i.e., samplings whose all weight is on one set only. Indeed, let $\Srv$ be an arbitrary sampling. For all subsets $S_j$ of $\NN$ define  $\Srv_j$  by $\Prob(\Srv_j = S_j) = 1$ and let $q_j = \Prob(\Srv = S_j)$. Then clearly, $\Srv = \sum_j q_j \Srv_j$.
\item All doubly uniform samplings arise as convex combinations of nice samplings.
\end{enumerate}

Often it is easier to establish ESO for a simple class of samplings (e.g., nice samplings) and then use it to obtain an ESO for a more complicated class (e.g., doubly uniform samplings as they arise as convex combinations of nice samplings). The following result is helpful in this regard.

\begin{theorem}[Convex Combination of Uniform Samplings]\label{thm:convex_comb_samplings} Let $\Srv_1,\dots, \Srv_m$ be uniform samplings satisfying $(f,\Srv_j)\sim ESO(\beta_j,w_j)$ and  let $q \in \R^m$ be a probability vector. If $\sum_j q_j \Srv_j$ is not nil, then
\[\left(f, \sum_{j=1}^m q_j \Srv_j\right) \sim ESO\left(\frac{1}{\sum_{j=1}^m q_j \E[|\Srv_j|]},\sum_{j=1}^m  q_j \E[|\Srv_j|]\beta_j w_j\right).\]
\end{theorem}

\begin{proof}
First note that from part (iv) of Lemma~\ref{lem:double_stuff} we know that $\hat{S}\eqdef \sum_j q_j \hat{S}_j$ is uniform and hence it makes sense to speak about ESO involving this sampling. Next, we can write
\begin{eqnarray*}
\E\left[f(x+h_{[\hat{S}]})\right]
&=&  \sum_{S}\Prob(\Srv =S)f(x+h_{[S]}) \overset{\eqref{eq:convx_comb_2_samplings}}{=} \sum_S \sum_{j} q_j \Prob(\Srv_j = S)f(x+h_{[S]})\\
&=&  \sum_{j} q_j \sum_S \Prob(\Srv_j = S)f(x+h_{[S]}) =  \sum_{j} q_j \E\left[f(x+h_{[\Srv_j]})\right].
\end{eqnarray*}
It now remains to use \eqref{eq:general_form_for_expectation} and  part (ii) of Lemma~\ref{lem:double_stuff}:
\begin{eqnarray*}
 \sum_{j=1}^m q_j \E\left[f(x+h_{[\Srv_j]})\right] &\overset{\eqref{eq:general_form_for_expectation}}{\leq} & \sum_{j=1}^m q_j \left(f(x)+ \tfrac{\E[|\Srv_j|]}{n} \left( \ve{\nabla f(x)}{h}  + \tfrac{\beta_j}{2}\|h\|_{w_j}^2 \right)\right)\\
&=& f(x) + \tfrac{\sum_j q_j \E[|\Srv_j|]}{n} \ve{\nabla f(x)}{h}
 +  \tfrac{1}{2n}\sum_{j} q_j \E[|\Srv_j|]\beta_j \|h\|_{w_j}^2 \\
&\overset{(\text{Lemma}~\ref{lem:double_stuff} \text{ (ii)})}{=}& f(x) + \tfrac{\E[|\Srv|]}{n} \left(\ve{\nabla f(x)}{h}
 +  \tfrac{\sum_{j} q_j \E[|\Srv_j|]\beta_j \|h\|_{w_j}^2}{2 \sum_j q_j \E[|\Srv_j|]} \right)\\
&=& f(x) + \tfrac{\E[|\Srv|]}{n} \left(\ve{\nabla f(x)}{h}
 +  \tfrac{1}{2\sum_j q_j \E[|\Srv_j|]}\|h\|_{w}^2\right),
\end{eqnarray*}
where $w =  \sum_{j} q_j \E[|\Srv_j|]\beta_j w_j$. In the third step we have also used the fact that $\E[|\Srv|]>0$ which follows from the assumption that $\Srv$ is not nil.
\end{proof}

\subsection{ESO for a conic combination of functions}

 We now establish an ESO for a conic combination of functions each of which is already equipped with an ESO. It offers a complementary result to Theorem~\ref{thm:convex_comb_samplings}.


\begin{theorem}[Conic Combination of Functions] If $(f_j,\Srv) \sim ESO(\beta_j,w_j)$ for $j=1,\dots,m$, then for any $c_1,\dots,c_m \geq 0$ we have
\[\left(\sum_{j=1}^m c_j f_j,\Srv\right) \sim ESO\left(1, \sum_{j=1}^m c_j \beta_j w_j\right).\]
\end{theorem}

\noindent {\em Proof.}
Letting $f=\sum_j c_j f_j$, we get
\begin{eqnarray*}\Exp\left[\sum_j c_j f_j\left(x+\vsubset{h}{\Srv}\right)\right]
&= &   \sum_j c_j \Exp\left[ f_j\left(x+\vsubset{h}{\Srv}\right)\right]\\
&\leq& \sum_j c_j \left(f_j(x)+ \tfrac{\E[|\Srv|]}{n} \left( \ve{\nabla f_j(x)}{h}  + \tfrac{\beta_j}{2}\|h\|_{w_j}^2 \right)\right)\\
&=& \sum_j c_j f_j(x)+ \tfrac{\E[|\Srv|]}{n} \left( \sum_j c_j \ve{\nabla f_j(x)}{h}  + \sum_j  \tfrac{c_j \beta_j}{2}\|h\|_{w_j}^2 \right)\\
&=& f(x) + \tfrac{\E[|\Srv|]}{n} \left( \ve{\nabla f(x)}{h} + \tfrac{1}{2}\|h\|_{\sum_j c_j \beta_j w_j}^2 \right). \qquad \qed
\end{eqnarray*}


\section{Expected Separable Overapproximation (ESO) of Partially Separable Functions}\label{SEC:ESO_for_PS_functions}

Here we derive ESO inequalities for partially separable smooth functions $f$ and (proper) uniform (Section~\ref{sec:uniform}), nonoverlapping uniform (Section~\ref{sec:NOuniform}), nice (Section~\ref{sec:nice}) and doubly uniform (Section~\ref{sec:DUsamplings}) samplings.

\subsection{Uniform samplings}\label{sec:uniform}

Consider an arbitrary proper sampling $\Srv$ and let $\nu = (\nu_1,\dots,\nu_n)^T$ be defined by
\[\nu_i \eqdef \Exp\left[\min\{\omega,|\Srv|\} \;|\; i \in \Srv \right] = \tfrac{1}{p_i} \sum_{S: i \in S} \pp(S) \min\{ \omega , |S| \}, \qquad i \in \NN.\]


\begin{lemma} If $\Srv$ is proper, then
\begin{equation}\label{eq:ESO_987987}\Exp\left[f(x+h_{[\Srv]})\right] \leq f(x) + \ve{\nabla f(x)}{h}_p +  \tfrac{1}{2}  \|h\|_{p \odot \nu \odot \Lip}^2.\end{equation}
\end{lemma}

\noindent {\em Proof.}
Let us use Theorem~\ref{thm:globalLipConstant} with $h$ replaced by $h_{[\Srv]}$. Note that
$\max_{J\in \JJ} |J\cap \support(h_{[\Srv]}) | \leq \max_{J \in \JJ} |J \cap \Srv| \leq \min\{\omega, |\Srv|\}$. Taking expectations of both sides of \eqref{eq:upperBoundOnLipConst} we therefore get
\begin{eqnarray}
\Exp\left[f(x+h_{[\Srv]})\right] & \overset{\eqref{eq:upperBoundOnLipConst}}{\leq} & f(x) + \Exp\left[\ve{\nabla f(x)}{h_{[\Srv]}}\right] + \tfrac{1}{2} \Exp \left[ \min\{\omega,|\Srv|\}\|h_{[\Srv]}\|_{\Lip}^2 \right]\notag\\
& \overset{\eqref{eq:simple3}}{=}&  f(x) + \ve{\nabla f(x)}{h}_p + \tfrac{1}{2} \Exp \left[ \min\{\omega,|\Srv|\}\|h_{[\Srv]}\|_{\Lip}^2 \right].\label{eq:oinas593}
\end{eqnarray}
It remains to bound the last term in the expression above. Letting $\theta_i = \Lip_i \|h^{(i)}\|_{(i)}^2$, we have
\begin{align}\Exp & \left[ \min\{\omega,|\Srv|\}\|h_{[\Srv]}\|_{\Lip}^2 \right] = \Exp \left[  \sum_{i\in \Srv} \min\{\omega,|\Srv|\} \Lip_i \|h^{(i)}\|_{(i)}^2 \right] = \sum_{S \subset \NN} \pp(S) \sum_{i\in S} \min\{\omega,|S|\} \theta_i\notag\\
& = \sum_{i=1}^n \theta_i \sum_{S : i \in S} \min\{\omega,|S|\} \pp(S) = \sum_{i=1}^n \theta_i p_i \Exp\left[\min\{\omega,|\Srv|\} \;|\; i\in \Srv\right] =\sum_{i=1}^n \theta_i p_i \nu_i = \|h\|_{p \odot \nu \odot \Lip}^2. \qquad \qed \label{eq:s5d9nsh}
\end{align}



The above lemma will now be used to establish ESO for arbitrary (proper) uniform samplings.

\begin{theorem}\label{thm:proper_uniform_ESO}
If  $\Srv$ is proper and uniform, then
\begin{equation}\label{eq:ESO_uniform}(f,\Srv) \sim ESO(1, \nu \odot \Lip).\end{equation}
If, in addition,  $\Prob(|\Srv|=\tau)=1$ (we say that $\Srv$ is $\tau$-uniform), then
\begin{equation}\label{eq:ESO_uniform_reliable}(f,\Srv) \sim ESO(\min\{\omega,\tau\},\Lip).\end{equation}
Moreover, ESO \eqref{eq:ESO_uniform_reliable} is monotonic.
\end{theorem}
\begin{proof}First, \eqref{eq:ESO_uniform} follows from \eqref{eq:ESO_987987} since for a uniform sampling one has $p_i=\E[|\Srv|]/n$ for all $i$. If  $\Prob(|\Srv|=\tau)=1$, we get $\nu_i=\min\{\omega,\tau\}$ for all $i$; \eqref{eq:ESO_uniform_reliable} therefore follows from \eqref{eq:ESO_uniform}. Let us now establish monotonicity. Using the deterministic separable overapproximation \eqref{eq:upperBoundOnLipConst} with $h=h_{[\Srv]}$,
\begin{align}\notag F(x+ h_{[\Srv]}) &\leq f(x) + \ve{\nabla f(x)}{h_{[\Srv]}} + \max_{J \in \JJ}\tfrac{|J \cap \Srv|}{2}\|h_{[\Srv]}\|_{\Lip}^2 + \cPsi(x+h_{[\Srv]})\\
\label{eq:1234567890}& \leq f(x) + \ve{\nabla f(x)}{h_{[\Srv]}} + \tfrac{\beta}{2}\|h_{[\Srv]}\|_{w}^2 + \cPsi(x+h_{[\Srv]})\\
\label{eq:98jupper} &= f(x) + \sum_{i\in \Srv} \underbrace{\left(\ve{\nabla f(x)}{U_i h^{(i)}} + \tfrac{\beta w_i}{2}\|h^{(i)}\|^2_{(i)} + \cPsi_i(x^{(i)}+h^{(i)})\right)}_{\eqdef \kappa_i(h^{(i)})} + \sum_{i \notin \Srv} \cPsi_i(x^{(i)}).\end{align}
Now let $h(x)=\arg \min_h H_{\beta,w}(x,h)$ and recall that
\begin{align*}H_{\beta,w}(x,h) &= f(x) + \ve{\nabla f(x)}{h} + \tfrac{\beta}{2}\|h\|_w^2 + \cPsi(x+h)\\
&= f(x) + \sum_{i=1}^n \left(\ve{\nabla f(x)}{U_i h^{(i)}} + \tfrac{\beta w_i}{2}\|h^{(i)}\|_{(i)}^2 + \cPsi_i(x^{(i)}+h^{(i)})\right) = f(x) + \sum_{i=1}^n \kappa_i(h^{(i)}).
\end{align*}
So, by definition, $(h(x))^{(i)}$ minimizes $\kappa_i(t)$ and hence, $(h(x))_{[\Srv]}$ (recall \eqref{eq:lllop09}) minimizes the upper bound \eqref{eq:98jupper}. In particular, $(h(x))_{[\Srv]}$ is better than a nil update, which immediately gives $F(x+(h(x))_{[\Srv]}) \leq f(x) + \sum_{i \in \Srv} \kappa_i (0) + \sum_{i \notin \Srv} \cPsi_i(x^{(i)}) = F(x)$.
\end{proof}

Besides establishing an ESO result, we have just shown that, in the case of $\tau$-uniform samplings with a conservative estimate for $\beta$, PCDM1 is monotonic, i.e.,  $F(x_{k+1})\leq F(x_k)$. In particular, PCDM1 and PCDM2 coincide. We call the estimate $\beta = \min\{\omega,\tau\}$ ``conservative'' because it can be improved (made smaller) in special cases; e.g., for the $\tau$-nice sampling. Indeed, Theorem~\ref{thm:ESO-nice} establishes an ESO for the $\tau$-nice sampling with the same $w$ ($w=\Lip$), but with $\beta = 1 + \tfrac{(\omega-1)(\tau-1)}{n-1}$, which is better (and can be much better than)  $\min\{\omega,\tau\}$. Other things equal, smaller $\beta$ directly translates into better complexity. The price for the small $\beta$ in the case of the $\tau$-nice sampling is the loss of monotonicity. This is not a problem for strongly convex objective, but for merely convex objective this is an issue as the analysis techniques we developed are only applicable to the monotonic method PCDM2 (see Theorem~\ref{thm:complexity-convex-case}).


\subsection{Nonoverlapping uniform  samplings} \label{sec:NOuniform}

Let $\Srv$ be a (proper) nonoverlapping uniform sampling as defined in \eqref{eq:non-overlap-uniform}. If $i\in S^j$, for some $j \in \{1,2,\dots,l\}$, define
\begin{equation}\label{eq:gamma}\gamma_i \eqdef \max_{J\in \JJ} |J \cap S^j|,\end{equation}
and let $\gamma = (\gamma_1,\dots,\gamma_n)^T$.

Note that, for example, if $\Srv$ is the serial uniform sampling, then $l=n$ and $S^j=\{j\}$ for $j=1,2,\dots,l$, whence $\gamma_i = 1$ for all $i\in \NN$. For the fully parallel sampling we have $l=1$ and $S^1 = \{1,2,\dots,n\}$, whence $\gamma_i = \omega$ for all $i\in \NN$.

\begin{theorem} \label{thm:ESO-nonoverlapping}If $\Srv$ a nonoverlapping  uniform sampling, then
\begin{equation}\label{eq:ESO_general}(f,\Srv) \sim ESO(1,\gamma \odot \Lip).\end{equation}
Moreover, this ESO is monotonic.
\end{theorem}
\begin{proof} By Theorem~\ref{thm:globalLipConstant}, used  with $h$ replaced by $h_{[S^j]}$ for $j=1,2,\dots,l$, we get
\begin{equation}\label{eq:gh77}f(x+h_{[S^j]}) \leq f(x) + \ve{\nabla f(x)}{h_{[S^j]}} + \max_{J \in \JJ}\tfrac{|J \cap S^j|}{2}\|h_{[S^j]}\|_L^2.\end{equation}
Since $\Srv=S^j$ with probability $\tfrac{1}{l}$,
\begin{eqnarray*}
\Exp\left[f(x+h_{[\Srv]})\right] & \overset{\eqref{eq:gh77}}{\leq} & \tfrac{1}{l}\sum_{j=1}^l \left( f(x) +  \ve{\nabla f(x)}{h_{[S^j]}} +  \max_{J \in \JJ}\tfrac{ |J \cap S^j|}{2}  \|h_{[S^j]}\|_{\Lip}^2\right) \notag\\
&\overset{\eqref{eq:gamma}}{=} &  f(x) + \tfrac{1}{l}\left(\ve{\nabla f(x)}{h} + \tfrac{1}{2} \sum_{j=1}^l   \sum_{i\in S^j}\gamma_i \Lip_i \|h^{(i)}\|_{(i)}^2\right)\\
& =&  f(x) + \tfrac{1}{l}\left(\ve{\nabla f(x)}{h} + \tfrac{1}{2} \|h\|_{\gamma \odot \Lip}^2\right),
\end{eqnarray*}
which establishes \eqref{eq:ESO_general}. It now only remains to establish monotonicity. Adding $\cPsi(x+h_{[\Srv]})$ to \eqref{eq:gh77} with $S^j$ replaced by $\Srv$, we get $F(x+ h_{[\Srv]}) \leq f(x) + \ve{\nabla f(x)}{h_{[\Srv]}} + \tfrac{\beta}{2}\|h_{[\Srv]}\|_{w}^2 + \cPsi(x+h_{[\Srv]})$.
From this point on the proof is identical to that in Theorem~\ref{thm:proper_uniform_ESO}, following equation \eqref{eq:1234567890}.
\end{proof}

\subsection{Nice samplings}\label{sec:nice}

In this section we establish an ESO for nice samplings.

\begin{theorem}\label{thm:ESO-nice} If $\Srv$ is the $\tau$-nice sampling and $\tau\neq 0$, then
\begin{equation}\label{eq:inequalityToGetZeta}(f,\Srv) \sim ESO \left( 1+
      \frac{ (\omega-1)(\tau-1)}{\max(1,n-1)}, L\right).
\end{equation}
\end{theorem}

\noindent {\em Proof.}
Let us fix $x$ and define $\phi$ and $\phi^J$ as in the proof of Theorem~\ref{thm:globalLipConstant}. Since
\begin{eqnarray*}
 \Exp \left[\phi(\vsubset{h}{\Srv})\right]
=
 \Exp \left[f(x+\vsubset{h}{\Srv})-f(x)- \ve{\nabla f(x)}{\vsubset{h}{\Srv}} \right] \stackrel{\eqref{eq:simple3unif}}{=}
 \Exp\left[f(x+\vsubset{h}{\Srv})\right]
       -f(x)-\tfrac{\tau}{n} \ve{\nabla f(x)}{h},
\end{eqnarray*}
it now only remains to show that

\begin{equation}\label{eq:aksoio4209809878}
 \Exp \left[\phi(\vsubset{h}{\Srv})\right]\leq
 \tfrac{\tau}{2 n}
 \left(1 + \tfrac{ (\omega-1) ( \tau-1 )  }{ \max(1,n-1) } \right)\|h\|_L^2.
\end{equation}
Let us now adopt the convention that  expectation conditional on an event which happens with probability 0 is equal to 0. Letting $\eta_J \eqdef |J\cap \Srv|$, and using this convention, we can write
\begin{eqnarray}
 \Exp\left[\phi(\vsubset{h}{\Srv})\right]
 =
     \sum_{J\in\JJ} \Exp \left[\phi^J(\vsubset{h}{\Srv})\right] &=&    \sum_{k=0}^n
     \sum_{J\in\JJ} \Exp \left[\phi^J(\vsubset{h}{\Srv}) \;|\; \eta_J=k\right] \Prob(\eta_J = k)\notag\\
&=&
    \sum_{k=0}^n \Prob(\eta_J = k)
     \sum_{J\in\JJ} \Exp \left[\phi^J(\vsubset{h}{\Srv}) \;|\; \eta_J=k\right].\label{eq:8893298d9}
\end{eqnarray}
Note that the last identity follows if we assume, without loss of generality, that all sets $J$ have the same cardinality $\omega$ (this can be achieved by introducing ``dummy'' dependencies). Indeed, in such a case $\Prob(\eta_J = k)$ does not depend on $J$. Now, for any $k\geq 1$ for which $\Prob(\eta_J =k)>0$ (for some $J$ and hence for all), using convexity of $\phi^J$, we can now estimate
\begin{eqnarray}
\Exp\left[\phi^J(\vsubset{h}{\Srv}) \;|\; \eta_J = k\right]
&=&
  \Exp
    \left[\left.\phi^J \left(\tfrac{1}{k} \sum_{i \in J\cap \Srv} k \U_i \vc{h}{i} \right) \right. \;|\; \eta_J = k\right] \notag \\
&\leq&
  \Exp
    \left[\left.
      \tfrac{1}{k}  \sum_{i \in J\cap \Srv}
          \phi^J \left(  k \U_i \vc{h}{i} \right) \right. \;|\; \eta_J=k\right] \overset{\eqref{eq:JEJEJEJE}}{=} \tfrac{1}{\omega}  \sum_{i \in J}
          \phi^J \left(  k \U_i \vc{h}{i} \right).\label{eq:87683893}
\end{eqnarray}

If we now sum the inequalities \eqref{eq:87683893} for all $J\in \JJ$, we get
\begin{eqnarray}
\sum_{J\in\JJ}\Exp \left[\phi^J(\vsubset{h}{\Srv}) \;|\; \eta_J = k\right]
&\stackrel{\eqref{eq:87683893}}{\leq}&  \tfrac{1}{\omega}
    \sum_{J\in\JJ} \sum_{i \in J}
          \phi^J \left(  k \U_i \vc{h}{i} \right)
=  \tfrac{1}{\omega}
    \sum_{J\in\JJ}     \sum_{i=1}^n
          \phi^J \left(  k \U_i \vc{h}{i} \right) \notag
\\
&=&  \tfrac{1}{\omega}
    \sum_{i=1}^n \sum_{J\in\JJ}
          \phi^J \left( k \U_i \vc{h}{i} \right)
=  \tfrac{1}{\omega}
    \sum_{i=1}^n
          \phi \left(  k \U_i \vc{h}{i} \right)\notag
\\ &\stackrel{\eqref{eq:987980980}}{\leq}&   \tfrac{1}{\omega}\sum_{i=1}^n \tfrac{\Lip_i}{2}  \|kh^{(i)}\|_{(i)}^2 = \tfrac{k^2}{2\omega}  \|h\|_L^2.\label{eq:8488dd8}
\end{eqnarray}
Finally, \eqref{eq:aksoio4209809878} follows after plugging \eqref{eq:8488dd8} into \eqref{eq:8893298d9}:
\[
 \Exp \left[\phi(\vsubset{h}{\Srv})\right]
\leq \sum_{k} \Prob(\eta_J = k) \tfrac{k^2}{2\omega} \|h\|_L^2
= \tfrac{1}{2\omega} \|h\|_L^2 \E [|J\cap \Srv|^2]
\stackrel{\eqref{eq:00sjs738}}{=}  \tfrac{\tau}{2 n}
 \left(1 + \tfrac{ (\omega-1)( \tau-1 )   }{\max(1, n-1) } \right)\|h\|_L^2. \qquad \qed
\]

\subsection{Doubly uniform samplings} \label{sec:DUsamplings}

We are now ready, using a bootstrapping argument, to formulate and prove a result covering all doubly uniform samplings.

\begin{theorem}\label{thm:ESO-DU} If  $\Srv$ is a (proper) doubly uniform sampling, then 
\begin{align}\label{eq:inequalitys8sjs8sjs8}
 (f,\Srv) \sim ESO\left(1+
      \frac{ (\omega-1)\left(\frac{\Exp[|\Srv|^2]}{\Exp[|\Srv|]}-1\right)}{\max(1,n-1)},L\right).
\end{align}

\end{theorem}
\begin{proof} Letting $q_k = \Prob(|\Srv| = k)$ and $d = \max\{1,n-1\}$, we have
\begin{eqnarray*}\label{eq:inequality090u9dn8}
 \Exp \left[f(x+\vsubset{h}{\Srv})\right]
&=&\Exp\left[ \Exp \left[f(x+\vsubset{h}{\Srv}) \;|\; |\Srv|  \right]\right] = \sum_{k=0}^n q_k \Exp \left[f(x+\vsubset{h}{\Srv}) \;|\; |\Srv| = k \right]\\
&\stackrel{\eqref{eq:inequalityToGetZeta}}{\leq}& \sum_{k=0}^n q_k \left[ f(x)+ \tfrac{k}{n} \left( \ve{\nabla f(x)}{ h}
 + \tfrac{1}{2}\left(1+
      \tfrac{ (\omega-1)(k-1)}{d}
  \right)\|h\|_L^2 \right)\right]\\
  &=& f(x) + \tfrac{1}{n}\sum_{k=0}^n q_k k \ve{\nabla f(x)}{h}  + \tfrac{1}{2n}\sum_{k=1}^n q_k \left[k\left(1-\tfrac{\omega-1}{d}\right) + k^2 \tfrac{\omega-1}{d}
 \right]\|h\|_L^2\\
&=& f(x) + \tfrac{\Exp[|\Srv|]}{n} \ve{\nabla f(x)}{h} + \tfrac{1}{2n}\left(\Exp [|\Srv|]\left(1-\tfrac{\omega-1}{d}\right) +  \Exp [|\Srv|^2] \tfrac{\omega-1}{d} \right)\|h\|_L^2.
\end{eqnarray*}
This theorem could have alternatively been proved by writing $\Srv$ as a convex combination of nice samplings and applying Theorem~\ref{thm:convex_comb_samplings}.
\end{proof}

\noindent Note that Theorem~\ref{thm:ESO-DU} reduces to that of Theorem~\ref{thm:ESO-nice} in the special case of a nice sampling, and gives the same result as Theorem~\ref{thm:ESO-nonoverlapping} in the case of the serial and fully parallel samplings.



\section{Iteration Complexity} \label{SEC:Iteration_Complexity}


In this section we prove two iteration complexity theorems\footnote{The development is similar to that in \cite{RT:UCDC} for the \emph{serial} block coordinate descent method, in the composite case. However, the results are vastly different.}. The first result (Theorem~\ref{thm:complexity-convex-case}) is for non-strongly-convex $F$ and covers PCDM2 with no restrictions  and PCDM1 only in the case when a  monotonic ESO is used. The second result (Theorem~\ref{thm:complexity-strongly-convex-case})  is for strongly convex $F$ and covers PCDM1 without any monotonicity restrictions.

Let us first establish two auxiliary results.

\begin{lemma}\label{lem:GGG67} For all $x\in \dom F$, $H_{\beta,w}(x,h(x)) \leq \min_{y\in \R^N} \{F(y) + \tfrac{\beta-\mu_f(w)}{2}\|y-x\|_w^2\}$.
\end{lemma}

\noindent {\em Proof.}
\begin{eqnarray}
\notag H_{\beta,w}(x,h(x))  \stackrel{\eqref{eq:h(x)}}{=}  \min_{y\in\R^{\N}} H_{\beta,w}(x,y-x)  &=&  \min_{y\in \R^{\N}} f(x)+ \ve{\nabla f(x)}{y-x} + \cPsi(y)+\tfrac{\beta}{2} \|y-x\|_w^2\\
\notag & \stackrel{\eqref{eq:strong_def}}{\leq} & \min_{y\in \R^{\N}}   f(y) -  \tfrac{\mu_f(w)}{2}\|y-x\|_w^2 + \cPsi(y)+\tfrac{\beta}{2} \|y-x\|_w^2. \hfill \qed
\end{eqnarray}

\begin{lemma}\label{lem:H-Fstar}
\begin{itemize}
\item[(i)] Let $x^*$ be an optimal solution of \eqref{eq:P}, $x\in \dom F$ and let $R = \|x-x^*\|_w$. Then
\begin{equation}\label{eq:nonsmooth:gamma1}
 H_{\beta,w}(x,h(x)) - F^* \leq \begin{cases} \left(1-\tfrac{F(x)-F^*}{2\beta R^2}\right)(F(x)-F^*), \quad & \text{if } F(x)-F^*\leq \beta R^2,\\
\tfrac{1}{2} \beta R^2 < \tfrac{1}{2}(F(x)-F^*), \quad & \text{otherwise.}
\end{cases}
\end{equation}
\item[(ii)] If $\mu_f(w) + \mu_\cPsi(w) > 0$ and $\beta\geq \mu_f(w)$, then for all $x\in \dom F$,
\begin{equation}\label{eq:nonsmooth:gammaxiUsage}
 H_{\beta,w}(x,h(x)) - F^* \leq
      \frac{\beta-\mu_f(w)}{\beta+\mu_\cPsi(w)}
 (F(x)-F^*).
\end{equation}
\end{itemize}
\end{lemma}

\begin{proof} Part (i): Since we do not assume strong convexity, we have $\mu_f(w) = 0$, and hence
\begin{eqnarray*}
\notag H_{\beta,w}(x,h(x)) &\overset{(\text{Lemma~\ref{lem:GGG67}})}{\leq} & \min_{y\in \R^{\N}}  \{ F(y) + \tfrac{\beta}{2} \|y-x\|_w^2\}\\
\notag &\leq & \min_{\lambda \in [0,1]}   \{F(\lambda x^* + (1-\lambda)x) + \tfrac{\beta \lambda^2}{2} \|x-x^*\|_w^2\}\\
\label{eq:R2} &\leq & \min_{\lambda \in [0,1]}\{ F(x)-\lambda (F(x)-F^*)+ \tfrac{\beta \lambda^2}{2} R^2\}.
\end{eqnarray*}
Minimizing the last expression in $\lambda$ gives $\lambda^* = \min\left\{1,(F(x)-F^*)/({\beta R^2})\right\}$; the result follows.
Part (ii): Letting $\mu_f = \mu_f(w)$, $\mu_\cPsi = \mu_\cPsi(w)$ and $\lambda^* = (\mu_f+\mu_\cPsi)/(\beta+\mu_\cPsi)\leq 1$, we have
\begin{eqnarray}
\notag H_{\beta,w}(x,h(x))  & \overset{(\text{Lemma~\ref{lem:GGG67}})}{\leq} & \min_{y\in \R^{\N}}   \{F(y) + \tfrac{\beta-\mu_f}{2} \|y-x\|_w^2\}\\
\notag & \leq & \min_{\lambda \in [0,1]} \{  F(\lambda x^* + (1-\lambda)x) + \tfrac{(\beta-\mu_f)\lambda^2}{2} \|x-x^*\|_w^2\}\\
\notag & \overset{\eqref{eq:strong_2}+\eqref{eq:mu_F}}{\leq} & \min_{\lambda \in [0,1]} \{\lambda F^* + (1-\lambda) F(x) - \tfrac{(\mu_f + \mu_\cPsi)\lambda(1-\lambda)-(\beta-\mu_f)\lambda^2}{2}\|x-x^*\|_w^2\}\\
\label{eq:asuv1398f0a0sdfa}
\notag & \leq & F(x) - \lambda^*(F(x)-F^*).
\end{eqnarray}
The last inequality follows from the identity $(\mu_f+\mu_\cPsi)(1-\lambda^*) - (\beta-\mu_f)\lambda^* = 0$.
\end{proof}

We could have formulated part (ii) of the above result using the weaker assumption $\mu_F(w)>0$, leading to a slightly stronger result. However, we prefer the above treatment as it gives more insight.

\subsection{Iteration complexity: convex case}\label{sec:Complexity-convex}

The following lemma will be used to finish off the proof of the complexity result of this section.

\begin{lemma} [Theorem 1 in \cite{RT:UCDC}] \label{l:randomVariableTrick}
Fix $x_0\in \R^N$ and let $\{x_k\}_{k\geq 0}$ be a sequence of random vectors in $\R^N$ with $x_{k+1}$ depending on $x_k$ only. Let $\phi:\R^N\to \R$ be a nonnegative function and define $\xi_k = \phi(x_k)$. Lastly, choose accuracy level $0<\epsilon<\xi_0$, confidence level $0 < \rho< 1$,  and assume that the sequence of random variables $\{\xi_k\}_{k\geq 0}$ is nonincreasing and has one of the following properties:
\begin{enumerate}
\item[(i)] $\E[\xi_{k+1} \;|\; x_k] \leq (1 - \tfrac{\xi_k}{c_1})\xi_k $, for all $k$, where $c_1>\epsilon$ is a constant,
\item[(ii)] $\E[\xi_{k+1} \;|\; x_k] \leq (1-\tfrac{1}{c_2}) \xi_k$, for all $k$ such that $\xi_k\geq \epsilon$, where $c_2>1$ is  a constant.
\end{enumerate}
If property (i) holds and we choose $K \geq 2 + \tfrac{c_1}{\epsilon} (1 - \tfrac{\epsilon}{\xi_0} + \log (\tfrac{1}{\rho}))$, or if property (ii) holds, and we choose
$K\geq c_2 \log (\tfrac{\xi_0}{\epsilon \rho})$, then $\Prob(\xi_K \leq \epsilon) \geq 1-\rho$.
\end{lemma}

This lemma  was recently extended in \cite{TRG:InexactCDM} so as to aid the analysis of a \emph{serial} coordinate descent method with \emph{inexact} updates, i.e., with $h(x)$ chosen as an \emph{approximate} rather than exact minimizer of $H_{1,\Lip}(x,\cdot)$ (see \eqref{eq:h(x)}). While in this paper we deal with exact updates only, the results can be extended to the inexact case.


\begin{theorem} \label{thm:complexity-convex-case}
 Assume that $(f,\Srv) \sim ESO(\beta,w)$, where  $\Srv$ is a proper uniform sampling, and let $\alpha = \tfrac{\E[|\Srv|]}{n}$.  Choose  $x_0\in \dom F$ satisfying
 \begin{equation}\label{eq:R(x_0)}\Rw{w}{x_0,x^*} \eqdef \max_x  \{\|x-x^*\|_w \;:\; F(x) \leq F(x_0)\} < +\infty,\end{equation}
where $x^*$ is an optimal point of \eqref{eq:P}. Further, choose target confidence level $0<\rho<1$, target accuracy level $\epsilon>0$ and iteration counter $K$ in any of the following two ways:
\begin{enumerate}
\item[(i)] $\epsilon<F(x_0)-F^*$  and
\begin{equation}\label{eq:k_composite_nonstrong} K \geq 2 + \frac{2\left(\tfrac{\beta}{\alpha}\right)\max\left\{\Rws{w}{x_0,x^*}, \tfrac{F(x_0)-F^*}{\beta}\right\}}{\epsilon} \left(1  - \frac{\epsilon}{F(x_0)-F^*} + \log \left(\frac{1}{\rho}\right)\right),\end{equation}
\item[(ii)] $\epsilon < \min\{2\left(\tfrac{\beta}{\alpha}\right)\Rws{w}{x_0,x^*}, F(x_0)-F^*\}$ and
\begin{equation}\label{eq:k_composite_nonstrong2} K \geq \frac{2  \left(\tfrac{\beta}{\alpha}\right)\Rws{w}{x_0,x^*}}{\epsilon} \log \left(\frac{F(x_0)-F^*}{\epsilon\rho }\right).\end{equation}
\end{enumerate}
If $\{x_k\}$, $k\geq 0$, are the random iterates of PCDM (use PCDM1 if the ESO is monotonic, otherwise use PCDM2), then
$\Prob(F(x_K)-F^*\leq \epsilon) \geq 1-\rho$.
\end{theorem}

\begin{proof} Since either PCDM2 is used (which is monotonic) or otherwise the ESO is monotonic, we must have
$F(x_k)\leq F(x_0)$ for all $k$. In particular, in view of \eqref{eq:R(x_0)} this implies that $\|x_k-x^*\|_w \leq {\cal R}_w(x_0,x^*)$.
Letting  $\xi_k = F(x_k)-F^*$, we have
\begin{eqnarray}
  \E [\xi_{k+1} \;|\; x_k]
   &\overset{\eqref{eq:1} }{\leq} &    (1-\alpha) \xi_k + \alpha (H_{\beta,w}(x_k,h(x_k))-F^*) \notag \\
   &\overset{\eqref{eq:nonsmooth:gamma1}}{\leq}& (1-\alpha) \xi_k + \alpha \max\left\{1-\frac{\xi_k}{2\beta \|x_k-x^*\|_w^2}, \frac{1}{2}\right\}\xi_k \notag \\
   &=& \max\left\{1-\frac{\alpha \xi_k}{2\beta \|x_k-x^*\|_w^2}, 1 - \frac{\alpha}{2}\right\}\xi_k \notag\\
    &\leq& \max\left\{1-\frac{\alpha \xi_k}{2\beta {\cal R}^2_w(x_0,x^*)}, 1 - \frac{\alpha}{2}\right\}\xi_k.\label{eq:the_one}
\end{eqnarray}
Consider case (i) and let $c_1=2\tfrac{\beta}{\alpha}\max\{{\cal R}^2_w(x_0,x^*), \tfrac{\xi_0}{\beta}\}$. Continuing with \eqref{eq:the_one}, we then get
\[\E [\xi_{k+1} \;|\; x_k]  \leq (1-\tfrac{\xi_k}{c_1})\xi_k\] for all $k \geq 0$. Since $\epsilon<\xi_0 < c_1$, it suffices to apply Lemma~\ref{l:randomVariableTrick}(i). Consider now case (ii) and let $c_2 = 2\tfrac{\beta}{\alpha}\frac{{\cal R}^2_w(x_0,x^*)}{\epsilon}$. Observe now that whenever $\xi_k \geq \epsilon$, from \eqref{eq:the_one} we get
$\E [\xi_{k+1} \;|\; x_k]  \leq (1-\tfrac{1}{c_2})\xi_k$. By assumption, $c_2 > 1$, and hence it remains to apply Lemma~\ref{l:randomVariableTrick}(ii).
\end{proof}

The important message of the above theorem is that the iteration complexity of our methods in the convex case is $O(\tfrac{\beta}{\alpha}\tfrac{1}{\epsilon})$. Note that for the serial method (PCDM1 used with $\Srv$ being the serial sampling) we have $\alpha = \tfrac{1}{n}$ and $\beta = 1$ (see Table~\ref{tbl:ESO}), and hence $\tfrac{\beta}{\alpha} = n$. It will be interesting to study the \emph{parallelization speedup factor} defined by
\begin{equation}\label{eq:speedup_factor_C}\text{parallelization speedup factor} = \frac{\tfrac{\beta}{\alpha} \text { of the serial method}}{\tfrac{\beta}{\alpha} \text { of a parallel method}} = \frac{n}{\tfrac{\beta}{\alpha} \text { of a parallel method}}.\end{equation}

Table~\ref{tbl:upperbounds}, computed from the data in Table~\ref{tbl:ESO}, gives expressions for the parallelization speedup factors for PCDM based on a DU sampling (expressions for 4 special cases are given as well).

\begin{table}[ht!]
\centering
\begin{tabular}{|c|c|}
\hline
$\Srv$   & Parallelization speedup factor\\
  \hline
\phantom{-}  &        \\
doubly uniform   & $\frac{\E[|\Srv|]}{1 + \tfrac{(\omega-1)\left((\E[|\Srv|^2]/\E[|\Srv|])-1\right)}{\max(1,n-1)}}$ \\
\phantom{-}  &        \\
\hline
\phantom{-}  &        \\
$(\tau,p_b)$-binomial & $\frac{\tau}{\tfrac{1}{p_b}+ \tfrac{(\omega-1)(\tau-1)}{\max(1,n-1)}}$   \\
\phantom{-}  &       \\
$\tau$-nice & $\frac{\tau}{1+ \tfrac{(\omega-1)(\tau-1)}{\max(1,n-1)}}$   \\
\phantom{-}  &        \\
fully parallel & $\frac{n}{\omega}$   \\
\phantom{-}  &        \\
serial   & $1$  \\
\phantom{-}  &        \\
\hline
\end{tabular}
 \caption{Convex $F$: Parallelization speedup factors for DU samplings. The factors below the line are special cases of the general expression. Maximum speedup is naturally obtained by the fully parallel sampling: $\tfrac{n}{\omega}$.}
\label{tbl:upperbounds}
\end{table}

The speedup of the serial sampling (i.e., of the algorithm based on it) is 1 as we are comparing it to itself. On the other end of the spectrum is the fully parallel sampling with a speedup of $\tfrac{n}{\omega}$. If the degree of partial separability is small, then this factor will be high --- especially so if $n$ is huge, which is the domain we are interested in. This provides an affirmative answer to the research question stated in italics in the introduction.

Let us now look at the speedup factor in the case of a $\tau$-nice sampling. Letting $r= \tfrac{\omega-1}{\max(1,n-1)} \in [0,1]$ (degree of partial separability normalized), the speedup factor can be written as
\[s(r) = \frac{\tau}{1+ r(\tau-1)}.\]
Note that as long as $r\leq \tfrac{k-1}{\tau-1}\approx \tfrac{k}{\tau}$, the speedup factor will be at least $\tfrac{\tau}{k}$. Also note that $\max\{1,\tfrac{\tau}{\omega}\} \leq s(r)\leq \min\{\tau, \tfrac{n}{\omega}\}$. Finally, if a speedup of at least $s$ is desired, where $s\in [0,\tfrac{n}{\omega}]$, one needs to use at least $\frac{1-r}{1/s-r}$ processors. For illustration, in Figure~\ref{fig:speedup_xx} we plotted $s(r)$ for a few values of $\tau$. Note that for small values of $\tau$, the speedup is significant and can be as large as the number of processors (in the separable case). We wish to stress that in many applications $\omega$ will be a constant independent of $n$, which means that $r$ will indeed be very small in the huge-scale optimization setting.

\begin{figure}[!ht]
 \centering
\includegraphics[width=4in]{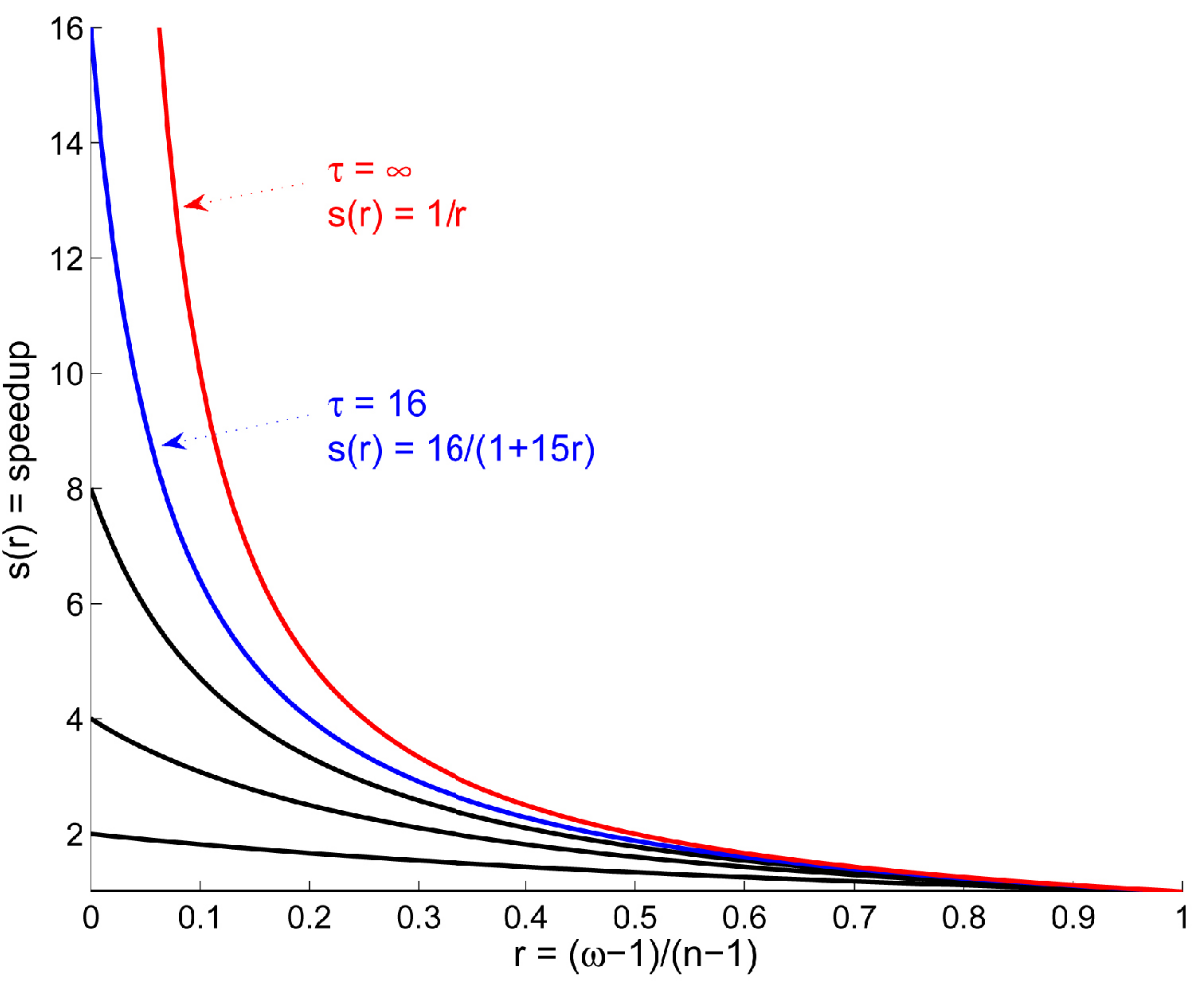}
 \caption{Parallelization speedup factor of PCDM1/PCDM2 used with $\tau$-nice sampling as a function of the normalized/relative degree of partial separability $r$.}
 \label{fig:speedup_xx}
\end{figure}

\subsection{Iteration complexity: strongly convex case}  \label{sec:composite_strong}

In this section we assume that $F$ is strongly convex with respect to the norm $\|\cdot\|_w$ and show that $F(x_k)$ converges to $F^*$ linearly, with high probability.

\begin{theorem} \label{thm:complexity-strongly-convex-case}
Assume $F$ is strongly convex with $\mu_f(w)+\mu_\cPsi(w)>0$. Further, assume $(f,\Srv) \sim ESO(\beta,w)$, where  $\Srv$ is a proper uniform sampling and let $\alpha = \tfrac{\E[|\Srv|]}{n}$.  Choose initial point $x_0\in \dom F$, target confidence level $0<\rho<1$, target accuracy level $0<\epsilon<F(x_0)-F^*$ and
\begin{equation}\label{eq:k_uniform_strong}
 K\geq
  \frac{1}{\alpha} \frac{\beta+\mu_\cPsi(w)}{\mu_f(w)+\mu_\cPsi(w)} \log \left(\frac{F(x_0)-F^*}{\epsilon\rho}\right).
\end{equation}
If $\{x_k\}$ are the random points generated by PCDM1 or PCDM2, then
$\Prob(F(x_K)-F^*\leq \epsilon) \geq 1-\rho$.
\end{theorem}

\noindent {\em Proof.}
Letting $\xi_k = F(x_k)-F^*$, we have
\[
  \E [\xi_{k+1} \;|\; x_k]
   \overset{\eqref{eq:1} }{\leq}    (1-\alpha) \xi_k + \alpha (H_{\beta,w}(x_k,h(x_k))-F^*)   \overset{\eqref{eq:nonsmooth:gammaxiUsage}}{\leq}   \left(1-\alpha \tfrac{\mu_f(w)+\mu_\cPsi(w)}{\beta+\mu_\cPsi(w)} \right) \xi_k \eqdef (1-\gamma)\xi_k.
\]
Note that $0 < \gamma \leq 1$ since $0<\alpha\leq 1$ and $\beta \geq \mu_f(w)$ by \eqref{eq:beta-mu}. By taking expectation in $x_k$, we obtain $\E[\xi_k]\leq (1 - \gamma)^k\xi_0$. Finally, it remains to use Markov inequality:
\[\Prob(\xi_K > \epsilon) \leq \frac{\E[\xi_K]}{\epsilon} \leq \frac{(1-\gamma)^K \xi_0}{\epsilon} \overset{\eqref{eq:k_uniform_strong}}{\leq} \rho. \qquad \qed\]

Instead of doing a direct calculation, we could have finished the proof of Theorem~\ref{thm:complexity-strongly-convex-case} by applying Lemma~\ref{l:randomVariableTrick}(ii) to the inequality $\E[\xi_{k+1}\;|\; x_k] \leq (1-\gamma)\xi_{k}$. However, in order to be able to use Lemma~\ref{l:randomVariableTrick}, we would have to first establish monotonicity of the sequence $\{\xi_k\}$, $k \geq 0$. This is not necessary using the direct approach of Theorem~\ref{thm:complexity-strongly-convex-case}. Hence, in the strongly convex case we can analyze PCDM1 and are not forced to resort to PCDM2. Consider now the following situations:
\begin{enumerate}
\item $\mu_f(w) = 0$. Then the leading term in \eqref{eq:k_uniform_strong} is $\tfrac{1+\beta/\mu_\cPsi(w)}{\alpha}$.
\item $\mu_\cPsi(w) = 0$. Then the leading term in \eqref{eq:k_uniform_strong} is $\tfrac{\beta/\mu_f(w)}{\alpha}$.
\item $\mu_\cPsi(w)$ is ``large enough''. Then   $\tfrac{\beta+\mu_{\cPsi}(w)}{\mu_f(w)+\mu_{\cPsi}(w)} \approx 1$ and the leading term in  \eqref{eq:k_uniform_strong} is $\tfrac{1}{\alpha}$.
\end{enumerate}
In a similar way as in the non-strongly convex case, define the parallelization speedup factor as the ratio of the leading term in \eqref{eq:k_uniform_strong} for the serial method (which has $\alpha=\tfrac{1}{n}$ and $\beta=1$) and the leading term for a parallel method:
\begin{equation}\label{eq:speedup_factor_SC}\text{parallelization speedup factor} =  \frac{n\tfrac{1 + \mu_\cPsi(w)}{\mu_f(w) + \mu_\cPsi(w)} }
{\tfrac{1}{\alpha} \tfrac{\beta + \mu_\cPsi(w)}{\mu_f(w) + \mu_\cPsi(w)}  } = \frac{n}{\frac{\beta + \mu_\cPsi(w)}{\alpha (1+\mu_\cPsi(w))}}.\end{equation}

First, note that the speedup factor is independent of $\mu_f$. Further, note that as $\mu_\cPsi(w)\to 0$, the speedup factor approaches the factor we obtained in the non-strongly convex case (see \eqref{eq:speedup_factor_C} and also Table~\ref{tbl:upperbounds}). That is, for large values of $\mu_\cPsi(w)$, the speedup factor is approximately equal $\alpha n = \E[|\Srv|]$, which is  the average number of blocks updated in a single parallel iteration. Note that thuis quantity \emph{does not} depend on the degree of partial separability of $f$.

%


\section{Numerical Experiments}\label{SEC:Numerical_Experiments}

In Section~\ref{sec:billion} we present preliminary but very encouraging results showing that PCDM1 run on a system with 24 cores can solve huge-scale partially-separable LASSO problems with a billion variables in  2 hours, compared with 41 hours on a single core. In Section~\ref{sec:TR} we demonstrate that our analysis is in some sense tight. In particular, we show that the speedup predicted by the theory can be matched almost exactly by actual wall time speedup for a particular problem.

\subsection{A LASSO problem with 1 billion variables}\label{sec:billion}


In this experiment we solve a single randomly generated huge-scale LASSO instance, i.e., \eqref{eq:P} with \[f(x)=\tfrac{1}{2}\|Ax-b\|_2^2, \qquad \cPsi(x) = \|x\|_1,\]  where $A=[a_1,\dots,a_n]$ has $2\times 10^9$ rows and $\N=n=10^9$ columns. We generated the problem  using a modified primal-dual generator \cite{RT:UCDC} enabling us to choose the optimal solution $x^*$ (and hence, indirectly, $F^*$) and thus to control its cardinality $\|x^*\|_0$, as well as the sparsity level of $A$. In particular, we made the following choices: $\|x^*\|_0 = 10^5$, each column of $A$ has exactly 20 nonzeros and the maximum cardinality of a row of $A$ is $\omega= 35$ (the degree of partial separability of $f$). The histogram of cardinalities is displayed in Figure~\ref{fig:histogram}.

\begin{figure}[ht!]
\begin{center}
\includegraphics[width=3in]{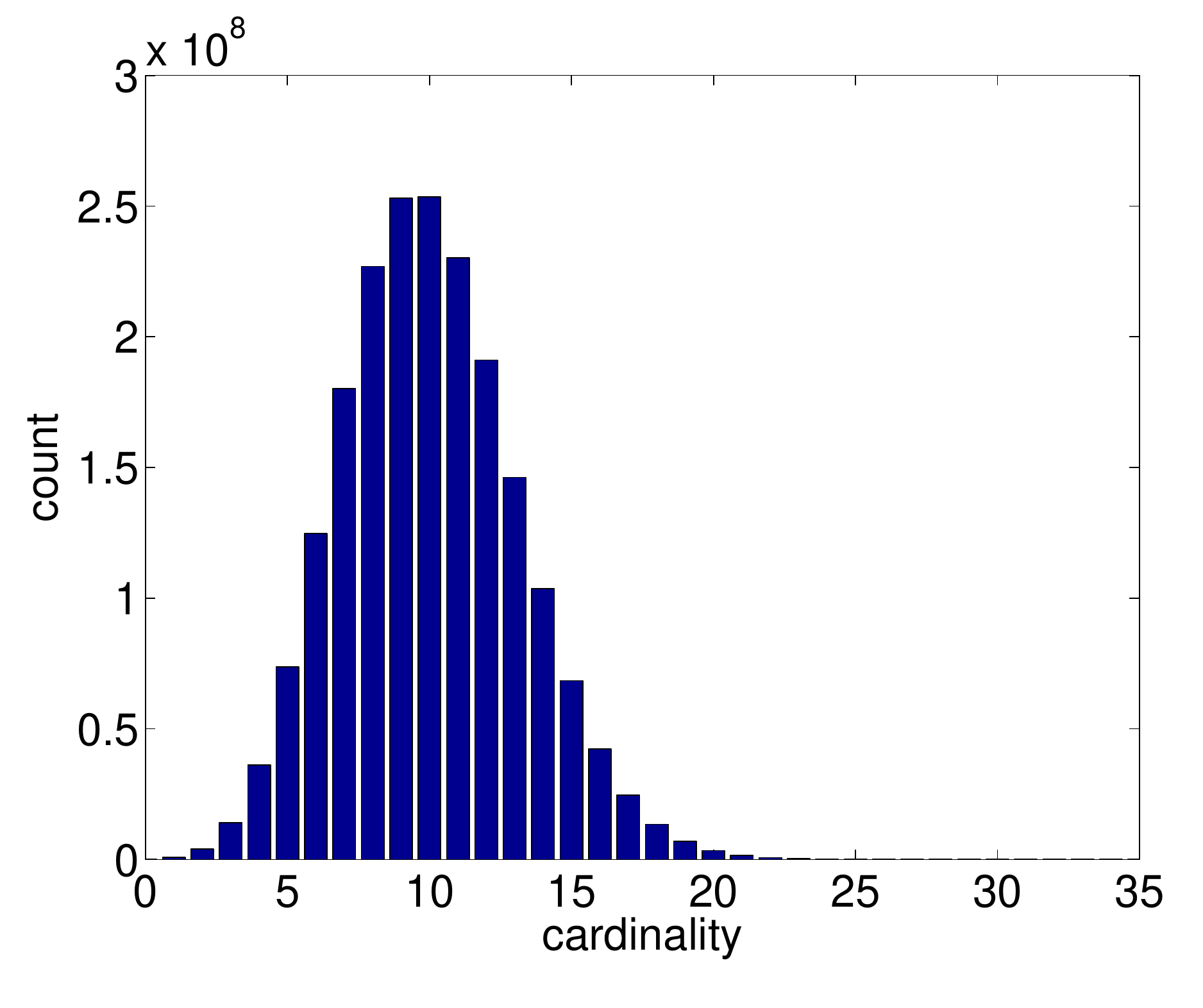}
\caption{Histogram of the cardinalities of the rows of $A$.}
\label{fig:histogram}
\end{center}
\end{figure}

\begin{table}[!ht]
 \centering
 \footnotesize
 \rotatebox{90}{
 \begin{tabular}{|c||c|c|c|c|c|c||c|c|c|c|c|c|}
  & \multicolumn{6}{c||} {$F(x_k)-F^*$}
  & \multicolumn{6}{|c|} {Elapsed Time}
  \\ \hline
  $\tfrac{\tau k}{n}$ &
  $\tau =1$ & $\tau = 2$ & $\tau = 4$  & $\tau = 8$ & $\tau = 16$ & $\tau = 24$
  &
  $\tau = 1$ & $\tau = 2$ & $\tau = 4$  & $\tau = 8$ & $\tau = 16$ & $\tau = 24$
  \\ \hline \hline
0 & 6.27e+22 & 6.27e+22 & 6.27e+22 & 6.27e+22 & 6.27e+22 & 6.27e+22 & 0.00 & 0.00 & 0.00 & 0.00 & 0.00 & 0.00 \\
1 & 2.24e+22 & 2.24e+22 & 2.24e+22 & 2.24e+22 & 2.24e+22 & 2.24e+22 & 0.89 & 0.43 & 0.22 & 0.11 & 0.06 & 0.05 \\
2 & 2.24e+22 & 2.24e+22 & 2.24e+22 & 3.64e+19 & 2.24e+22 & 8.13e+18 & 1.97 & 1.06 & 0.52 & 0.27 & 0.14 & 0.10 \\
3 & 1.15e+20 & 2.72e+19 & 8.37e+19 & 1.94e+19 & 1.37e+20 & 5.74e+18 & 3.20 & 1.68 & 0.82 & 0.43 & 0.21 & 0.16 \\
4 & 5.25e+19 & 1.45e+19 & 2.22e+19 & 1.42e+18 & 8.19e+19 & 5.06e+18 & 4.28 & 2.28 & 1.13 & 0.58 & 0.29 & 0.22 \\
5 & 1.59e+19 & 2.26e+18 & 1.13e+19 & 1.05e+17 & 3.37e+19 & 3.14e+18 & 5.37 & 2.91 & 1.44 & 0.73 & 0.37 & 0.28 \\
6 & 1.97e+18 & 4.33e+16 & 1.11e+19 & 1.17e+16 & 1.33e+19 & 3.06e+18 & 6.64 & 3.53 & 1.75 & 0.89 & 0.45 & 0.34 \\
7 & 2.40e+16 & 2.94e+16 & 7.81e+18 & 3.18e+15 & 8.39e+17 & 3.05e+18 & 7.87 & 4.15 & 2.06 & 1.04 & 0.53 & 0.39 \\
8 & 5.13e+15 & 8.18e+15 & 6.06e+18 & 2.19e+14 & 5.81e+16 & 9.22e+15 & 9.15 & 4.78 & 2.37 & 1.20 & 0.61 & 0.45 \\
9 & 8.90e+14 & 7.87e+15 & 2.09e+16 & 2.08e+13 & 2.24e+16 & 5.63e+15 & 10.43 & 5.39 & 2.67 & 1.35 & 0.69 & 0.51 \\
10 & 5.81e+14 & 6.52e+14 & 7.75e+15 & 3.42e+12 & 2.89e+15 & 2.20e+13 & 11.73 & 6.02 & 2.98 & 1.51 & 0.77 & 0.57 \\
11 & 5.13e+14 & 1.97e+13 & 2.55e+15 & 1.54e+12 & 2.55e+15 & 7.30e+12 & 12.81 & 6.64 & 3.29 & 1.66 & 0.84 & 0.63 \\
12 & 5.04e+14 & 1.32e+13 & 1.84e+13 & 2.18e+11 & 2.12e+14 & 1.44e+12 & 14.08 & 7.26 & 3.60 & 1.83 & 0.92 & 0.68 \\
\hline
13 & 2.18e+12 & 7.06e+11 & 6.31e+12 & 1.33e+10 & 1.98e+14 & 6.37e+11 & {\bf 15.35} & {\bf 7.88} & {\bf 3.91} & {\bf 1.99} & {\bf 1.00} & 0.74 \\
\hline
14 & 7.77e+11 & 7.74e+10 & 3.10e+12 & 3.43e+09 & 1.89e+12 & 1.20e+10 & 16.65 & 8.50 & 4.21 & 2.14 & 1.08 & 0.80 \\
15 & 1.80e+10 & 6.23e+10 & 1.63e+11 & 1.60e+09 & 5.29e+11 & 4.34e+09 & 17.94 & 9.12 & 4.52 & 2.30 & 1.16 & 0.86 \\
16 & 1.38e+09 & 2.27e+09 & 7.86e+09 & 1.15e+09 & 1.46e+11 & 1.38e+09 & 19.23 & 9.74 & 4.83 & 2.45 & 1.24 & 0.91 \\
17 & 3.63e+08 & 3.99e+08 & 3.07e+09 & 6.47e+08 & 2.92e+09 & 7.06e+08 & 20.49 & 10.36 & 5.14 & 2.61 & 1.32 & 0.97 \\
18 & 2.10e+08 & 1.39e+08 & 2.76e+08 & 1.88e+08 & 1.17e+09 & 5.93e+08 & 21.76 & 10.98 & 5.44 & 2.76 & 1.39 & 1.03 \\
19 & 3.81e+07 & 1.92e+07 & 7.47e+07 & 1.55e+06 & 6.51e+08 & 5.38e+08 & 23.06 & 11.60 & 5.75 & 2.91 & 1.47 & 1.09 \\
20 & 1.27e+07 & 1.59e+07 & 2.93e+07 & 6.78e+05 & 5.49e+07 & 8.44e+06 & 24.34 & 12.22 & 6.06 & 3.07 & 1.55 & 1.15 \\
21 & 4.69e+05 & 2.65e+05 & 8.87e+05 & 1.26e+05 & 3.84e+07 & 6.32e+06 & 25.42 & 12.84 & 6.36 & 3.22 & 1.63 & 1.21 \\
22 & 1.47e+05 & 1.16e+05 & 1.83e+05 & 2.62e+04 & 3.09e+06 & 1.41e+05 & 26.64 & 13.46 & 6.67 & 3.38 & 1.71 & 1.26 \\
23 & 5.98e+04 & 7.24e+03 & 7.94e+04 & 1.95e+04 & 5.19e+05 & 6.09e+04 & 27.92 & 14.08 & 6.98 & 3.53 & 1.79 & 1.32 \\
24 & 3.34e+04 & 3.26e+03 & 5.61e+04 & 1.75e+04 & 3.03e+04 & 5.52e+04 & 29.21 & 14.70 & 7.28 & 3.68 & 1.86 & 1.38 \\
25 & 3.19e+04 & 2.54e+03 & 2.17e+03 & 5.00e+03 & 6.43e+03 & 4.94e+04 & 30.43 & 15.32 & 7.58 & 3.84 & 1.94 & 1.44 \\
\hline
26 & 3.49e+02 & 9.62e+01 & 1.57e+03 & 4.11e+01 & 3.68e+03 & 4.91e+04 & {\bf 31.71} & {\bf 15.94} & {\bf 7.89} & {\bf 3.99} & {\bf 2.02} & 1.49 \\
\hline
27 & 1.92e+02 & 8.38e+01 & 6.23e+01 & 5.70e+00 & 7.77e+02 & 4.90e+04 & 33.00 & 16.56 & 8.20 & 4.14 & 2.10 & 1.55 \\
28 & 1.07e+02 & 2.37e+01 & 2.38e+01 & 2.14e+00 & 6.69e+02 & 4.89e+04 & 34.23 & 17.18 & 8.49 & 4.30 & 2.17 & 1.61 \\
29 & 6.18e+00 & 1.35e+00 & 1.52e+01 & 2.35e-01 & 3.64e+01 & 4.89e+04 & 35.31 & 17.80 & 8.79 & 4.45 & 2.25 & 1.67 \\
30 & 4.31e+00 & 3.93e-01 & 6.25e-01 & 4.03e-02 & 2.74e+00 & 3.15e+01 & 36.60 & 18.43 & 9.09 & 4.60 & 2.33 & 1.73 \\
31 & 6.17e-01 & 3.19e-01 & 1.24e-01 & 3.50e-02 & 6.20e-01 & 9.29e+00 & 37.90 & 19.05 & 9.39 & 4.75 & 2.41 & 1.78 \\
32 & 1.83e-02 & 3.06e-01 & 3.25e-02 & 2.41e-03 & 2.34e-01 & 3.10e-01 & 39.17 & 19.67 & 9.69 & 4.91 & 2.48 & 1.84 \\
33 & 3.80e-03 & 1.75e-03 & 1.55e-02 & 1.63e-03 & 1.57e-02 & 2.06e-02 & 40.39 & 20.27 & 9.99 & 5.06 & 2.56 & 1.90 \\
34 & 7.28e-14 & 7.28e-14 & 1.52e-02 & 7.46e-14 & 1.20e-02 & 1.58e-02 & 41.47 & 20.89 & 10.28 & 5.21 & 2.64 & 1.96 \\
35&-&-&1.24e-02&-&1.23e-03&8.70e-14&-&-& 10.58&-&2.72&  2.02 \\
36&-&-&2.70e-03&-&3.99e-04&-&-&-& 10.88&-&2.80&  - \\
37&-&-&7.28e-14&-&7.46e-14&-&-&-& 11.19&-&2.87&  - \\
\hline
\end{tabular}
}
 \caption{A LASSO problem with $10^9$ variables solved by PCDM1 with $\tau =$ 1, 2, 4, 8, 16 and  24.}
 \label{tbl:large}
\end{table}

We solved the problem using PCDM1 with $\tau$-nice sampling $\Srv$, $\beta = 1+ \tfrac{(\omega-1)(\tau-1)}{n-1}$ and $w=L=(\|a_1\|^2_2,\cdots,\|a_n\|_2^2)$, for $\tau=1,2,4,8,16, 24$, on a single large-memory computer utilizing $\tau$ of its 24 cores. The problem description took around 350GB of memory space. In fact, in our implementation we departed from the just described setup in two ways. First, we implemented an \emph{asynchronous} version of the method; i.e., one in which cores do not wait for others to update the current iterate within an iteration before reading $x_{k+1}$ and proceeding to another update step. Instead, each core reads the current iterate whenever it is ready with the previous update step and applies the new update as soon as it is computed. Second, as mentioned in Section~\ref{SEC:Block_Samplings},  the $\tau$-independent sampling is for $\tau\ll n$ a very good approximation of the $\tau$-nice sampling. We therefore allowed each processor to pick a block uniformly at random, independently from the other processors.

\textbf{Choice of the first column of Table~\ref{tbl:large}.} In Table~\ref{tbl:large} we show the development of the gap $F(x_k)-F^*$ as well as the elapsed time. The choice and meaning of the first column of the table, $\tfrac{\tau k}{n}$, needs some commentary. Note that exactly $\tau k$ coordinate updates are performed after $k$ iterations. Hence, the first column denotes the total number of coordinate updates normalized by the number of coordinates $n$. As an example, let $\tau_1=1$ and $\tau_2=24$. Then if the serial method is run for $k_1=24$ iterations  and the parallel one for $k_2=1$ iteration, both methods would have updated the same number ($\tau_1 k_1 = \tau_2 k_2 = 24$) of coordinates; that is, they would ``be'' in the same row of Table~\ref{tbl:large}. In summary, each  row of the table represents, in the sense described above, the ``same amount of work done'' for each choice of $\tau$.

\textbf{Progress to solving the problem.} One can conjecture that the above meaning of the phrase ``same amount of work done'' would perhaps be roughly equivalent to a different one: ``same progress to solving the problem''. Indeed, it turns out, as can be seen from the table and also from Figure~\ref{fig:largeProblem}(a), that in each row for all algorithms the value of $F(x_k)-F^*$ is roughly of the same order of magnitude. This is not a trivial finding since, with increasing $\tau$, older information is used to update the coordinates, and hence one would expect  that convergence would be slower. It does seem to be slower---the gap $F(x_k)-F^*$ is generally higher if more processors are used---but the slowdown is limited. Looking at Table~\ref{tbl:large} and/or  Figure~\ref{fig:largeProblem}(a), we see that for all choices of $\tau$,  PCDM1 managed to push the gap below $10^{-13}$ after $34n$ to $37n$ coordinate updates.

\begin{figure}[!ht]
 \centering
 \centering
\subfigure[For each $\tau$, PCDM1 needs roughly the same number of coordinate updates to solve the problem.]{
\includegraphics[width=3in]{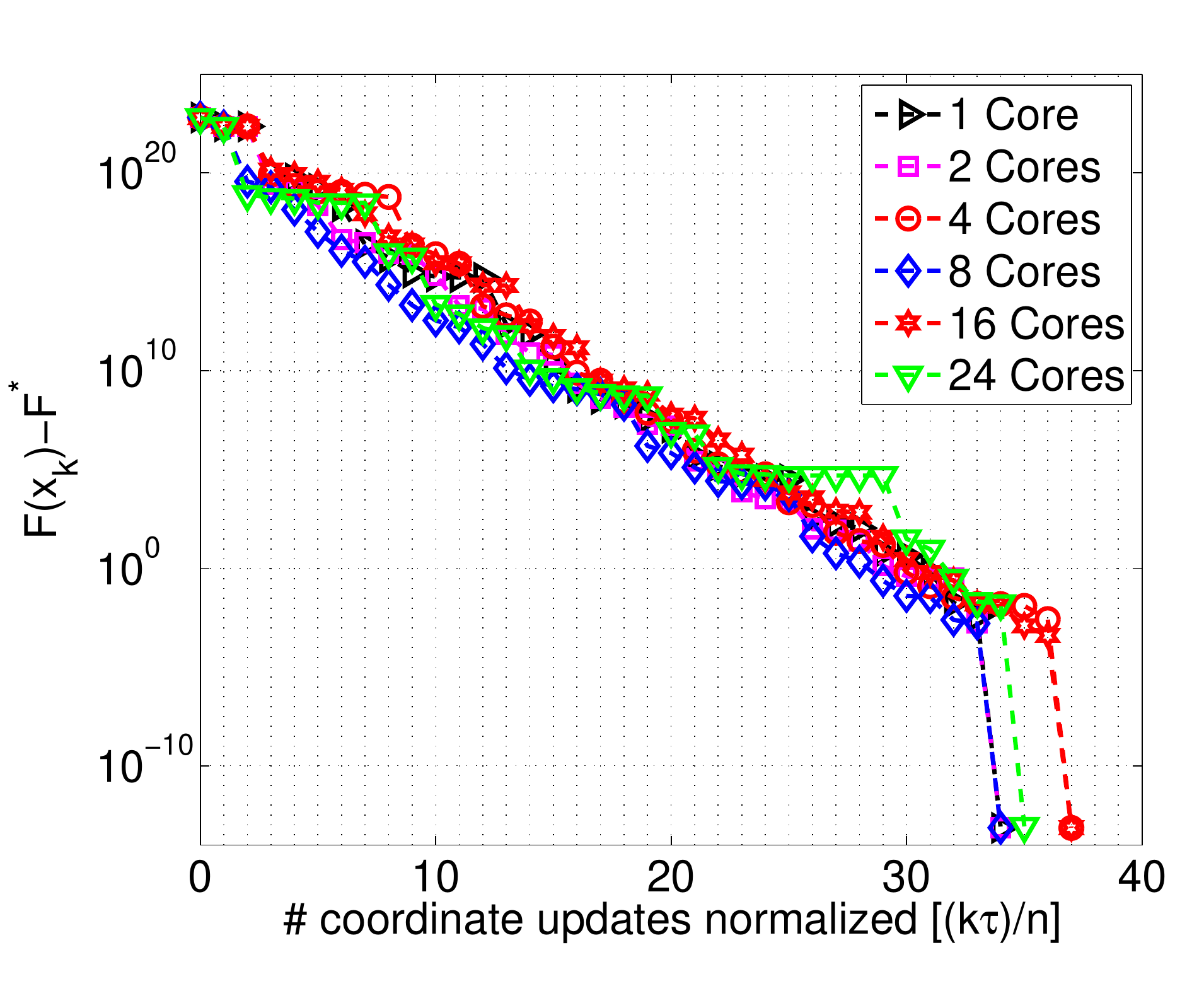}
\label{fig:largeProblem-coordinates}
}
\subfigure[Doubling the number of cores corresponds to roughly halving the number of iterations.]{
\includegraphics[width=3in]{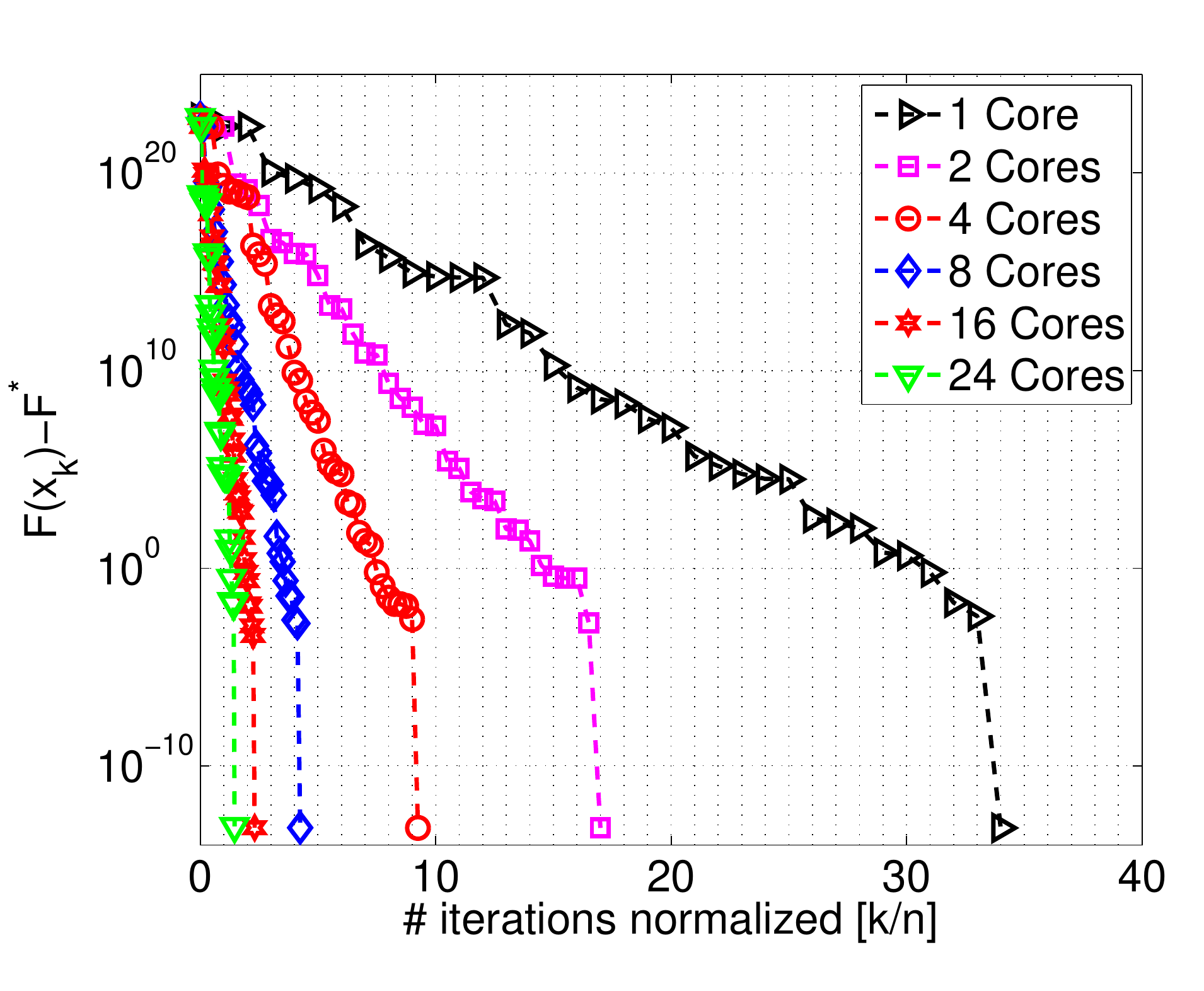}
\label{fig:largeProblem-steps}
}
\subfigure[Doubling the number of cores corresponds to roughly halving the wall time.]{
\includegraphics[width=3in]{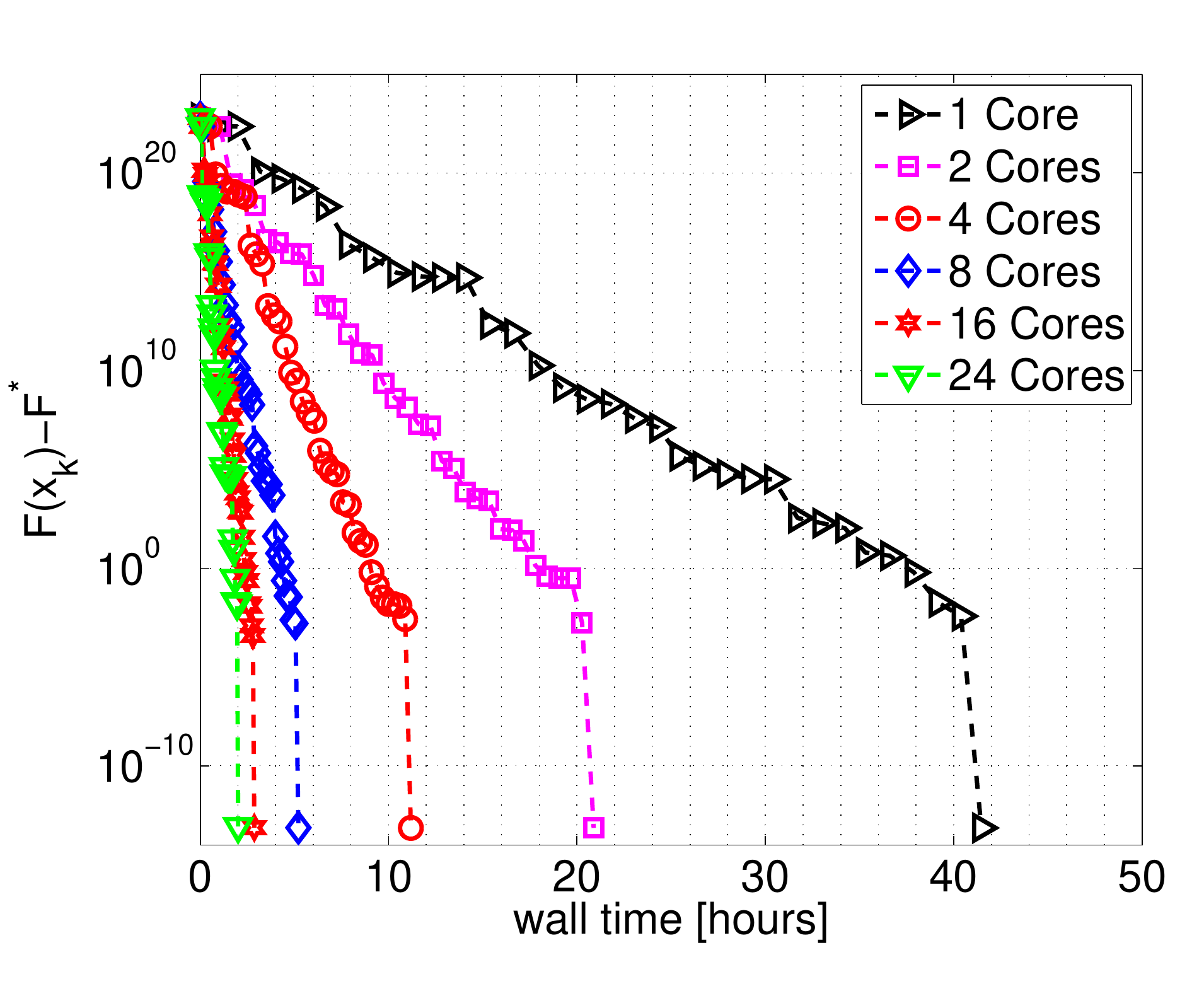}
\label{fig:largeProblem-times}
}
\subfigure[Parallelization speedup is essentially equal to the number of cores.]{
\includegraphics[width=3in]{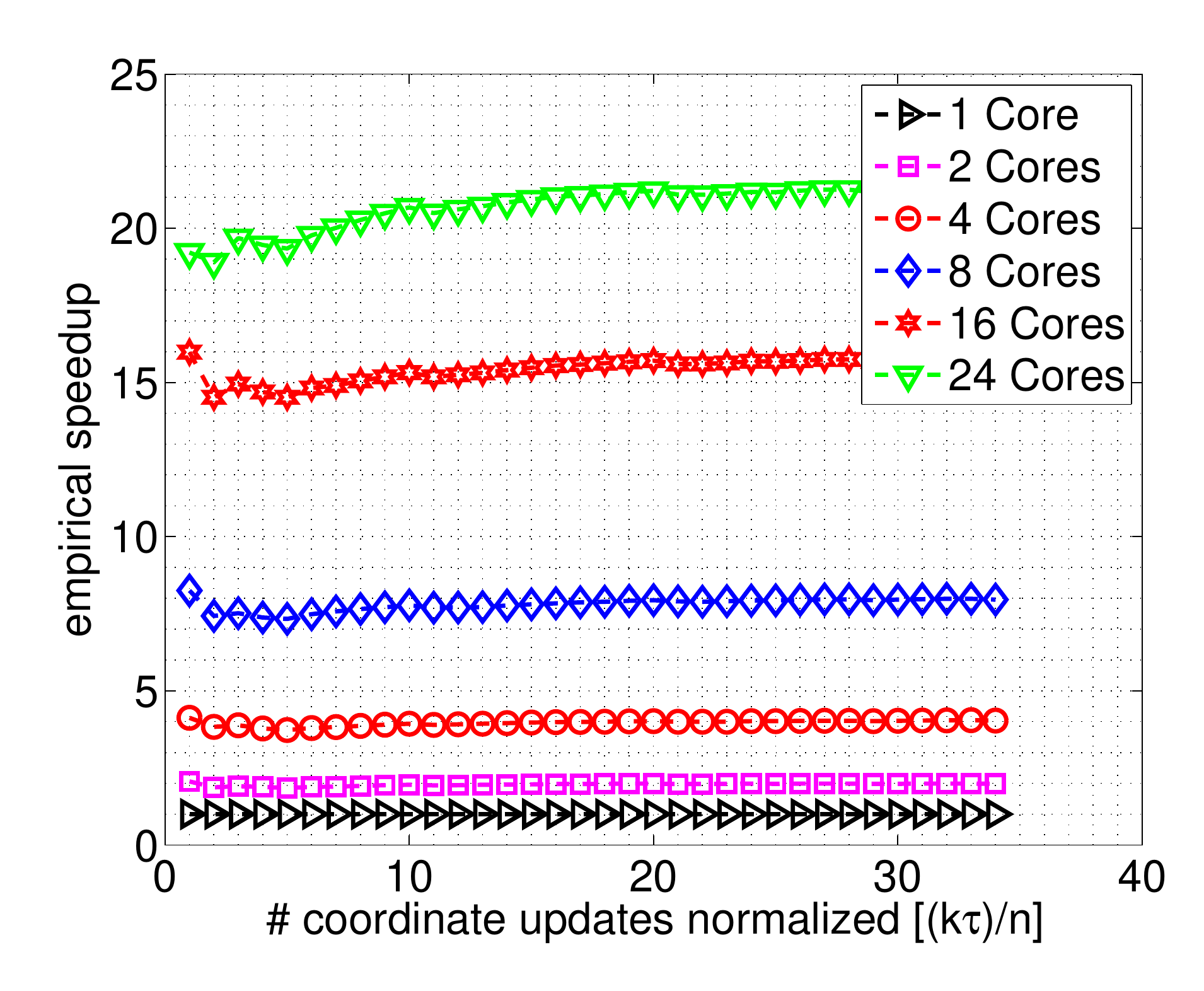}
\label{fig:largeProblem-speedup}
}
\caption{Four computational insights into the workings of PCDM1.}
\label{fig:largeProblem}
\end{figure}

The progress to solving the problem during the final 1 billion coordinate updates (i.e., when moving from the last-but-one to the last nonempty line in each of the columns of Table~\ref{tbl:large} showing $F(x_k)-F^*$ ) is remarkable. The method managed to push the optimality gap by 9-12 degrees of magnitude. We do not have an explanation for this phenomenon; we do not give local convergence estimates in this paper. It is certainly the case though that once the method managed to find the nonzero places of $x^*$, fast local convergence comes in.

\textbf{Parallelization speedup.} Since a parallel method utilizing $\tau$ cores manages to do the same number of coordinate updates as the serial one  $\tau$ times faster, a direct consequence of the above observation is that doubling the number of cores corresponds to roughly halving the number of iterations (see Figure~\ref{fig:largeProblem}(b). This is due to the fact that $\omega\ll n$ and $\tau \ll n$. It turns out that the number of iterations is an excellent predictor of wall time; this can be seen by comparing Figures~\ref{fig:largeProblem}(b) and \ref{fig:largeProblem}(c). Finally, it follows from the above, and can be seen in Figure~\ref{fig:largeProblem}(d), that the speedup of PCDM1 utilizing $\tau$ cores is roughly equal to $\tau$. Note that this is caused by the fact that the problem is, relative to its dimension, partially separable to a very high degree.

\subsection{Theory versus reality}\label{sec:TR}

In our second experiment we  demonstrate numerically that our parallelization speedup estimates are in some sense tight. For this purpose it is not necessary to reach for complicated problems and high dimensions; we hence minimize the function $\frac12 \|Ax-b\|_2^2$ with $A\in\R^{3000 \times 1000}$.
Matrix $A$ was generated so that its every row contains exactly $\omega$ non-zero values all of which are equal (recall the construction in point 3 at the end of Section~\ref{subsec:DSO}).

\begin{figure}[!ht]
 \centering
\includegraphics[width=3in]{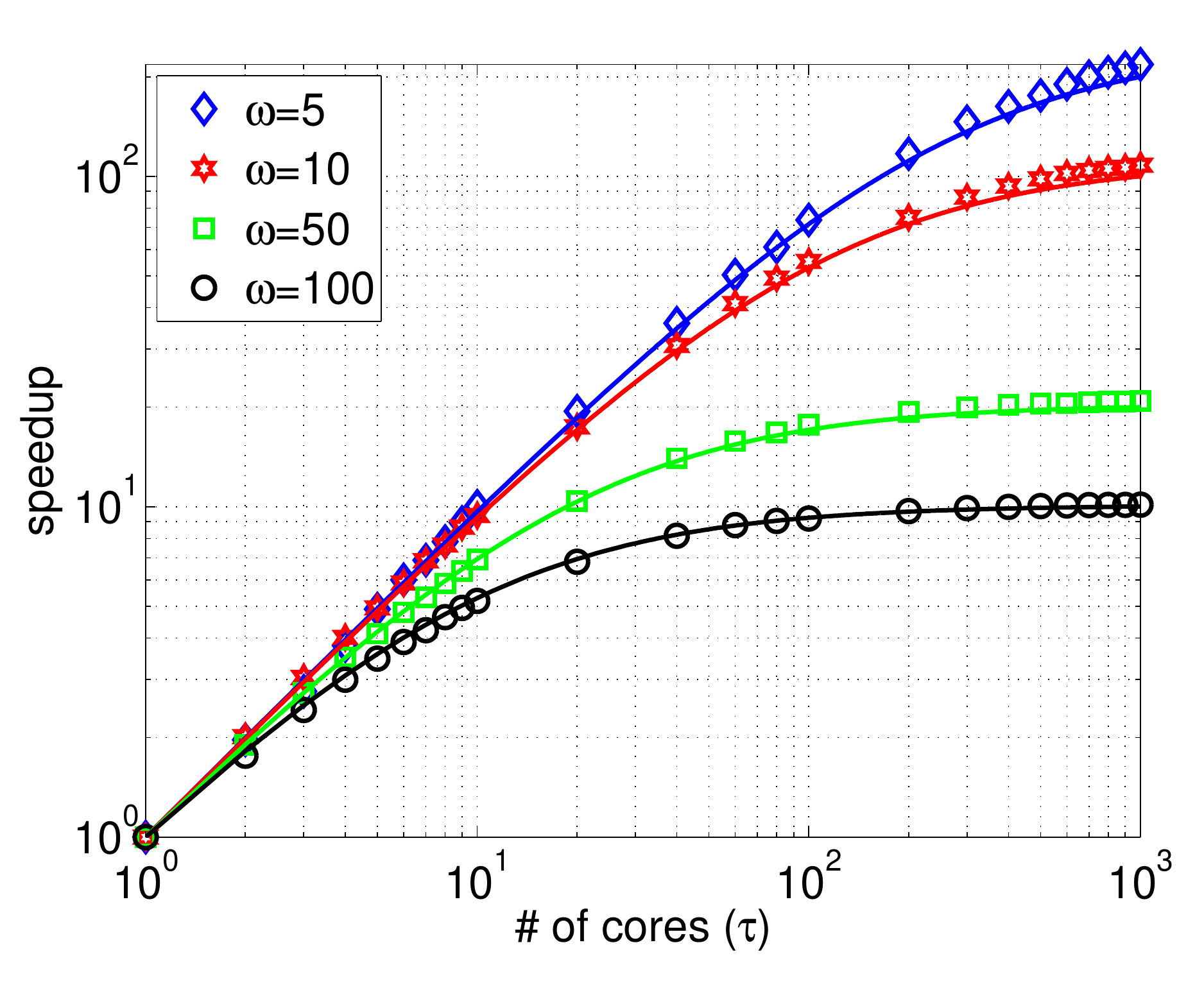}
 \caption{Theoretical speedup factor predicts the actual speedup almost exactly for a carefully constructed problem.}
 \label{fig:theoryvspractise}
\end{figure}

We generated 4 matrices with $\omega =5, 10, 50$ and $ 100$ and measured the number of iterations needed for PCDM1 used with $\tau$-nice sampling  to get within $\epsilon = 10^{-6}$ of the optimal value. The experiment was done for a range of values of $\tau$ (between 1 core and 1000 cores).

The solid lines in Figure~\ref{fig:theoryvspractise} present the \emph{theoretical speedup factor} for the $\tau$-nice sampling, as presented in Table~\ref{tbl:upperbounds}. The markers in each case correspond to \emph{empirical speedup factor} defined as

\[\frac{\mbox{\# of iterations till $\epsilon$-solution is found by PCDM1 used with serial sampling}}
{\mbox{\# of iterations till $\epsilon$-solution is found by PCDM1 used with $\tau$-nice sampling}}.
\]

As can be seen in Figure~\ref{fig:theoryvspractise}, the match between theoretical prediction and reality is remarkable! A partial explanation  of this phenomenon lies in the fact that we have carefully designed the problem so as to ensure that the degree of partial separability is equal to the Lipschitz constant $\sigma$ of $\nabla f$ (i.e., that it is not a gross overestimation of it; see Section~\ref{subsec:DSO}). This fact is useful since it is possible to prove complexity results with $\omega$ replaced by $\sigma$. However, this answer is far from satisfying, and a deeper understanding of the phenomenon remains an open problem.

\subsection{Training linear SVMs with bad data for PCDM}

In this experiment we test PCDM on the problem of training a linear Support Vector Machine (SVM) based on $n$ labeled training examples: $(y_i,A_i)\in \{+1,-1\}\times \R^d$, $i=1,2,\dots,n$. In particular, we consider the primal problem of minimizing L2-regularized  average hinge-loss,

\[\min_{w\in \R^d} \left\{g(w) \eqdef \frac{1}{n} \sum_{i=1}^n [1-y_i \ve{w}{a_i}]_+ + \frac{\lambda}{2}\|w\|_2^2\right\},\]
and the dual problem of maximizing a concave quadratic subject to zero-one box constraints,
\[
\max_{x\in \R^n,\; 0\leq x^{(i)} \leq 1} \left\{ -f(x) \eqdef -\frac{1}{2\lambda n^2}x^T Z x + \frac{1}{n}\sum_{i=1}^n x^{(i)} \right\},
\]
where $Z \in \R^{n\times n}$ with $Z_{ii}=y_i y_j \ve{A_i}{A_j}$.
It is a standard practice to apply \emph{serial} coordinate descent to the dual. Here we apply \emph{parallel} coordinate descent (PCDM; with $\tau$-nice sampling of coordinates) to the dual; i.e., minimize the convex function $f$ subject to box constraints. In this setting all blocks are of size $N_i=1$. The dual can be written in the form \eqref{eq:P}, i.e., \[\min_{x \in \R^n} \{F(x)=f(x)+\cPsi(x)\},\] where $\cPsi(x) = 0$ whenever $x^{(i)} \in [0,1]$ for all $i=1,2,\dots,n$, and $\cPsi(x)=+\infty$ otherwise.


We consider the  \texttt{rcv1.binary} dataset\footnote{\url{http://www.csie.ntu.edu.tw/~cjlin/libsvmtools/datasets/binary.html\#rcv1.binary}}.
The training data has $n = 677,399$ examples, $d= 47,236$ features, $49,556,258$ nonzero elements and  requires cca 1GB of RAM for storage. Hence, this is a small-scale problem. The degree of partial separability of $f$ is $\omega = 291,516$ (i.e., the maximum number of examples sharing a given feature). This is a very large number relative to $n$, and hence our theory would predict rather bad behavior for PCDM. We use PCDM1 with $\tau$-nice sampling ( approximating it by $\tau$-independent sampling for added efficiency) with $\beta$ following Theorem~\ref{thm:ESO-nice}: $\beta=1+ \frac{(\tau-1)(\omega-1)}{n-1}$.

The results of our experiments are summarized in Figure~\ref{fig:rcv1}. 
 Each column corresponds to a different level of regularization: $\lambda \in \{1,10^{-3},10^{-5}\}$. 
The rows show the 1) duality gap, 2) dual suboptimality, 3) train error and 4) test error; each for 1,4 and 16 processors ($\tau = 1,4,16$).
Observe that the plots in the first two rows are nearly identical; which means that the method is able to solve the primal problem at about the same speed as it can solve the dual problem\footnote{Revision comment: We did not propose primal-dual versions of PCDM in this paper, but we do so in the follow up work \cite{minibatch-ICML2013}. In this paper, for the SVM problem, our methods and theory apply to the dual only.}.

Observe also that in all cases, duality gap of around $0.01$ is sufficient for training as training error (classification performance of the SVM on the train data) does not decrease further after this point. Also observe the effect of $\lambda$ on training accuracy: accuracy increases from about $92\%$ for $\lambda=1$, through $95.3\%$
 for $\lambda =10^{-3}$ to above $97.8\%$ with $\lambda=10^{-5}$. In our case, choosing smaller  $\lambda$ does not lead to overfitting; the test error on test dataset (\# features =677,399, \# examples = 20,242) increases as $\lambda$ decreases, quickly reaching about $95\%$ (after 2 seconds of training) for $\lambda=0.001$ and for the smallest $\lambda$ going beyond $97\%$.

Note that PCDM with $\tau=16$ is about 2.5$\times$ faster than PCDM with $\tau=1$. This is much less than linear speedup, but is fully in line with our theoretical predictions. Indeed, for $\tau=16$ we get $\beta = 7.46$. Consulting Table~\ref{tbl:upperbounds}, we see that the theory says that with $\tau=16$ processors we should expect the parallelization speedup to be
$PSF= \tau/\beta = 2.15 $.



\begin{figure}[h!]
 
 \begin{tabular}{ccc}
 $\lambda=1$ & $\lambda=0.001$ & $\lambda=0.00001$ \\
 \includegraphics[width=2in]{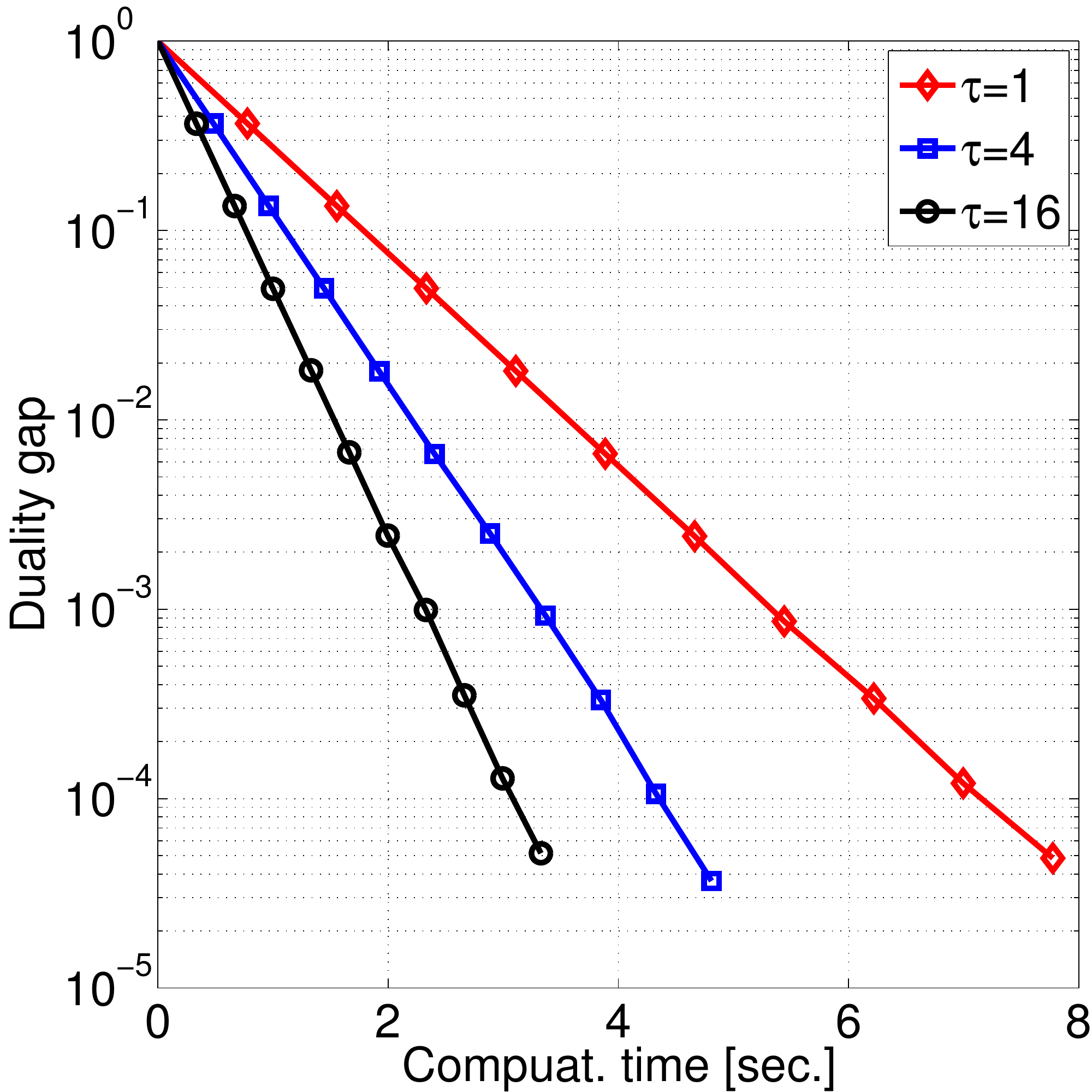}&
\includegraphics[width=2in]{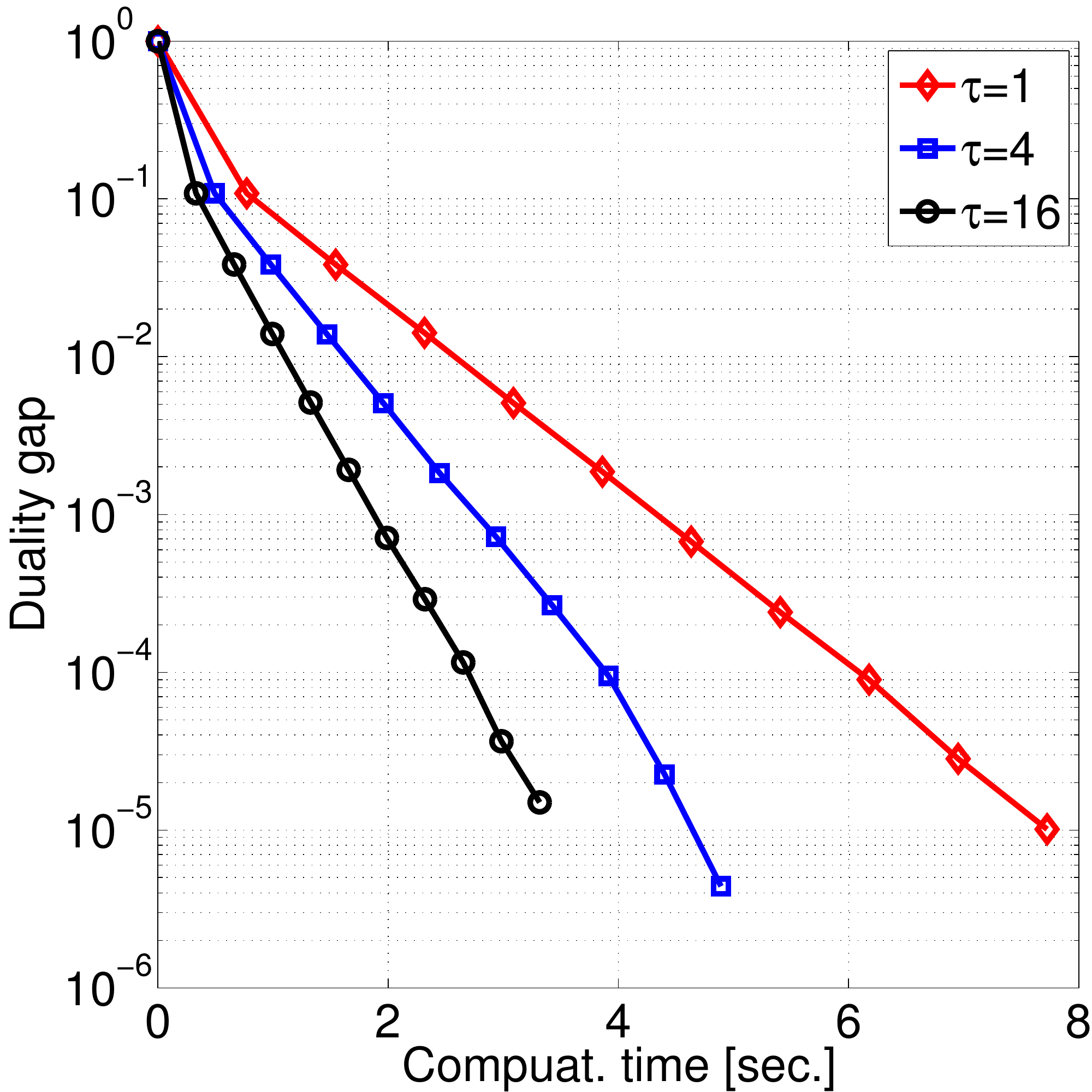} &
\includegraphics[width=2in]{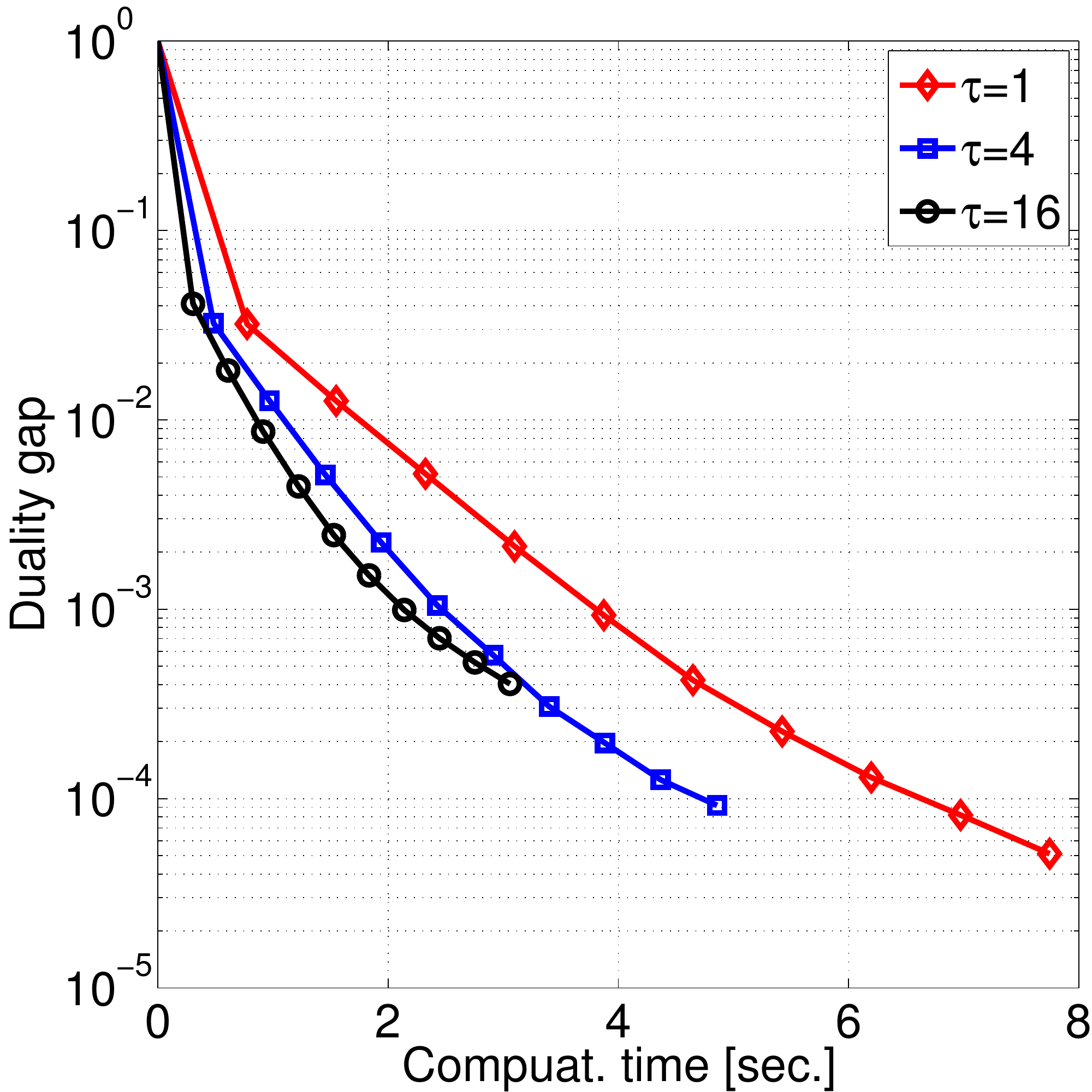} 
\\
 \includegraphics[width=2in]{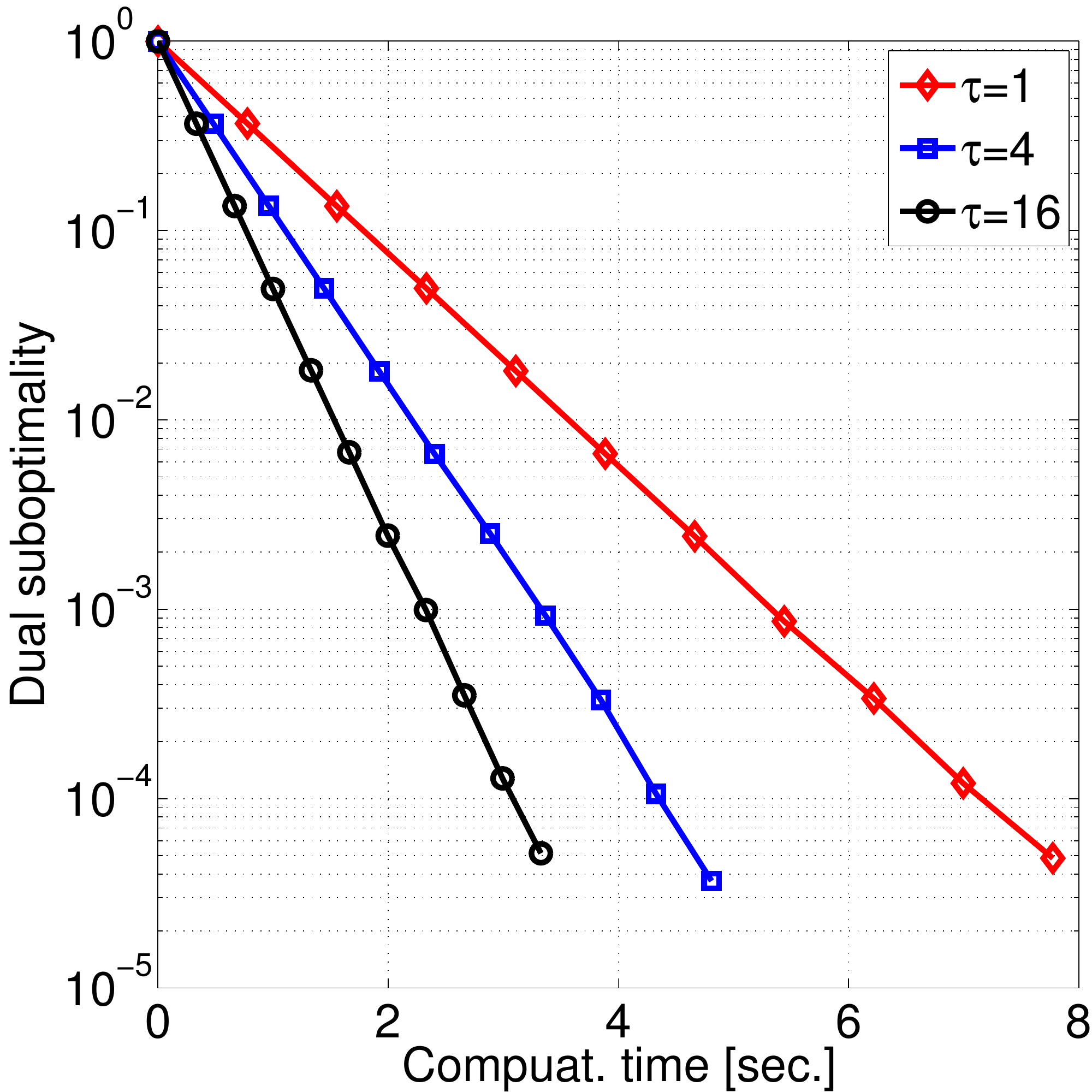}&
\includegraphics[width=2in]{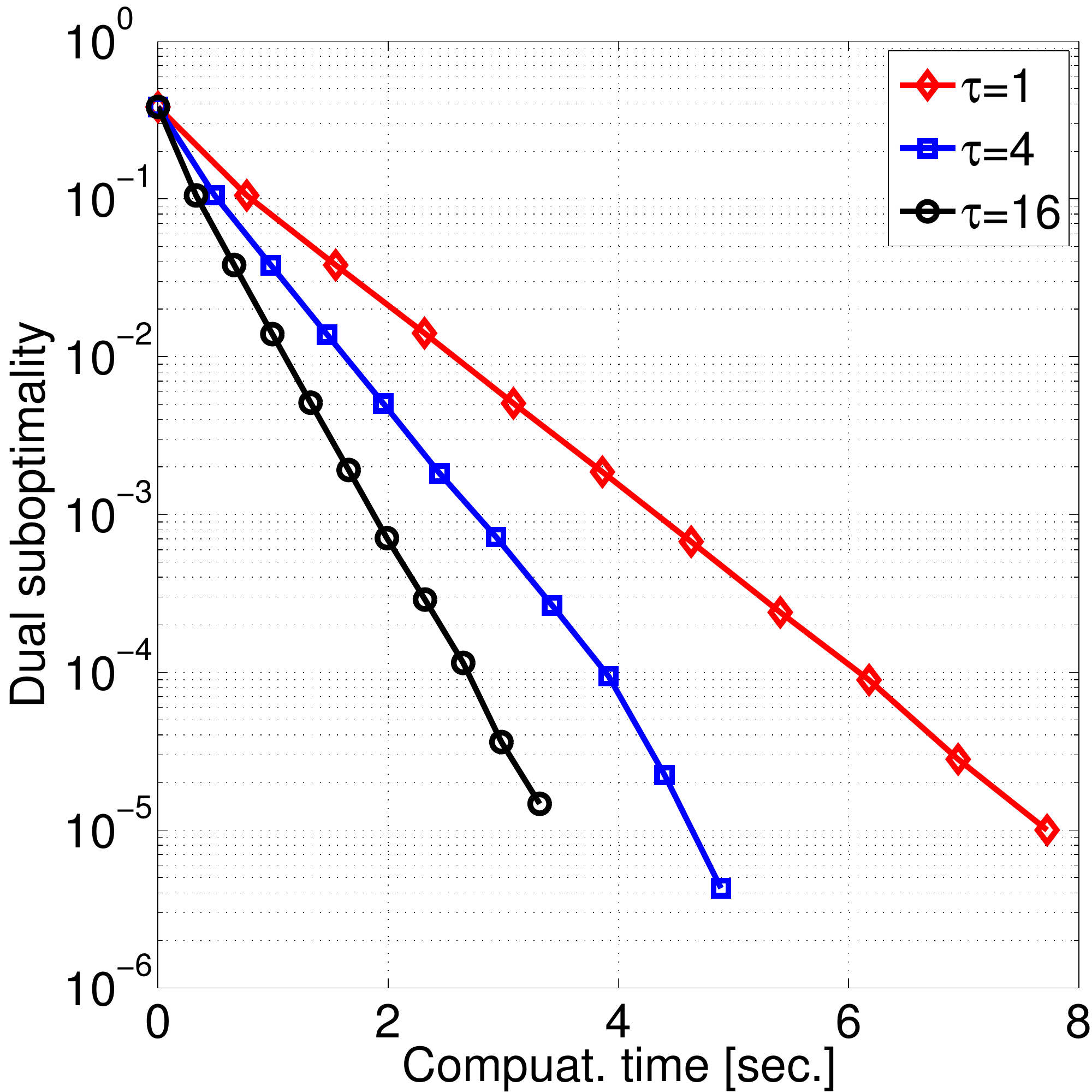} &
\includegraphics[width=2in]{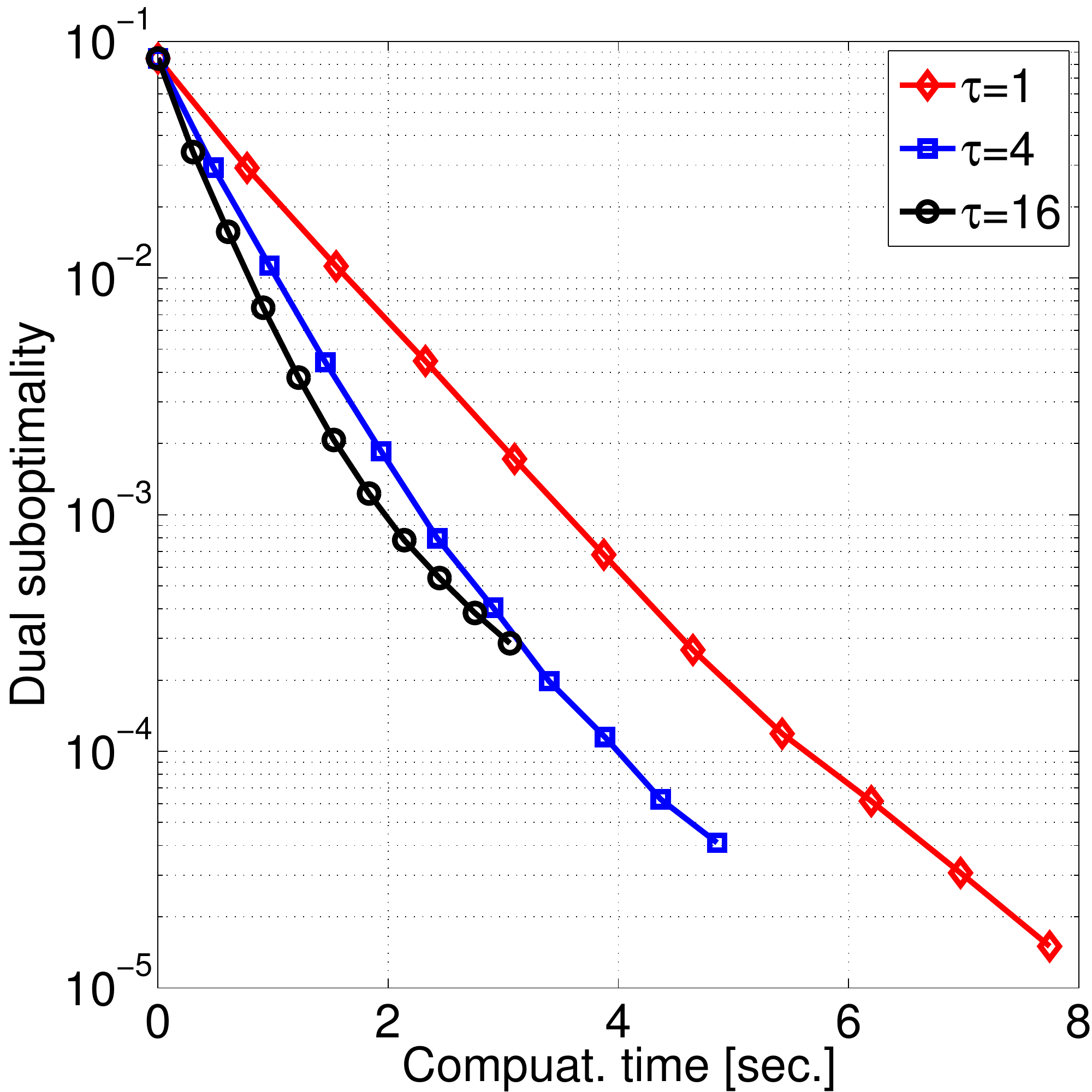} 
\\

\includegraphics[width=2in]{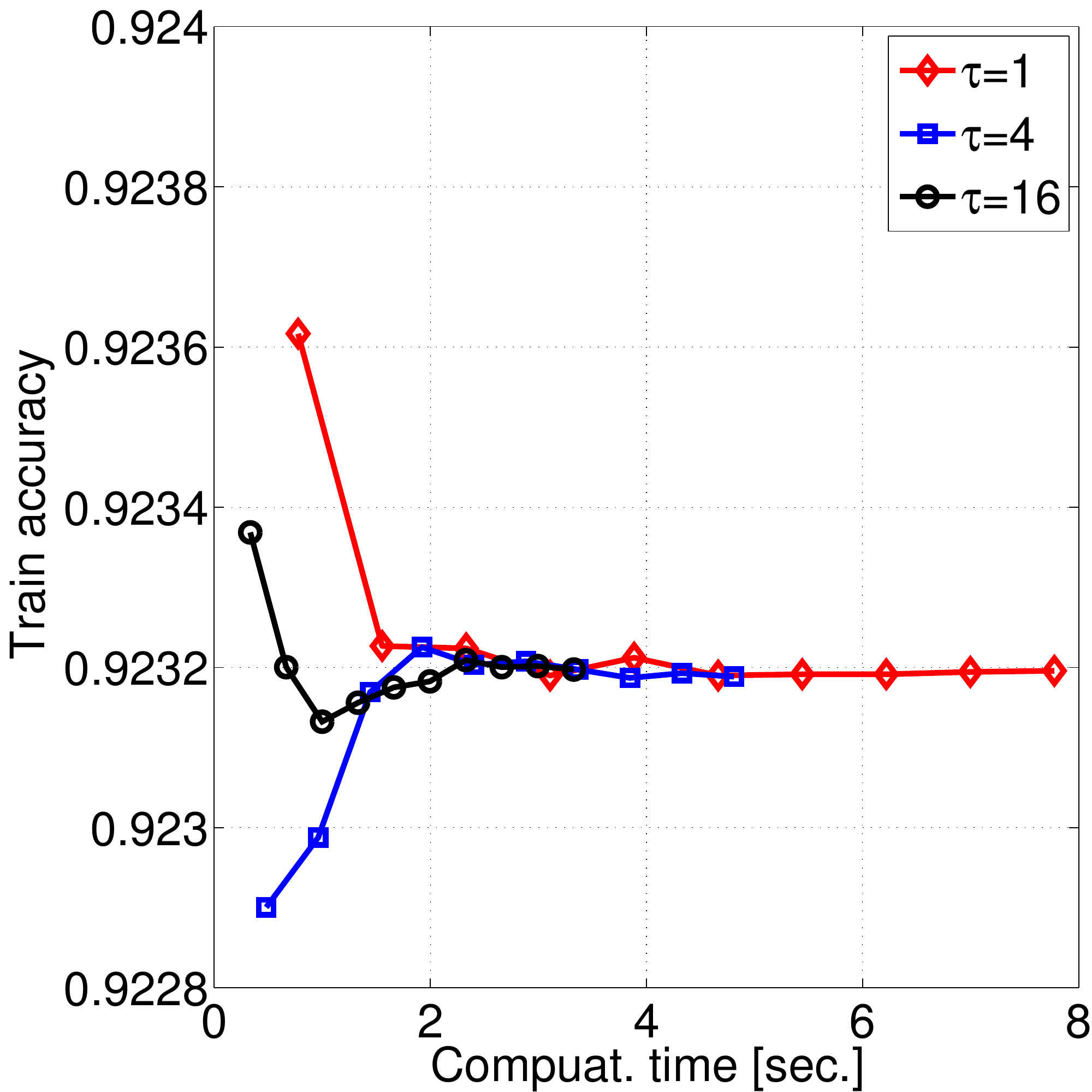}&
\includegraphics[width=2in]{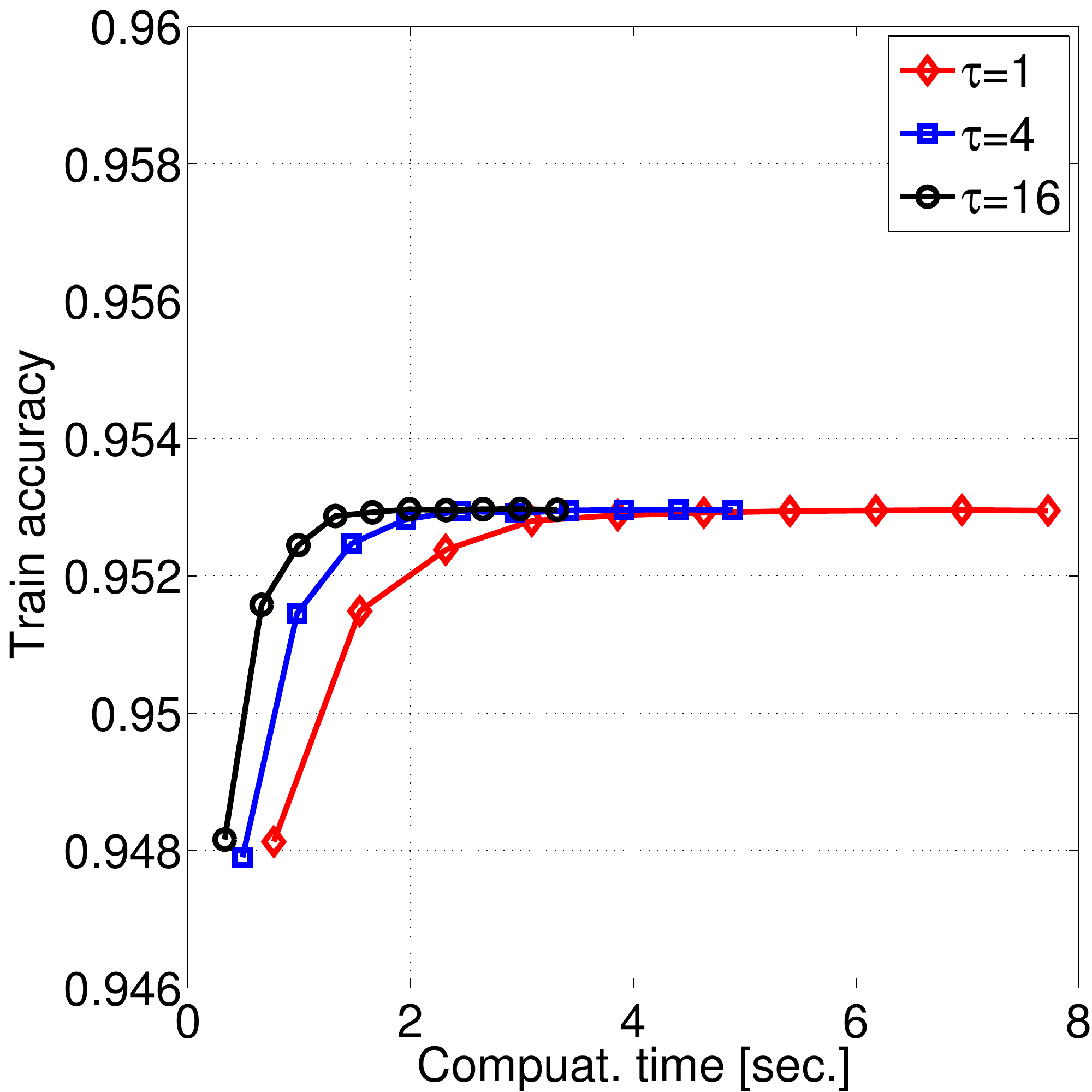}&
\includegraphics[width=2in]{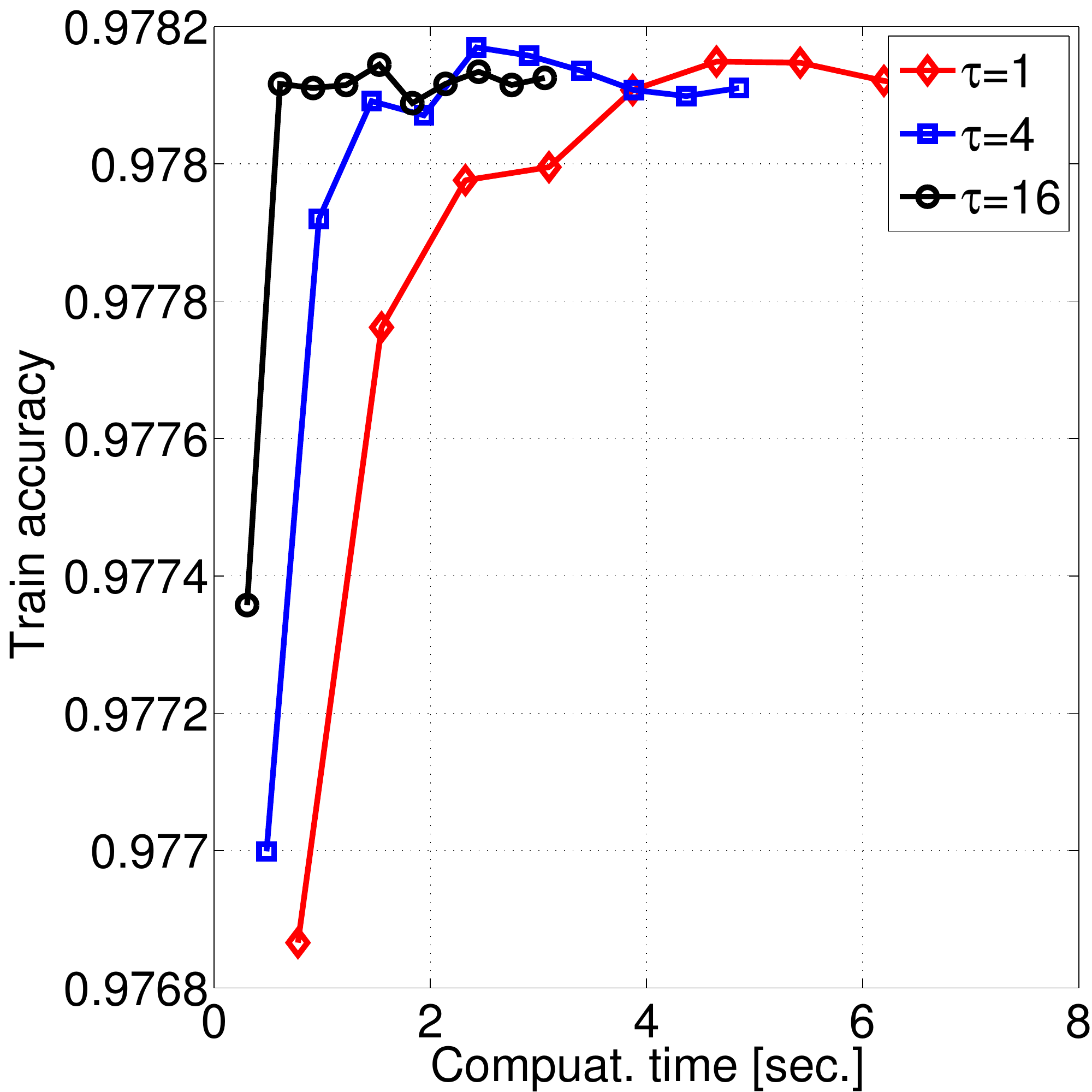}
\\

\includegraphics[width=2in]{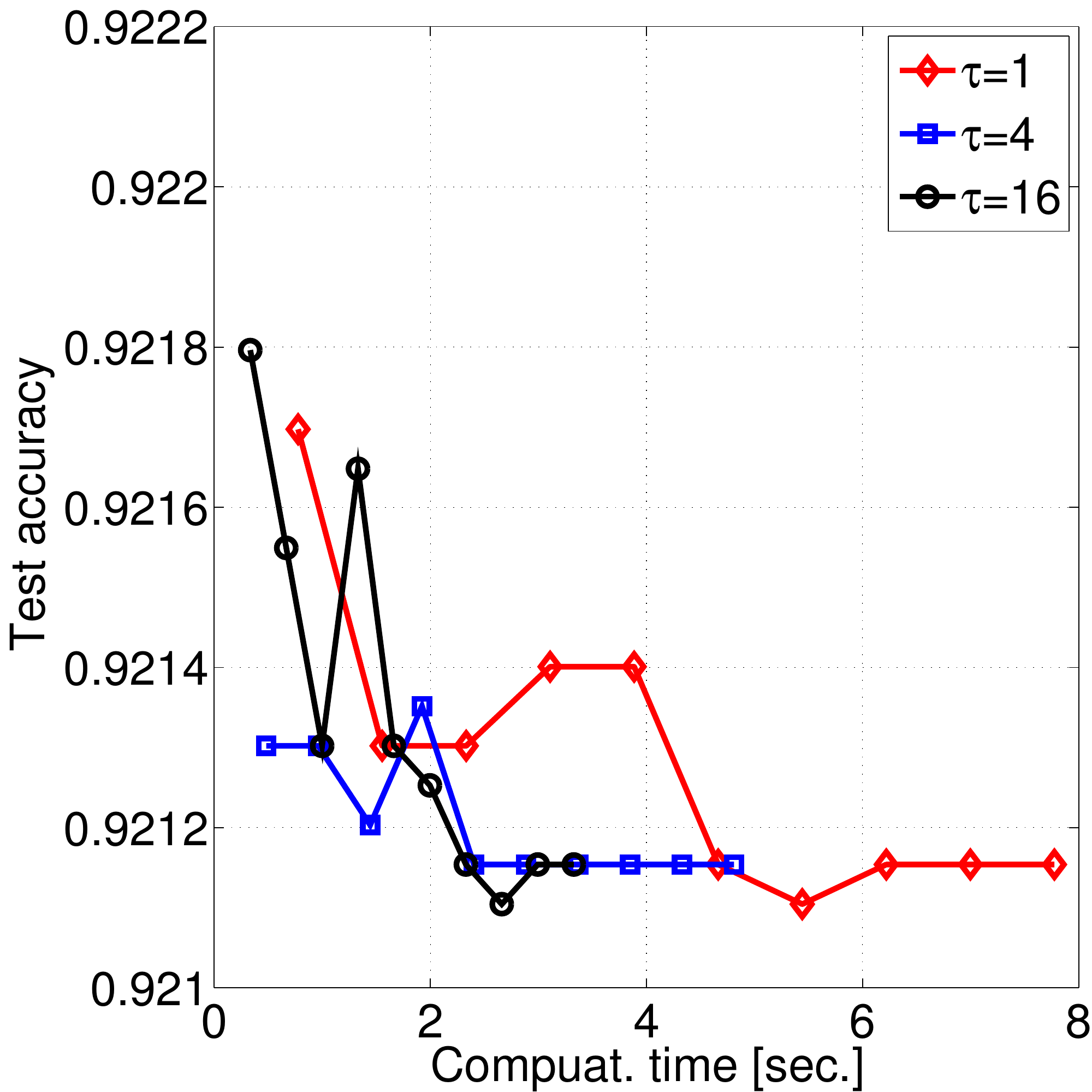}&
\includegraphics[width=2in]{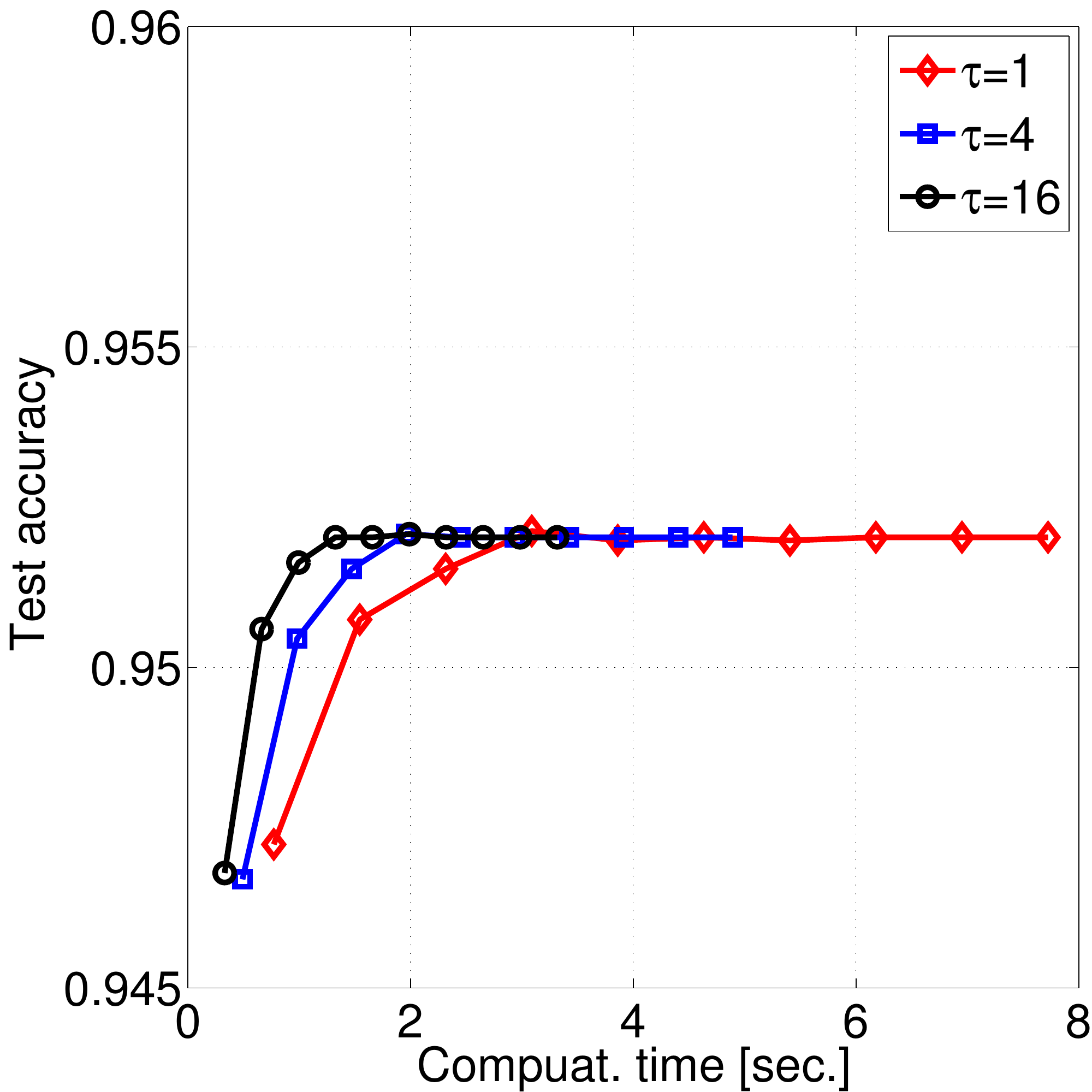}&
\includegraphics[width=2in]{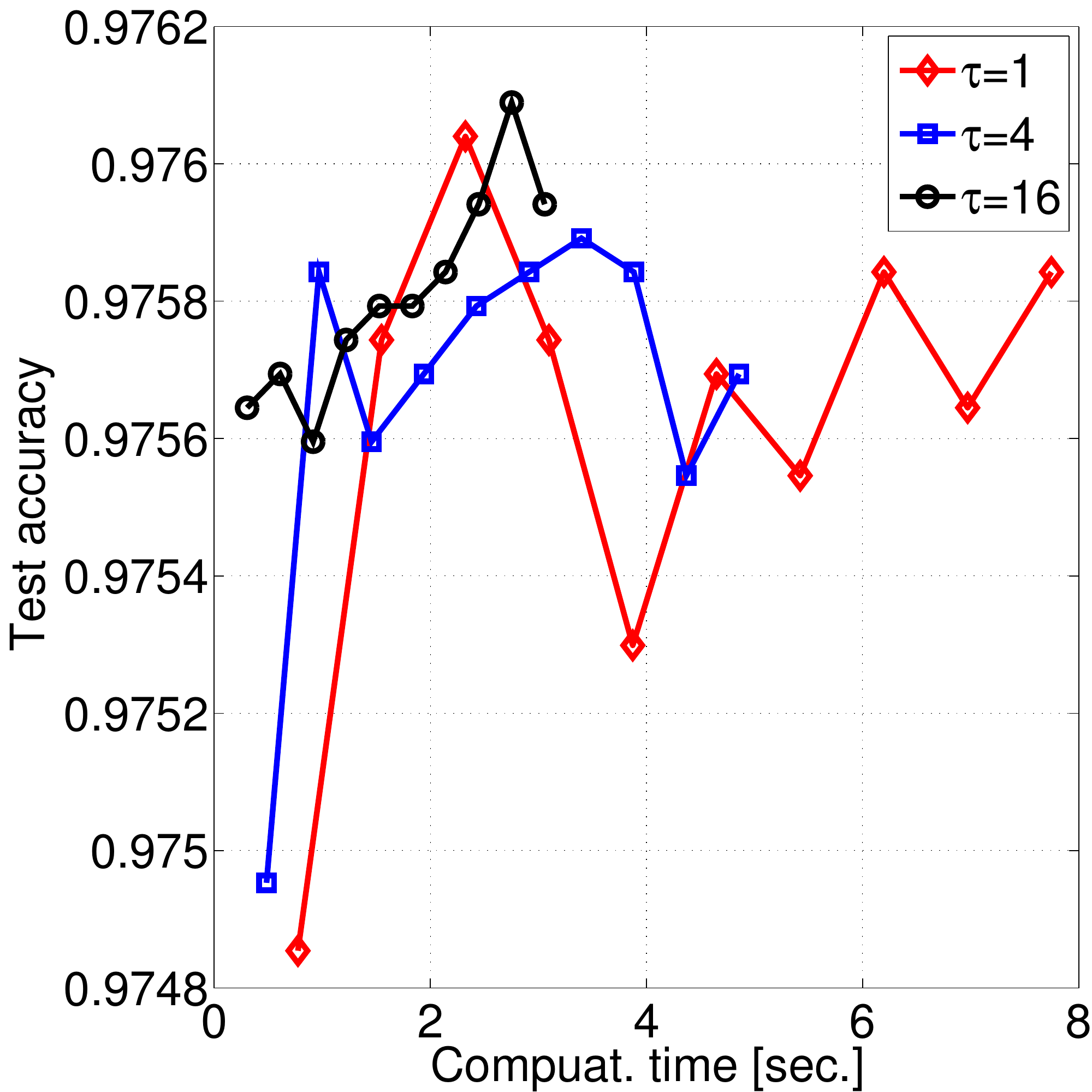}

 \end{tabular}
\caption{The performance of PCDM on the rcv1 dataset (this dataset is \emph{not} good for the method.). }
 \label{fig:rcv1}
\end{figure}
  
  \clearpage

\subsection{$L2$-regularized logistic regression with good data for PCDM}

In our last experiment we solve a problem of the form \eqref{eq:P} with $f$ being a sum of logistic losses and $\cPsi$ being an L2 regularizer,

\[\min_{x \in \R^n} \left\{\sum_{j=1}^d \log(1 + e^{-y_j A_j^T x}) + \lambda \|x\|_2^2\right\},\]
where $(y_j,A_j)\in \{+1,-1\}\times \R^n$, $j=1,2,\dots,d$, are labeled examples.

We have used the the \texttt{KDDB} dataset from the same source as the \texttt{rcv1.binary} dataset considered in the previous experiment. The data contains $n= 29,890,095$ features and is divided into two parts: a training set  with $d=19,264,097$ examples (and $566,345,888$ nonzeros; cca 8.5 GB) and a testing with $d= 748,401$ examples (and $21,965,075$ nonzeros; cca 0.32 GB). 

This training dataset is good for PCDM as each example depends on at most 75 features. That is, $\omega=75$, which is much smaller than $n$. As before, we will use PCDM1 with $\tau$-nice sampling (approximated by $\tau$-independent sampling) for $\tau=1,2,4,8$ and set $\lambda=1$.

Figure~\ref{fig:KDDD} depicts the evolution of the regularized loss $F(x_k)$ throughout the run of the 4 versions of PCDM (starting with $x_0$ for which $F(x_0) = 13,352,855$). Each marker corresponds to approximately $n/3$ coordinate updates ($n$ coordinate updates will be referred to as an ``epoch''). Observe that as more processors are used, it takes less time to achieve any given level of loss; nearly in exact proportion to the increase in  the number of processors.

 \begin{figure}[!h]
 \begin{center}
 \includegraphics[width=3in]{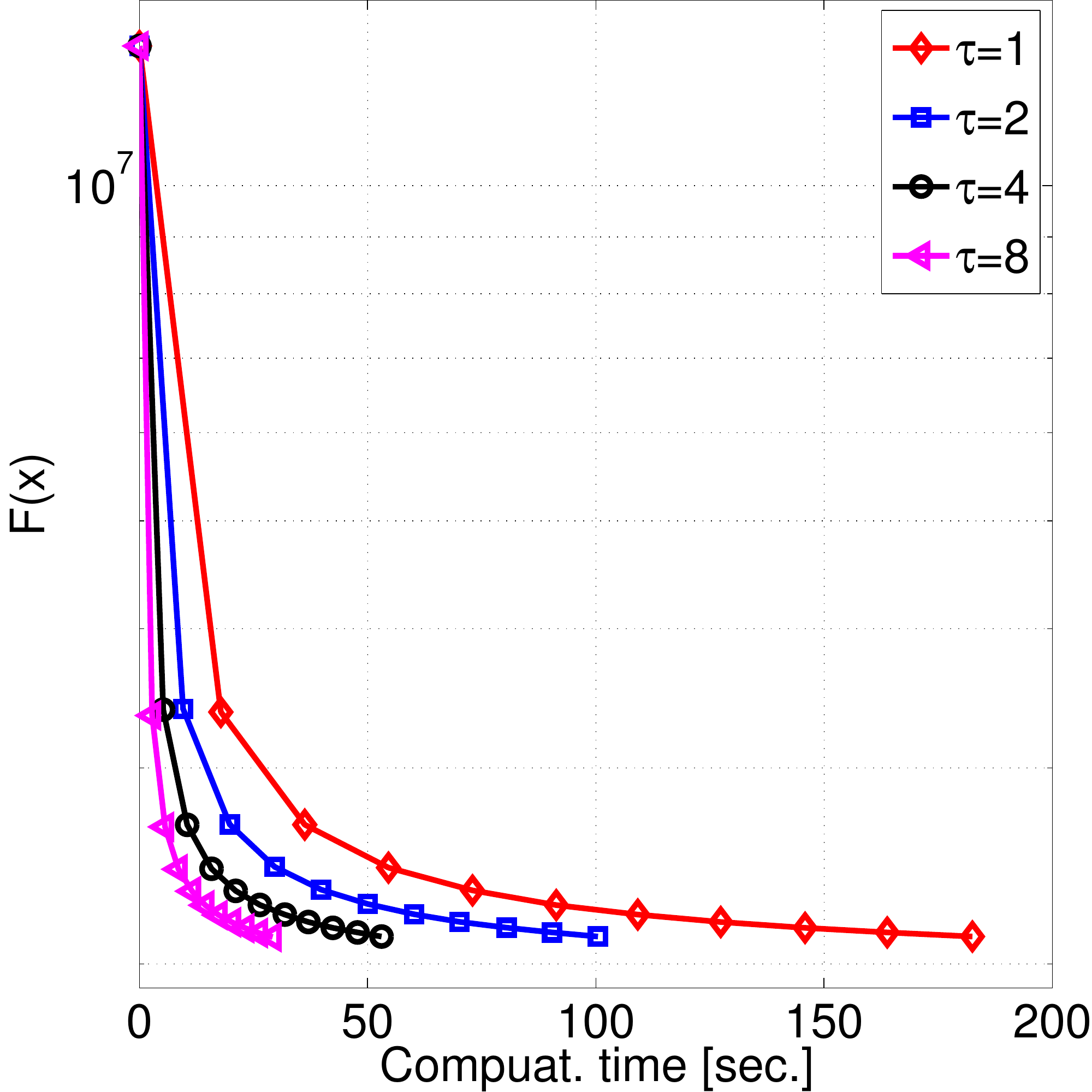}
 \end{center}
  \caption{PCDM accelerates well with more processors on a dataset with small $\omega$.}
 \label{fig:KDDD}
\end{figure}

Table~\ref{tbl:KDDD} offers an alternative view of the same experiment. In the first 4 columns ($F(x_0)/F(x_k)$)  we can see that no matter how many processors are used, the methods produce similar loss values after working through the same number of coordinates. However, since the method utilizing $\tau=8$ processors updates 8 coordinates in parallel, it does the job approximately 8 times faster. Indeed, we can see this speedup in the table.

Let us remark that the training and testing accuracy stopped increasing after having trained the classifier for 1 epoch; they were  $86.07\%$ and $88.77\%$, respectively. This  is in agreement with the common wisdom in machine learning that training beyond a single pass through the data rarely improves testing accuracy (as it may lead to overfitting). This is also the reason behind the success of light-touch methods, such as coordinate descent and stochastic gradient descent, in machine learning applications.

 \begin{table}[!h]
 \centering
  \begin{tabular}{c||c|c|c|c||r|r|r|r}
    & \multicolumn{4}{|c||}{$F(x_0)/F(x_k)$} & \multicolumn{4}{|c}{time}
   \\ \hline
   Epoch& $\tau=1$ & $\tau=2$ & $\tau=4$ & $\tau=8$ & $\tau=1$ & $\tau=2$ & $\tau=4$ & $\tau=8$  
   \\ \hline \hline
   1 &  3.96490 &
3.93909 &
3.94578 &
3.99407 &
17.83 &
9.57 &
5.20 &
2.78
   \\ \hline 
   2 &
   5.73498 &
5.72452 &
5.74053 &
5.74427 &
73.00 &
39.77 &
21.11 &
11.54
   \\ \hline
   3 & 
6.12115 &
6.11850 &
6.12106 &
6.12488 &
127.35 &
70.13 &
37.03 &
20.29
  \end{tabular}
\caption{PCDM accelerates linearly in $\tau$ on a good dataset.}
\label{tbl:KDDD}
 \end{table}


\bibliographystyle{plain} 
\bibliography{pcdm}

\clearpage
\appendix
\section{Notation glossary}

\begin{table}[!h]
\begin{center}
\begin{tabular}{|c|l|c|}
 \hline
 \multicolumn{3}{|c|}{{\bf Optimization problem} (Section~\ref{sec:intro})}\\
 \hline
$N$ & dimension of the optimization variable & \eqref{eq:P}\\
 
 $x, h$ &  vectors in $\R^N$ & \\
$f$ & smooth convex function ($f: \R^N \to \R$) & \eqref{eq:P}\\
$\Omega$ & convex block separable function ($\Omega: \R^N \to \R\cup \{+\infty\}$) & \eqref{eq:P}\\
$F$ & $F=f+\Omega$ (loss / objective function)  & \eqref{eq:P}\\
$\omega$ & degree of partial separability of
 $f$ & \eqref{eq:strucutre_of_f},\eqref{eq:omega}\\

\hline
\multicolumn{3}{|c|}{{\bf Block structure} (Section~\ref{sec:block_structure})}\\
\hline
$n$ & number of blocks & \\
$[n]$ & $[n]=\{1,2,\dots,n\}$ (the set of blocks) & Sec~\ref{sec:block_structure}\\
$N_i$ & dimension of block $i$ ($N_1+\dots+N_n = N$) & Sec~\ref{sec:block_structure}\\
$U_i$ & an $N_i \times N$ column submatrix of the $N \times N$ identity matrix& Prop~\ref{prop:decomposition}\\
$x^{(i)} $ & $x^{(i)}=U_i^T x \in\R^{N_i}$ (block $i$ of vector $x$)&Prop~\ref{prop:decomposition}\\
$\nabla_i f(x)$ & $\nabla_i f(x) = U_i^T \nabla f(x)$ (block gradient of $f$ associated with block $i$)& \eqref{eq:f_iLipschitzder}\\
$L_i$ & block Lipschitz constant of the gradient of $f$ & \eqref{eq:f_iLipschitzder}\\
$L$ & $L = (L_1,\dots,L_n)^T \in \R^n$ (vector of block Lipschitz constants)&\\
$w$ & $w = (w_1,\dots,w_n)^T \in \R^n$ (vector of positive weights) &\\
$\support(h)$ & $\support(h) = \{i \in [n] \;:\; x^{(i)} \neq 0\}$ (set of nonzero blocks of $x$) & \\
$B_i$ & an $N_i\times N_i$ positive definite matrix &\\
$\|\cdot\|_{(i)}$ & $\|x^{(i)}\|_{(i)} = \ve{B_i x^{(i)}}{x^{(i)}}^{1/2}$ (norm associated with block of $i$) & \\
$\|x\|_w$ &  $\|x\|_w=(\sum_{i=1}^n w_i \|x^{(i)}\|^2_{(i)})^{1/2}$ (weighted norm associated with $x$)& \eqref{eq:norms}\\
$\Omega_i$ & $i$-th componet of $\Omega = \Omega_1 + \dots + \Omega_n$ & \eqref
{eq:Psi_block_def}\\
$\mu_{\Omega}(w)$ & strong convexity constant of $\Omega$ with respect to the norm $\|\cdot\|_w$ & \eqref{eq:strong_def}\\
$\mu_f(w)$ & strong convexity constant of $f$ with respect to the norm $\|\cdot\|_w$ & \eqref{eq:strong_def}\\

\hline
\multicolumn{3}{|c|}{{\bf Block samplings} (Section~\ref{SEC:Block_Samplings})}\\
\hline

$S, J$ & subsets of $\{1,2,\dots,n\}$ & \\

$\hat{S}, S_k$ & block samplings (random subsets of $\{1,2,\dots,n\}$) & \\

$x_{[S]}$ & vector in $\R^N$ formed from $x$ by zeroing out blocks $x^{(i)}$ for $i \notin S$ & \eqref{eq:lllop09},\eqref
{eq:lllop09jhkjh}\\

$\tau$ & \# of blocks updated in 1 iteration (when $\Prob(|\hat{S}|=\tau)=1$)& \\
$\Exp[|\hat{S}|]$ & average \# of blocks updated in 1 iteration (when $\mathbf{Var}[|\hat{S}|]>0$) &\\
$p(S)$ & $p(S) = \Prob(\hat{S}=S)$& \eqref{eq:p(S)-general}  \\
$p_i$ & $p_i= \Prob(i \in \hat{S})$ & \eqref{eq:p_i}\\
$p$ & $p = (p_1,\dots,p_n)^T \in \R^n$ & \eqref{eq:p_i}\\

\hline
\multicolumn{3}{|c|}{{\bf Algorithm} (Section~\ref{sub:Algorithms})}\\
\hline

$\beta$ &  stepsize parameter depending on $f$ and $\hat{S}$ (a central object in this paper) &\\
$H_{\beta,w}(x,h)$ & $H_{\beta,w}(x,h) = f(x) + \ve{\nabla f(x)}{h} + \tfrac{\beta}{2}\|h\|_w^2 + \cPsi(x+h)$ &\eqref{eq:H_{beta,w}}\\
$h(x)$ & $h(x) = \arg \min_{h \in \R^\N} H_{\beta,w}(x,h)$ & \eqref{eq:h(x)}\\
$h^{(i)}(x)$ & $h^{(i)}(x) = (h(x))^{(i)} = \arg \min_{t \in \R^{N_i}} \ve{\nabla_i f(x)}{t} + \tfrac{\beta w_i}{2}\|t\|_{(i)}^2 + \Omega_i(x^{(i)}+t)$ & \eqref{eq:h(x)}\\
$x_{k+1}$ & $x_{k+1} =  x_k + \sum_{i\in S_k} U_i h^{(i)}(x_k)$ \quad ($x_k$ is the $k$th iterate of PCDM) & \\

\hline
\end{tabular}
\end{center}
\caption{The main notation used in the paper.}
\label{tbl:notation}
\end{table}

\end{document}